\pdfoutput=1
\documentclass[a4paper]{amsart}

\usepackage{amsthm}
\usepackage{amssymb}
\usepackage{dsfont}
\usepackage{enumerate}
\usepackage{enumitem}
\usepackage{stmaryrd}
\usepackage[mathcal]{euscript}
\usepackage{tikz-cd} 
\usepackage[all]{xy}
\SelectTips{cm}{}
\usepackage{microtype}
\usepackage{mathtools}

\usepackage[top=34mm, bottom=34mm, outer=30mm, inner=30mm, marginparwidth=20mm]{geometry}

\usepackage{hyperref}
\hypersetup{%
  bookmarksnumbered=true,%
  colorlinks=true,%
  linkcolor=[RGB]{0, 0, 153},
  citecolor=[RGB]{0,153,51},      
  urlcolor=[RGB]{153,0,0},
  filecolor=blue,%
  menucolor=blue,%
  bookmarksopen=true,%
  bookmarksdepth=2,%
  pageanchor=true,
  breaklinks=true,
  pdfusetitle=true}

\usepackage[nameinlink,capitalize]{cleveref}
\crefname{equation}{}{}

\usepackage[
sorting=nyt,
backend=bibtex,
style=alphabetic,
maxnames=50
]{biblatex}
\renewbibmacro{in:}{}
\makeatletter
\@namedef{subjclassname@2020}{\textup{2020} Mathematics Subject Classification}
\makeatother

\numberwithin{equation}{section}

\theoremstyle{plain}
\newtheorem*{thm-intro}{Theorem}
\newtheorem*{prop-intro}{Proposition}
\newtheorem*{thm-intro2}{Theorem D}
\newtheorem{theorem}[equation]{Theorem}
\newtheorem{proposition}[equation]{Proposition}
\newtheorem{lemma}[equation]{Lemma} 
\newtheorem{corollary}[equation]{Corollary}
\newtheorem{conjecture}[equation]{Conjecture}

\theoremstyle{definition}
\newtheorem{definition}[equation]{Definition}
\newtheorem{example}[equation]{Example}

\newtheorem*{claim}{Claim}

\theoremstyle{remark}
\newtheorem{remark}[equation]{Remark} 

\newtheorem*{ack}{Acknowledgements}
\newtheorem*{conventions}{Conventions}
\newtheorem*{orga}{Organization of the paper}

\newcommand*{\intref}[2]{\def\tmp{#1}\ifx\tmp\empty\hyperref[#2]{\ref*{#2}}\else\hyperref[#2]{#1~\ref*{#2}}\fi}

\newcommand{\cc}{\operatorname{c}}
\newcommand{\CH}{\operatorname{CH}}
\newcommand{\ch}{\operatorname{ch}}

\newcommand{\codim}{\operatorname{codim}}
\newcommand{\Coh}{\operatorname{Coh}}

\renewcommand{\dim}{\operatorname{dim}}
\newcommand{\End}{\operatorname{End}}
\newcommand{\Ext}{\operatorname{Ext}}

\newcommand{\Hom}{\operatorname{Hom}}
\newcommand{\id}{\operatorname{id}}

\newcommand{\Pic}{\operatorname{Pic}}

\newcommand{\Proj}{\operatorname{Proj}}

\newcommand{\sheafhom}{\mathop{\mathcal{H}\! \mathit{om}}\nolimits}

\newcommand{\Spec}{\operatorname{Spec}}
\newcommand{\Supp}{\operatorname{Supp}}

\newcommand{\mcE}{\mathcal{E}}
\newcommand{\mcF}{\mathcal{F}}
\newcommand{\mcG}{\mathcal{G}} 
\newcommand{\mcH}{\mathcal{H}} 

\newcommand{\mcL}{\mathcal{L}}
\newcommand{\mcO}{\mathcal{O}}

\newcommand{\sfD}{\mathsf D}

\newcommand{\bbA}{\mathbb A}

\newcommand{\bbF}{\mathbb F}

\newcommand{\bbP}{\mathbb P}
\newcommand{\bbQ}{\mathbb Q} 
 
\newcommand{\bbZ}{\mathbb Z}

\iffalse

\newcommand{\Johannes}[1]{}
\newcommand{\Charles}[1]{}

\else

\newcommand{\marg}[1]{\normalsize{{
			\color{red}\footnote{{\color{blue}#1}}}{\marginpar[\vskip
			-.25cm{\color{red}\hfill\thefootnote$\implies$}]{\vskip
				-.2cm{\color{red}$\impliedby$\tiny\thefootnote}}}}}
\newcommand{\Johannes}[1]{\marg{(Johannes) #1}}
\newcommand{\Charles}[1]{\marg{(Charles) #1}}

\fi

\title[On proper splinters in positive characteristic]{On proper splinters in positive characteristic}

\author[Johannes~Krah]{Johannes Krah}

\author[Charles~Vial]{Charles Vial}
\address{Fakult\"at f\"ur Mathematik, Universit\"at Bielefeld, D-33501 Bielefeld, Germany
}
\email{jkrah@math.uni-bielefeld.de,
	vial@math.uni-bielefeld.de}

\thanks{The research of both authors was funded by the Deutsche Forschungsgemeinschaft (DFG, German Research Foundation) -- Project-ID~491392403 -- TRR~358}

\begin{document}

\begin{abstract} 
While the splinter property is a local property for Noetherian schemes in characteristic zero, Bhatt observed that it imposes strong conditions on the global geometry of proper schemes in positive characteristic.
We show that if a proper scheme over a field of positive characteristic is a splinter, then its Nori fundamental group scheme is trivial and its Kodaira dimension is negative. 
In another direction, Bhatt also showed that any splinter in positive characteristic is a derived splinter. We ask whether the splinter property is a derived invariant for projective 
varieties in positive characteristic and give a positive answer for normal Gorenstein projective 
varieties with big anticanonical divisor. 
We also show that global $F$-regularity is a derived invariant for normal Gorenstein projective varieties in positive characteristic.
\end{abstract}


\date{\today}

\maketitle

\section{Introduction}	
	A Noetherian scheme $X$ is a \emph{splinter} if for all finite surjective morphisms $f\colon Z \to X$ the map $\mcO_X \to f_*\mcO_Z$ splits in the category of coherent $\mcO_X$-modules. The direct summand conjecture, now a theorem due to Andr\'e~\cite{andre}, stipulates that any regular Noetherian affine scheme is a splinter.
	In characteristic zero, the splinter property is a local property\,: a Noetherian scheme over $\bbQ$ is a splinter if and only if it is normal. 
	In positive characteristic, the splinter property is no longer a local property in general.
	Bhatt's beautiful \cite[Thm.~1.5]{bhatt_derived_splinters_in_positive_characteristic}, inspired from Hochster and Huneke's \cite[Thm.~1.2]{hochster_huneke_big_CM}, shows that, for a proper scheme over an affine Noetherian scheme
	 of positive characteristic, the positive-degree cohomology of the structure sheaf vanishes up to finite covers. 
	Bhatt draws two consequences for splinters of positive characteristic\,: first that the positive-degree cohomology of semiample invertible sheaves on proper splinters vanishes, and second that splinters and derived splinters coincide. 
	Our first aim is to provide further global constraints on proper splinters in positive characteristic. 
	Our second aim is to study whether the splinter property is a derived invariant for projective varieties in positive characteristic.

	\subsection*{Global constraints on proper splinters in positive characteristic}

	Recall that a smooth projective, separably rationally connected, variety over an  algebraically closed field of positive characteristic has trivial Nori fundamental group scheme~\cite{biswas}, has negative Kodaira dimension~\cite[Ch.~IV, Cor.~1.11 \& Prop.~3.3]{kollar_rational_curves_algebraic_varieties}, and has no nonzero global differential forms~\cite[Ch.~IV, Cor.~3.8]{kollar_rational_curves_algebraic_varieties}.
	Motivated by the intriguing question whether proper splinters over an algebraically closed field of positive characteristic are separably rationally connected, we show\,:

\begin{thm-intro}
	Let $X$ be a connected proper scheme over a field $k$ of positive characteristic. \linebreak
	Assume that $X$ is a splinter. 
	\begin{enumerate}[label={\bf(\Alph*)}]
		\item \label{item:thm:triv_fund_grp} \emph{(\cref{prop:trivial_fundamental_group})}
		If $X$ has a $k$-rational point $x\in X(k)$, then the Nori fundamental group scheme
		 $\pi^N_1(X,x)$ is trivial. In particular, if $k$ is algebraically closed, then any finite torsor over $X$ is trivial.
		\item \label{item:thm:negative_kod_dim} \emph{(\cref{prop:K_equvialenceodaira_dim_f_split_variety})} 
		If $X$ is positive-dimensional, then $X$ has negative Kodaira dimension, i.e., $H^0(X,\mcO_X(nK_X)) = 0$ for all $n>0$, where $K_X$ denotes the canonical divisor of $X$.
		\item \label{item:thm:global-1-forms} \emph{(\cref{thm:global-1-forms})}
		If $X$ is smooth, then $H^0(X,\Omega_X^1) = 0$.
	\end{enumerate}
\end{thm-intro}

	Note from their respective proofs that \labelcref{item:thm:global-1-forms}, and \labelcref{item:thm:negative_kod_dim} in the Gorenstein case, follow from the known fact, recalled in \cref{prop:pic_torsion_free}, that the Picard group of a proper splinter over a field of positive characteristic is torsion-free.
    The proofs of \labelcref{item:thm:triv_fund_grp}, and of \labelcref{item:thm:negative_kod_dim} in the general non-Gorenstein case, rely on a lifting property for splinters along finite torsors established in \cref{prop:torsor},
     which itself relies on the more general key lifting property established in \cref{lem:splinter_ascend_pi_!}\labelcref{item:sameH0}. 
    The new idea, which makes it in particular possible to avoid any Gorenstein assumption, is the use of the exceptional inverse image functor for finite morphisms.
    By using the more general exceptional inverse image functor for proper morphisms, 
    we further observe in~\cref{rmk:crepant} that the derived splinter property for Noetherian schemes lifts along crepant morphisms.
    In \cref{lem:splinter_ascend_pi_!}\labelcref{item:etale} and building on \cref{lem:splinter_ascend_pi_!}\labelcref{item:reducedfibers},
 we establish notably the following lifting property for splinters\,:

	\begin{thm-intro}[\cref{prop:quasitorsor}]
	Let $\pi\colon Y \to X$ be a finite surjective
	morphism of normal Nagata Noetherian schemes.
	Assume that $\pi$ satisfies either of the following conditions\,:
	\begin{enumerate}[label=\bf{(\Alph*)}]   \setcounter{enumi}{3}
		\item \label{thm:intro_d-etale} $\pi$ is quasi-\'etale.
		\item  \label{thm:intro_d-torsor} $\pi$ is a quasi-torsor over a ring, 
		and $\pi^\sharp \colon H^0(X, \mcO_X) \to H^0(Y , \mcO_Y)$ has reduced closed fibers.  
\end{enumerate}
If $X$ is a splinter, then $Y$ is a splinter.
\end{thm-intro}

Theorem~\labelcref{thm:intro_d-etale} extends \cite[Thm.~A]{datta_tucker_on_some_permanence_properties_of_derived_splinters},
where it is shown that any essentially \'etale cover of a Noetherian affine splinter is a splinter.
Besides, Theorems~\labelcref{thm:intro_d-etale,thm:intro_d-torsor}
are stated more generally (by \cref{rmk:splinter-vs-globally+}) for globally $+$-regular pairs as introduced in \cite{bhatt_et_al_globally_+_regular_varieties_and_mmp_for_threefolds_in_mixed_char}, and extend \cite[Prop.~6.20]{bhatt_et_al_globally_+_regular_varieties_and_mmp_for_threefolds_in_mixed_char} which deals with finite quasi-\'etale covers of normal excellent proper schemes over a Noetherian ring. 
Regarding Theorem~\labelcref{item:thm:triv_fund_grp}, 
we also show in \cref{prop:trivial_fundamental_group} that if $X$ is a proper splinter over a separably closed field of positive characteristic, then its \'etale fundamental group is trivial.
 In that direction, we also refer to \cite[Thm.~7.0.3]{cai-lee-ma-schwede-tucker}, where  it is in particular showed that the \'etale fundamental group of the regular locus of a normal projective globally +-regular variety satisfying some additional technical assumptions is finite, but also to the references in the introduction of \emph{loc.\ cit.}\ regarding the \'etale fundamental group of regular loci of near smooth Fano varieties.
\medskip

	Let $k$ be an algebraically closed field of positive characteristic. 
	It is well-known that a proper curve over $k$ is a splinter if and only if it is isomorphic to $\bbP^1_k$.
	For proper surfaces, we have the following results regarding splinters.
	 In~\cref{prop:splinter_surface_rational}, we use \cref{prop:K_equvialenceodaira_dim_f_split_variety} to show that if a proper surface over $k$ is a splinter, then it is rational. 
	We provide in \cref{prop:blow_up_of_points_on_a_line_or_conic} new examples of proper rational surfaces that are splinters, by establishing that the blow-up of $\bbP^2_k$ in any number of points lying on a conic is a splinter. 
	On the other hand, in \cref{prop:example_9_points_not_splinter} and \cref{prop:15_points_on_quartic_curve}, we give examples of proper rational surfaces that are not splinters. For instance, we show that over a finite field the blow-up of $\bbP^2_k$ in 9 points lying on a smooth cubic curve is not a splinter.

	\subsection*{$\mcO$-invariance and $D$-invariance of the splinter property}
	The second aim of this paper is to study whether the splinter property, and the related notion of \emph{global $F$-regularity}, is a derived invariant among projective varieties. 
	We say that two projective varieties $X$ and $Y$ over a field $k$ are $D$-equivalent if there is a $k$-linear equivalence $\sfD^b(X) \cong \sfD^b(Y)$ between their bounded derived categories of coherent sheaves.
	Given that a Gorenstein projective splinter, resp.\ a Gorenstein projective globally $F$-regular variety, in positive characteristic is expected to, resp.\ is known to,  have big anticanonical divisor (see \cref{conj:splinter_big_anticanonical_class} due to \cite{bhatt_et_al_globally_+_regular_varieties_and_mmp_for_threefolds_in_mixed_char}, resp.\ \cref{prop:gFr-big} due to \cite[Cor.~4.5]{schwede_smith_globally_f_regular_and_log_Fano_varieties}), 
	we obtain the following positive answer for varieties\,:

	\begin{thm-intro}
		Let $X$ and $Y$ be normal Gorenstein projective varieties over a field $k$ of positive characteristic.
	 Assume that $X$ and $Y$ are $D$-equivalent. Then\,:
		\begin{enumerate}[label={\bf(\Alph*)}]   \setcounter{enumi}{5}
			\item \label{thm:D-splinter} \emph{(\cref{cor:d_equvialent_splinters_pseudo-rational})} $X$ is a splinter if and only if $Y$ is a splinter, provided $-K_X$ is big.
			\item \label{thm:D-gFr} \emph{(\cref{cor:d-equiv-gFr})} $X$ is globally $F$-regular if and only if $Y$ is globally $F$-regular.
		\end{enumerate} 
	\end{thm-intro}
	
	For that purpose, we introduce in \cref{def:O_equiv} the notion of $\mcO$-equivalence for separated schemes of finite type over a Noetherian base. 
	By \cref{prop:K-O-eq}, this notion coincides with the classical notion of $K$-equivalence in the case of normal Gorenstein varieties over a field.
	As before, but now for proper morphisms that are not necessarily finite, the new idea
 is to use the exceptional inverse image functor of Grothendieck, which allows for more flexibility.
	In \cref{prop:d_equiv_implies_o_equiv}, we observe that Kawamata's \cite[Thm.~1.4(2)]{kawamata_d_equivalence_and_k_equivalence}, stating that two $D$-equivalent smooth projective complex varieties $X$ and $Y$ with $K_X$ or $-K_X$ big are $K$-equivalent, extends to the case of normal Gorenstein projective varieties over an arbitrary field. 
	If in addition either $X$ or $Y$ has at worst canonical singularities, then $X$ and $Y$ are in fact, by
	\cref{prop:K_equvialence-eqivalence_implies_small_birational_map}\labelcref{item:one_canonical_proper} and \cref{prop:K-O-eq}, strongly $\mcO$-equivalent in the sense of \cref{def:O_equiv}.
	As is well-known, a splinter (resp.\ a globally $F$-regular variety) has pseudo-rational singularities, see \cref{prop:bhatt_normal_CM}, and a pseudo-rational normal Gorenstein variety has canonical singularities, see \cref{prop:canonical_pseudorational_Gorenstein}.
	As such, Theorem~\ref{thm:D-splinter} and Theorem~\ref{thm:D-gFr} follow from the following\,:
	
	\begin{thm-intro}
		Let $X$ and $Y$ be 	proper varieties over a field $k$ of positive characteristic. 
		Assume that $X$ and~$Y$ are strongly $\mcO$-equivalent. Then\,:
	
		\begin{enumerate}[label={\bf(\Alph*)}]   \setcounter{enumi}{7}
			\item \label{thm:O-splinter} \emph{(\cref{cor:splinters-O-eq})} 
			$X$ is a splinter if and only if $Y$ is a splinter.
		\end{enumerate}
		Assume in addition that $X$ and $Y$ are normal projective over $k$.
		Then\,:
		\begin{enumerate}[label={\bf(\Alph*)}]   \setcounter{enumi}{8}
			\item \label{thm:O-gFr} \emph{(\cref{prop:globF_stable_O_equiv})} 
			$X$ is globally $F$-regular if and only if $Y$ is globally $F$-regular.
		\end{enumerate}
	\end{thm-intro}

We observe that both 
Theorems~\labelcref{thm:O-splinter,thm:O-gFr}
 hold without any ($\bbQ$-)Gorenstein assumption and in fact without any restrictions on the singularities of $X$ nor $Y$. Again, this is made possible by 
the systematic use of the exceptional inverse image functor. 
For the sake of illustration, we show in \cref{cor:splinter-gFr-K-eq}
  that the splinter property is invariant under $K$-equivalence of \emph{terminal} varieties, by utilizing the fact stated in \cref{prop:K_equvialence-eqivalence_implies_small_birational_map}\labelcref{item:temrinal}
   that a $K$-equivalence between two such varieties induces a small birational map. 
 Note however that, in positive characteristic, splinters and globally $F$-regular varieties may have worse singularities\,; in fact, both the splinter property and global $F$-regularity are expected to locally be the analogues of klt singularities in the complex setting.

\begin{orga}
	In \cref{S:reflexive_sheaes_and_dualizing_complexes,S:prelim_splinters_globF,S:invariance_small_birat_maps},
	we mostly fix notation and collect basic and known facts about
	reflexive sheaves and Weil divisors, inverse image functors, traces, and splinters and globally $F$-regular varieties. Our first new contributions are contained in \cref{S:LD-splinter,S:LD-globF} dealing respectively with splinters and  globally $F$-regular varieties.
	Notably, the use of the exceptional inverse image functor makes its first appearance in \cref{SS:lift-splinter}, where we prove the key \cref{lem:splinter_ascend_pi_!} stating in particular that the splinter property lifts along finite surjective morphisms $\pi\colon Y \to X$ of Noetherian schemes such that $\pi^!\mcO_X \cong \mcO_Y$ as coherent sheaves on $Y$ with $H^0(Y, \mcO_Y)$ a field or $H^0(X,\mcO_X) = H^0(Y,\mcO_Y)$.
	
	In \cref{S:finite_torsors_over_splitners-generalcase}, 
	we extend \cref{lem:splinter_ascend_pi_!} to show in \cref{prop:torsor} that, under certain conditions, the splinter property lifts along finite torsors for Noetherian schemes, and establish Theorems~\labelcref{thm:intro_d-etale} and~\labelcref{thm:intro_d-torsor}.
	
	In \cref{S:finite_torsors_over_splitners,S:splitners_neg_kod_dim,S:vanishing_global_1_forms,S:splinters_surfaces},
we explore global constraints on proper splinters in positive characteristic and establish Theorems~\labelcref{item:thm:triv_fund_grp,item:thm:negative_kod_dim,item:thm:global-1-forms}. Their proof relies ultimately on \cref{lem:splinter_ascend_pi_!}\labelcref{item:sameH0}.
	In \cref{S:splinters_surfaces}, we show that proper splinter surfaces in positive characteristic are rational, and give examples of rational surfaces that are splinters as well as examples of rational surfaces that are not splinters.
	
	In \cref{S:O-eq}, which can be read mostly independently of the rest of the paper,
	we introduce the notion of $\mcO$-equivalence for separated schemes of finite type over a Noetherian base scheme and compare it, in the case of normal Gorenstein varieties, to the usual notions of $K$-equivalence and $D$-equivalence. 
	We then establish Theorems~\labelcref{thm:D-splinter,thm:D-gFr,thm:O-splinter,thm:O-gFr}.	
\end{orga}

	\begin{conventions}
		A \emph{variety} is an integral separated scheme of finite type over a field. 
		For a scheme~$X$ over the finite field $\bbF_p$ with $p$ elements, the \emph{Frobenius} is denoted by $F \colon X\to X$\,; it is the identity on the underlying topological space and sends each local section of $\mcO_X$ to its $p$-th power.
		A scheme $X$ over $\bbF_p$ is said to be \emph{$F$-finite} if the Frobenius map $F\colon X \to X$ is finite. 
		Note that a scheme of finite type over a field $k$ is $F$-finite if and only if the  field $k$ is $F$-finite.
		For a Noetherian scheme $X$, we denote by $X_{\mathrm{reg}}$ its regular locus and by $X_{\mathrm{sing}} \coloneqq X\setminus X_{\mathrm{reg}}$ its singular locus. If $X$ is  reduced, then $X_{\mathrm{reg}}$ contains the generic points of the irreducible components of $X$ and hence is dense in~$X$.  If $X$ is J-1, e.g.\ if $X$ is excellent,  then $X_{\mathrm{reg}} \subseteq X$ is an open embedding \cite[\href{https://stacks.math.columbia.edu/tag/07P6}{Tag 07P6}]{stacks-project}. 
		We denote by $\nu\colon X^\nu \to X$ the normalization of $X$\,; if $X$ is Nagata, e.g.\ if $X$ is excellent, then $\nu$ is a finite morphism \cite[\href{https://stacks.math.columbia.edu/tag/035S}{Tag 035S}]{stacks-project}. 
		Both the Nagata and the excellence properties are stable under locally of finite type extensions\,; see \cite[\href{https://stacks.math.columbia.edu/tag/0359}{Tag 0359}]{stacks-project} and \cite[\href{https://stacks.math.columbia.edu/tag/07QS}{Tag 07QS}]{stacks-project}.
		For a Noetherian scheme $X$, we denote by $\sfD^b(X) \coloneqq \sfD^b(\Coh(X))$ the bounded derived category of coherent sheaves on $X$ and by $\sfD_{\mathrm{Coh}}(\mcO_X)$ (resp.\ $\sfD_{\mathrm{Coh}}^b(\mcO_X)$, resp.\ $\sfD_{\mathrm{Coh}}^+(\mcO_X)$)
		the derived category of unbounded (resp.\ bounded, resp.\ bounded below)
		 complexes of $\mcO_X$-modules with coherent cohomology sheaves. 
		The functor $\sfD^b(X) \to \sfD_{\mathrm{Coh}}(\mcO_X)$ is fully faithful with essential image $\sfD^b_{\mathrm{Coh}}(\mcO_X)$ \cite[\href{https://stacks.math.columbia.edu/tag/08E0}{Tag 08E0}]{stacks-project}.
		If $X$ is not necessarily Noetherian, we use the notation $\sfD_{\mathrm{QCoh}}(\mcO_X)$
		 (resp.\ $\sfD_{\mathrm{QCoh}}^+(\mcO_X)$)
	for	the derived category of unbounded (resp.\  bounded below) complexes of $\mcO_X$-modules with quasi-coherent cohomology sheaves.
		For a  proper morphism $f\colon Z \to X$ of Noetherian schemes, we denote by $f^\sharp \colon \mcO_X \to f_* \mcO_Z$, resp. $f^\sharp \colon \mcO_X \to \mathrm R f_* \mcO_Z$, the canonical morphism in the category, resp.\ bounded derived category, of coherent sheaves on~$X$.
		Given an effective Weil divisor~$D$ on a normal Noetherian scheme~$X$,  we write $\sigma_D \colon  \mcO_X \to \mcO_X(D)$ for the morphism determined by $D$.
	\end{conventions}

	\begin{ack}
		We thank Javier Carvajal-Rojas, Karl Schwede, and the anonymous referee, for useful comments. 
		\cref{lem:splinter_ascend_pi_!}\labelcref{item:reducedfibers} started taking shape at the Fall School ``Methods in Mixed Characteristic Geometry''  in Mainz in October 2024 and we are grateful to the organizers and participants for the pleasant and stimulating atmosphere. 
	\end{ack}

	\section{Reflexive sheaves and exceptional inverse image functors}
	\label{S:reflexive_sheaes_and_dualizing_complexes}

	\subsection{Reflexive sheaves and Weil divisors}\label{S:reflexive_sheaves_and_weil_divisors}
	Let $X$ be an integral Noetherian scheme. 
	Recall that a coherent sheaf $\mcF$ on $X$ is called \emph{reflexive} if the canonical map $\mcF \to \mcF^{\vee \vee}$ is an isomorphism, where by definition $\mcF^\vee \coloneqq \sheafhom_{\mcO_X}(\mcF, \mcO_X)$. 

	\begin{enumerate}
		\item For any coherent sheaf $\mcF$ and any reflexive coherent sheaf $\mcG$ the sheaf $\sheafhom_{\mcO_X}(\mcF, \mcG)$ is reflexive \cite[\href{https://stacks.math.columbia.edu/tag/0AY4}{Tag 0AY4}]{stacks-project}.
		\item If $X$ is normal, then a coherent sheaf is reflexive if and only if it is $S_2$ \cite[Thm.~1.9]{hartshorne_gen_divisors_on_gorenstein_schemes}.
		\item \label{item:restriction_to_big_open_equiv_of_cat} If $X$ is normal and $i \colon U \hookrightarrow X$ is an open immersion such that $\codim_X (X \setminus U)\geq 2$, then $i_* i^* \mcF \cong \mcF$ for any reflexive coherent sheaf $\mcF$. Furthermore, the restriction $i^*$ induces an equivalence of categories from reflexive coherent sheaves on $X$ to reflexive coherent sheaves on $U$ \cite[Thm.~1.12]{hartshorne_gen_divisors_on_gorenstein_schemes}.
		\item \label{item:restriction_to_big_open_Cartier}
		If $X$ is normal and $\mcF$ is a reflexive coherent sheaf, then there exists an open immersion $i \colon U \hookrightarrow X$ such that $\codim_X (X \setminus U)\geq 2$ and such that $i^* \mcF$ is finite locally free \cite[\href{https://stacks.math.columbia.edu/tag/0AY6}{Tag 0AY6}]{stacks-project}.
	\end{enumerate}
	To any Weil divisor $D$ on an integral normal Noetherian scheme $X$
	 with function field $K(X)$, one can associate a coherent sheaf $\mcO_X(D) \subseteq K(X)$ whose sections on  open subsets $V \subseteq X$ are given by
	$$ \Gamma(V, \mcO_X(D))\coloneqq \{f \in K(X)^\times \mid \mathrm{div}(f)\vert_V +D\vert_V \geq 0\} \cup \{0\}.$$
	The sheaf $\mcO_X(D)$ is reflexive of rank~$1$ and it is invertible if and only if $D$ is Cartier.
	Since any reflexive rank $1$ sheaf is isomorphic to a subsheaf of the locally constant sheaf $K(X)$, we have
	\cite[\href{https://stacks.math.columbia.edu/tag/0EBM}{Tag 0EBM}]{stacks-project} 
	 a 1-1 correspondence
	$$\{\mbox{Weil divisors on $X$ up to linear equivalence} \} \leftrightarrow \{\mbox{reflexive sheaves of rank 1 on } X\}/\cong .$$
Moreover thanks to \labelcref{item:restriction_to_big_open_equiv_of_cat,item:restriction_to_big_open_Cartier}
	this bijection turns out to be a group homomorphism, see \cite[Prop.~2.8(d)]{hartshorne_gen_divisors_on_gorenstein_schemes}, provided one takes the double dual of the usual tensor product, i.e.,
	$$ \mcO_X(D+D')\cong (\mcO_X(D) \otimes \mcO_X(D'))^{\vee \vee}.$$
	A Weil divisor $D$ on a normal Noetherian scheme is effective if and only if $\mcO_X \subseteq \mcO_X(D) \subseteq K(X)$. 
	Thus a reflexive rank $1$ sheaf $\mcF$ corresponds to an effective Weil divisor $D$ if and only if there is an injective morphism $\mcO_X \hookrightarrow \mcF$. 
	For later use recall the following criterion.
	\begin{lemma}
		\label{lem:torsion_weil_divisor} 
		Let $X$ be an integral normal Noetherian scheme and assume that $H^0(X, \mcO_X)$ is a field. 
		Then a Weil divisor $D$ on $X$ is linearly equivalent to zero if and only if both $\mcO_X(D)$ and $\mcO_X(-D)$ admit  nonzero global sections.
	\end{lemma}
	\begin{proof}
		The only if part is obvious. Assume $s\colon \mcO_X \to \mcO_X(D)$ and $t\colon \mcO_X \to  \mcO_X(-D)$ are nontrivial sections.
		We have isomorphisms of sheaves $\sheafhom_{\mcO_X}(\mcO_X, \mcO_X(-D))\cong \sheafhom_{\mcO_X}(\mcO_X(D), \mcO_X)$ and $\sheafhom_{\mcO_X}(\mcO_X(D), \mcO_X(D))\cong \sheafhom_{\mcO_X}(\mcO_X, \mcO_X)$. 
		Therefore we can interpret $t$ as a global section of $\sheafhom_{\mcO_X}(\mcO_X(D), \mcO_X)$. The compositions $s\circ t$ and $t\circ s$ are both nonzero and give elements in $H^0(X, \mcO_X)$, thus they are isomorphisms. Hence $\mcO_X \cong \mcO_X(D)$, which shows that $D$ is trivial.
	\end{proof}

	Given two Weil divisors $D$ and $E$ on a normal Noetherian scheme $X$, we write $D\sim E$ if $D$ and~$E$ are linearly equivalent. Likewise, for two $\bbQ$-Weil divisors $D$ and $E$, we write  $D\sim_\bbQ E$ if $D$ and $E$ are $\bbQ$-linearly equivalent.
	
	Let $f\colon Y \to X$ be a finite surjective morphism of integral normal
	Noetherian schemes, and let $D$ be a $\bbQ$-Weil divisor on $X$.
	 Assume that $X$ is universally catenary, e.g.\ Cohen--Macaulay, so that if $U\subseteq X$ is an open subset with $\codim_X(X\setminus U) \geq 2$, then $\codim_Y(Y\setminus f^{-1}(U)) \geq 2$. By \labelcref{item:restriction_to_big_open_Cartier}, choose such a $U$ such that $D$ restricts to a $\bbQ$-Cartier divisor.
	The pullback $f^*D$ can be defined by restricting $D$ to $U$, pulling back to $f^{-1}(U)$, and then extending the pullback uniquely to a $\bbQ$-Weil divisor on $Y$\,;
compare with \cite[Prop.~2.18]{hartshorne_gen_divisors_on_gorenstein_schemes} and with the explicit construction of \cite[\S 2.2]{schwede_tucker_on_the_behaviour_of_test_ideals_under_finite_morphisms}.

	\subsection{Exceptional inverse image functor and dualizing complexes}
	\label{SS:uppershriek}
	Let $f\colon Y \to X$ be a morphism of schemes of finite type and separated over a Noetherian scheme~$S$. We consider the \emph{exceptional inverse image functor} $f^!\colon \sfD^+_{\mathrm{QCoh}}(\mcO_X) \to \sfD^+_{\mathrm{QCoh}}(\mcO_Y)$ as defined in 
	\cite[\href{https://stacks.math.columbia.edu/tag/0A9Y}{Tag 0A9Y}]{stacks-project}.
The functor $f^!$ sends $\sfD^+_{\mathrm{Coh}}(\mcO_X) $ to $ \sfD^+_{\mathrm{Coh}}(\mcO_Y)$
	\cite[\href{https://stacks.math.columbia.edu/tag/0ATZ}{Tag 0ATZ}]{stacks-project}.
	 If $f \colon Y \to X$ is proper,
	 then $f^!$ is right adjoint to $\mathrm{R}f_* \colon \sfD^+_{\mathrm{Coh}}(\mcO_Y) \to \sfD^+_{\mathrm{Coh}}(\mcO_X)$ \cite[\href{https://stacks.math.columbia.edu/tag/0AU3}{Tag 0AU3}]{stacks-project}.
	 
	 Assume now that $X$ admits a \emph{dualizing complex} $\omega_X^\bullet$ as defined in~\cite[\href{https://stacks.math.columbia.edu/tag/0A85}{Tag 0A85}]{stacks-project}. 
	 Then $\omega_Y^\bullet \coloneqq f^! \omega_X^\bullet$
	 defines a dualizing complex on $Y$ \cite[\href{https://stacks.math.columbia.edu/tag/0AU3}{Tag 0AU3}]{stacks-project}.
	If $X$ is equidimensional, $\omega_X \coloneqq \mcH^{-\dim X}(\omega_X^\bullet)$
	 is an $S_2$ sheaf, called the \emph{dualizing sheaf} \cite[\href{https://stacks.math.columbia.edu/tag/0AWH}{Tag 0AWH}]{stacks-project}.
	Moreover, $X$ is Cohen--Macaulay if and only if $\omega_X^\bullet = \omega_X[\dim X]$ \cite[\href{https://stacks.math.columbia.edu/tag/0AWQ}{Tag 0AWQ}]{stacks-project}
	and $X$ is Gorenstein if and only if $\omega_X^\bullet$ is an invertible object \cite[\href{https://stacks.math.columbia.edu/tag/0AWV}{Tag 0AWV}]{stacks-project}. The latter condition is further equivalent to 
	$\omega_X^\bullet = \omega_X[\dim X]$ with $\omega_X$ an invertible sheaf
	\cite[\href{https://stacks.math.columbia.edu/tag/0FPG}{Tag 0FPG}]{stacks-project}.
	If $X$ is normal and equidimensional, then $\omega_X$ is a reflexive sheaf of rank~$1$ and 
	there thus exists  a unique  (up to linear equivalence) Weil divisor $K_X$, called the \emph{canonical divisor}, such that $\omega_X \cong \mcO_X(K_X)$.

	Let now $h \colon X \to \Spec k$ be a  scheme of finite type and separated over a field~$k$. 
	Then  $\omega_X^\bullet \coloneqq h^! \mcO_{\Spec k} \in \sfD^+_{\mathrm{Coh}}(\mcO_X)$
	is a dualizing complex on $X$ \cite[\href{https://stacks.math.columbia.edu/tag/0AWJ}{Tag 0AWJ}]{stacks-project}.	
	If $X$ is smooth over $k$, then the dualizing sheaf $\omega_X$ coincides with the \emph{canonical sheaf} $\bigwedge^{\dim X} \Omega_{X/k}^1$ \cite[\href{https://stacks.math.columbia.edu/tag/0E9Z}{Tag 0E9Z}]{stacks-project}.
	If $h\colon X \to \Spec k$ is proper and $X$ is equidimensional,
	then $\mathrm{R}h_*$ is left adjoint to $h^!$ and for
	every $K \in \sfD^b(X)$, there is a functorial isomorphism
	$\Ext^i_X(K, \omega^\bullet_X) = \Hom_k (H^i(X, K), k)$ compatible with shifts and exact triangles, see, e.g., \cite[\href{https://stacks.math.columbia.edu/tag/0FVU}{Tag 0FVU}]{stacks-project}. 
	By Yoneda, the object $\omega_X^\bullet$ is unique up to unique isomorphism among all objects satisfying this universal property.
	
	In the more general situation where $f \colon Y \to X$ is a morphism of quasi-compact quasi-separated schemes, $\mathrm{R}f_* \colon \sfD_{\mathrm{QCoh}} (\mcO_Y) \to \sfD_{\mathrm{QCoh}} (\mcO_X)$ still admits a right adjoint $f^\times$, see, e.g., \cite[Thm.~25.17]{goertz_wedhorn_ii}. 
The functor $ f^\times$ sends 
	$\sfD^+_{\mathrm{QCoh}}(\mcO_X)$ to $\sfD^+_{\mathrm{QCoh}}(\mcO_Y)$ and,
	if $X$ is Noetherian and $f$ is proper,  agrees with $f^!$\,; see \cite[Lem.~25.23 \& Thm.~25.61]{goertz_wedhorn_ii}.

	The following general \cref{lem:k_and_o_equvalence} is a consequence of the properties of the exceptional inverse image functor.
	In particular, choosing $q\colon Z \to Y$ in the statement of \cref{lem:k_and_o_equvalence} to be an isomorphism, it shows that if $p \colon Z \to X$ is a separated morphism of finite type, then
	 $p^! \mcO_X \cong \mcO_Z$ if and only if $ \mathrm{L} p^*\omega_X^\bullet \cong \omega_Z^\bullet$.

	\begin{lemma}\label{lem:k_and_o_equvalence}
		Let $X$ and $Y$ be schemes of finite type and separated over a Noetherian scheme $S$ such that $S$ admits a dualizing complex $\omega_S^\bullet$.
		Denote by $h_X \colon X \to S$ and $h_Y \colon Y \to S$ the structure morphisms.
		Let $Z$ be a scheme of finite type and separated over $S$ with $S$-morphisms $p\colon Z \to X$ and $q\colon Z \to Y$.
		Then in $\sfD_{\mathrm{Coh}}(\mcO_Z)$
		we have
		$$\mathrm{L}p^* \omega_X^\bullet \cong \mathrm{L}q^*\omega_Y^\bullet \iff p^! \mcO_X \cong q^! \mcO_Y,$$
		where $\omega_X^\bullet = h_X^! \omega_S^\bullet$ and $\omega_Y^\bullet = h_Y^! \omega_S^\bullet$.
		In particular, if $X$ and $Y$ are equidimensional and Gorenstein, then 
		$$p^* \omega_X [\dim X]\cong q^*\omega_Y[\dim Y] \iff p^! \mcO_X \cong q^! \mcO_Y,$$
		where $\omega_X = \mcH^{-\dim X} (\omega_X^\bullet)$ and $\omega_Y = \mcH^{-\dim Y} (\omega_Y^\bullet)$.
	\end{lemma}
	\begin{proof}
		Let $\omega_Z^\bullet  \coloneqq p^! \omega_X^\bullet = q^!\omega_Y^\bullet$.
		 Recall, e.g., from \cite[\href{https://stacks.math.columbia.edu/tag/0AU3}{Tag 0AU3}]{stacks-project}, that the functor \linebreak $\mathrm{R}\sheafhom_{\mcO_Z}(-, \omega_Z^\bullet)$ defines an involution of $\sfD_{\mathrm{Coh}}(\mcO_Z)$
		 and that we have the duality formula
		\begin{equation*}\label{eq:dualizing_upper_shriek}
		\mathrm{R}\sheafhom_{\mcO_Z}(p^! M, \omega_Z^\bullet) = \mathrm{L}p^*\mathrm{R}\sheafhom_{\mcO_X}(M, \omega_X^\bullet)
		\end{equation*}
	    which holds naturally in $M \in \sfD^+_{\mathrm{Coh}}(\mcO_X)$ and similarly for $q \colon Z \to Y$.
		Setting $M=\mcO_X$ or $M=\mcO_Y$ yields
		\begin{equation*}
		\mathrm{R}\sheafhom_{\mcO_Z}(p^! \mcO_X, \omega_Z^\bullet) = \mathrm{L}p^* \omega_X^\bullet 
		\quad \text{and} \quad
		\mathrm{R}\sheafhom_{\mcO_Z}(q^! \mcO_Y, \omega_Z^\bullet) = \mathrm{L}q^* \omega_Y^\bullet.
		\end{equation*}
		Therefore, $p^! \mcO_X \cong q^! \mcO_Y$ if and only if $\mathrm{L} p^* \omega_X ^\bullet \cong \mathrm{L}  q^* \omega_Y^\bullet$.
	\end{proof}

	\begin{remark}\label{rmk:lift_Gorenstein}
		Let $\pi \colon Y \to X$ be a separated morphism of finite type of Noetherian schemes. If $X$ admits a dualizing complex $\omega_X^\bullet$ and $\pi^! \mcO_X \cong \mcO_Y$, then \cref{lem:k_and_o_equvalence} implies that $\mathrm{L}\pi^*\omega_X^\bullet \cong \omega_Y^\bullet$.
		Since a scheme admitting a dualizing complex is Gorenstein if and only if it admits an invertible dualizing complex, we observe that if $X$ is Gorenstein, then $Y$ is Gorenstein. 
		Likewise, since a scheme  admitting a dualizing complex is Cohen--Macaulay if and only if it admits a dualizing complex that is the shift of a sheaf, we observe that if $X$ is Cohen--Macaulay and $\pi$ is flat, then $Y$ is Cohen--Macaulay.
	\end{remark}
	
	\begin{remark} \label{rmk:uppershriek_condition} 
		A proper surjective morphism of integral excellent Noetherian schemes $\pi\colon Y \to X$ such that $\pi^!\mcO_X \cong \mcO_Y$ and such that $X$ admits a dualizing complex on a dense open subset is generically finite.
		Indeed, the regular locus $X_{\mathrm{reg}}$ of $X$ is open (since $X$ is excellent) and dense (since $X$ is integral) and, 
		likewise, the regular locus $U$ of $\pi^{-1}(X_{\mathrm{reg}})$ is open and dense.
		Denote by $p\colon U \to X_{\mathrm{reg}}$ the restriction of $\pi$. 
		Since the restriction to open subsets commutes with exceptional inverse image functors \cite[\href{https://stacks.math.columbia.edu/tag/0G4J}{Tag 0G4J}]{stacks-project}, we have $p^! \mcO_{X_{\mathrm{reg}}} \cong \mcO_U$.
		By \cref{lem:k_and_o_equvalence}, we have an isomorphism $\omega_U[\dim Y] \cong  p^*\omega_{X_{\mathrm{reg}}}[\dim X]$, and thus $\dim Y = \dim X$.
	\end{remark}

	\subsection{Exceptional inverse image along finite morphisms}\label{sec:upper_shriek_finite_morphism}

	For the sake of the proof of the lifting \cref{lem:splinter_ascend_pi_!}, we will need to consider the exceptional inverse functor for finite morphisms, and in particular for cases~\labelcref{item:etale} and~\labelcref{item:reducedfibers}, for finite morphisms between
	possibly non-Noetherian schemes.
	
Let $\pi\colon Y \to X$ be a finite morphism of Noetherian schemes.
The exact functor $\pi_* \colon \Coh(Y) \to \Coh(X)$ admits a right adjoint, which by abuse 
we denote by $\pi^{!}\colon \Coh(X) \to \Coh(Y)$, given by
$\pi^!(\mcF) \coloneqq \sheafhom_{\mcO_X}(\pi_* \mcO_Y, \mcF)$
considered as 
 an $\mcO_Y$-module, see, e.g., \cite[Ch.~III, Ex.~6.10]{hartshorne_algebraic_geometry} or \cite[\href{https://stacks.math.columbia.edu/tag/0AWZ}{Tag 0AWZ}]{stacks-project}.
In the context of \cref{SS:uppershriek}, the right adjoint of $\pi_* = \mathrm{R}\pi_* \colon \sfD_{\mathrm{Coh}}(\mcO_X) \to \sfD_{\mathrm{Coh}} (\mcO_Y)$  can be obtained from $\pi^{!} = \sheafhom_{\mcO_X}(\pi_* \mcO_Y, -)$ by forming the right derived functor \cite[\href{https://stacks.math.columbia.edu/tag/0AU3}{Tag 0AU3}]{stacks-project}.
If $\pi$ is finite and flat, then $\pi_* \mcO_Y$ is a vector bundle and  so $\sheafhom_{\mcO_X}(\pi_* \mcO_Y, -) = \mathrm{R}\sheafhom_{\mcO_X}(\pi_* \mcO_Y, -)$,
 i.e., in that case, the notation $\pi^!$ is unambiguous.
 In general, in order to avoid any ambiguity, we will write $\pi^!\mcO_X \cong \mcO_Y$ in $\mathrm{Coh}(Y)$ if there is an isomorphism of $\mcO_Y$-modules 
$\sheafhom_{\mcO_X}(\pi_* \mcO_Y, \mcO_X) \cong \mcO_Y$, and we will write $\pi^!\mcO_X \cong \mcO_Y$ in $\sfD_{\mathrm{Coh}} (\mcO_Y)$ if there is an isomorphism $\mathrm{R}\sheafhom_{\mcO_X}( \pi_* \mcO_Y, \mcO_X) \cong \mcO_Y$ in $\sfD_{\mathrm{Coh}} (\mcO_Y)$.

If $\pi \colon Y \to X$ is a finite morphism of quasi-compact quasi-separated schemes, $\mathrm{R}\pi_*\colon \sfD_{\mathrm{QCoh}}(\mcO_Y) \to  \sfD_{\mathrm{QCoh}}(\mcO_Y)$ admits $\pi^\times$ as a right adjoint. 
If $\pi$ is finite and \emph{pseudo-coherent}, then $\pi^\times$ coincides with $\mathrm{R}\sheafhom_{\mcO_X}(\pi_* \mcO_Y, -)$ on $\sfD^+_{\mathrm{QCoh}}(\mcO_X)$ \cite[\href{https://stacks.math.columbia.edu/tag/0AX2}{Tag 0AX2}]{stacks-project}. A morphism of finite type of Noetherian schemes is pseudo-coherent \cite[\href{https://stacks.math.columbia.edu/tag/0684}{Tag 0684}]{stacks-project} and a flat base change of a pseudo-coherent morphism is pseudo coherent \cite[\href{https://stacks.math.columbia.edu/tag/0680}{Tag 0680}]{stacks-project}.
We will use this more general adjoint in the proof of \cref{lem:splinter_ascend_pi_!} for a flat base change of a finite morphism of Noetherian schemes.

\begin{remark}
	Let $\pi\colon Y \to X$ be a finite morphism of Noetherian schemes. 
	By taking the cohomology sheaves in degree zero, one sees that if $\pi^!\mcO_X \cong \mcO_Y$ in $\sfD_{\mathrm{Coh}} (\mcO_Y)$, then $\pi^!\mcO_X \cong \mcO_Y$ in $\mathrm{Coh}(Y)$.
\end{remark}

\begin{remark}
	A finite morphism $\pi\colon Y \to X$ of Noetherian schemes with $X$ integral and such that $\pi^!\mcO_X \cong \mcO_Y$ in $\mathrm{Coh}(Y)$ is surjective. Indeed, the question is local on $X$ and we may assume both $X$ and $Y$ are affine, say $X=\Spec A$ and $Y = \Spec B$. 
	Let $I \coloneqq \ker (\pi^\sharp \colon A \to B)$. 
	Then any $\varphi \in \Hom_A (B,A)$ has image contained in $\mathrm{Ann}_A(I)$. 
	Since $A$ is a domain, $\Hom_A (B,A) \cong B \neq 0$ forces $I$ to be zero and hence $\pi$ to be surjective.
\end{remark}

\subsection{Traces for finite morphisms}\label{ss:trace}
The \emph{classical trace map} $\mathrm{tr}_{Y/X}\colon  \pi_* \mcO_Y \to \mcO_X$ for a finite surjective morphism $\pi \colon Y \to X$ of integral schemes with $X$ normal is defined as follows.
Denote $K(X)$ and $K(Y)$ the function fields of $X$ and $Y$, respectively. 
On affine open subsets $U = \Spec A \subseteq X$ with $\pi^{-1}(U) = \Spec B \subseteq Y$, $\mathrm{tr}_{B/A} \colon B \to A$ sends $b\in B$ to the trace of the multiplication by $b$ map on $K(Y)$ viewed as a $K(X)$-vector space and, by normality of $A$, defines an element of $A$.

If $U\subseteq X$ is any open subset in an integral scheme $X$, then $H^0(U, \mcO_X) = \bigcap_{x \in U} \mcO_{X,x} \subseteq K(X)$, and if $X$ is further assumed to be normal, then  $H^0(U, \mcO_X)$ is integrally closed in $K(X)$.
As such, the classical trace map $\mathrm{tr}_{Y/X} \colon \pi_* \mcO_Y \to \mcO_X$ is given on open subsets $U\subseteq X$ as the map
sending a section $s \in H^0(\pi^{-1}(U), \mcO_Y)$ to the trace of the multiplication by $s$ map on $K(Y)$ viewed as a $K(X)$-vector space.
By \cref{lem:integral_on_global_sections} below, 
such a section $s$ is integral over $H^0(U, \mcO_X)$ and hence, the trace is an element of $H^0(U, \mcO_X)$.

\begin{lemma}\label{lem:integral_on_global_sections}
	Let $\pi \colon Y \to X$ be an integral surjective morphism of integral schemes with $X$ normal.
	Then $\pi^\sharp \colon H^0(U, \mcO_X) \to H^0(\pi^{-1}(U), \mcO_Y)$ is an integral ring extension for every open subset $U\subseteq X$.
	Moreover, if $Y$ is
	 normal, then $\pi$ is the normalization of $X$ in $K(Y)$ and for every open subset $U \subseteq X$, $H^0(\pi^{-1}(U), \mcO_Y)$ is the integral closure of $H^0(U, \mcO_X)$ in the function field $K(Y)$.
\end{lemma}
\begin{proof}
		Without loss of generality we can assume $U=X$.
	Let $\Spec A \subseteq X$ be an affine open subset and let $\Spec B  \coloneqq \pi^{-1}(\Spec A) \subseteq Y$ be the pre-image under $\pi$.
Let $s \in S\coloneqq H^0(Y,\mcO_Y)$ be nonzero. 
Since $A \to B$ is integral, there exists a monic irreducible polynomial $P_{A} \in A[T]$ such that $P_{A}(s)=0$ in $B$.
By normality of $X$, $A$ is integrally closed in $K(X)$, and it follows that $P_{A}$ is irreducible as a polynomial in $K(X)[T]$.
Indeed, if $P_{A} = QR$ in $K(X)[T]$ with $Q$ and $R$ monic, then
denoting $C$ the integral closure of $A$ in a splitting field of $P_A$ over $K(X)$, we  have $Q, R \in C[T] \cap K(X)[T] = A [T]$. 
Let now $\Spec A' \subseteq X$ be another affine open subset and let $P \in K(X)[T]$ be the gcd of $P_{A}$ and~$P_{A'}$. Then $P(s) =0$, so $P$ is not constant as $s \neq 0$ in $K(Y)$ by integrality of $Y$. As $P_{A}$ and $P_{A'}$ are irreducible, we obtain $P_{A} = P = P_{A'}$.
It follows that $P=P_{A}$ has coefficients in $R = H^0(X, \mcO_X) = \bigcap A$, where the intersection runs over all affine open subsets of $X$. Thus $s$ is integral over $R$.
		
		If $Y$ is normal, then $H^0(V, \mcO_Y)=\bigcap_{y \in V} \mcO_{Y,y} \subseteq K(Y)$ is integrally closed for every open subset~$V \subseteq Y$.
		Hence, for every open subset $U\subseteq X$, $H^0(\pi^{-1}(U), \mcO_Y)$ is the integral closure of $H^0(U, \mcO_X)$ in $K(Y)$.
		This coincides with the relative normalization in \cite[\href{https://stacks.math.columbia.edu/tag/0BAK}{Tag 0BAK}]{stacks-project} if $U$ is affine.
\end{proof}

The following lemma  is a slight generalization of  \cite[Lem.~9]{speyer_frob_split_subvars} (we do not assume that the fields are the fraction fields of the integral rings involved). It will be used in the proofs of \cref{lem:splinter_ascend_pi_!}\labelcref{item:etale}
  with $R = H^0(X,\mcO_X), S=H^0(Y,\mcO_Y)$, $K=K(X)$, and $L=K(Y)$.

\begin{lemma}\label{lem:speyer}
	Let $R \hookrightarrow S$ be an integral
	extension of integral domains fitting into a commutative diagram
	\[
	\begin{tikzcd}
	R \dar[hookrightarrow] \rar[hookrightarrow]& S \dar[hookrightarrow]\\
	K \rar[hookrightarrow] & L
	\end{tikzcd}
	\]
	with $K \to L$ a finite field extension such that $R$ and $S$ are integrally closed in $K$ and $L$, respectively.
	Then the classical trace $\mathrm{tr}_{L/K} \colon L \to K$ sends $S$ to $R$, 
	and sends $\sqrt{\mathfrak{p}S}$ to $ \mathfrak{p}$  for every prime ideal $\mathfrak{p} \subseteq R$.
\end{lemma}
\begin{proof}
	If $L/K$ is not separable, then $\mathrm{tr}_{L/K}$ is zero. 
	Thus we can and do assume that $L/ K$ is separable.
	Let $M$ be the Galois closure of $L$ over $K$, denote by $G= \mathrm{Gal}(M/K)$ the Galois group, and let $H \subseteq G$ be the stabilizer of $L$.
	Let $T\subseteq M$ be the integral closure of $R$ in $M$.
	Since $R \hookrightarrow S$ is integral, $S \subseteq T$.
	Note that  for any $s\in L$, we have $\mathrm{tr}_{L/ K} (s) = \sum_{g \in G/H} g(s)$, where $g$ runs over a set of coset representatives of $G/H$.
	By definition of the integral closure, any automorphism $ g \in G$ satisfies $g(T) = T \subseteq M$, so $\mathrm{tr}_{L/ K} (s)$ lies in $T\cap K = R$ for any $s\in S$.
	Similarly, any $g \in G$ preserves $\sqrt{T \mathfrak{p}} \subseteq T$.
	Hence, for any $s \in \sqrt{S \mathfrak{p}} \subseteq \sqrt{T \mathfrak{p}}$, $\mathrm{tr}_{L/ K} (s)$ lies in $\sqrt{T \mathfrak{p}}\cap R = \mathfrak{p}$.
\end{proof}

\begin{remark}
  Note that in the situation of \cref{lem:speyer}, the classical trace map $\mathrm{tr}_{L/K} \colon L \to K$ need not restrict to the classical trace map $\mathrm{tr}_{S/R} \colon S \to R$. Consider for instance the Frobenius $F \colon \bbP^1_{\bbF_p} \to \bbP^1_{\bbF_p}$. It induces the identity on global sections but on function fields induces the purely inseparable field extension ${\bbF_p}(t) \to {\bbF_p}(t^{1/p})$.
\end{remark}

The \emph{classical trace map} $\mathrm{tr}_{Y/X}\colon  \pi_* \mcO_Y \to \mcO_X$ for a finite flat morphism $\pi \colon Y \to X$ locally of finite presentation is 
given by the $\mcO_X$-linear map sending a local section $s$ of the finite locally free $\mcO_X$-module~$\pi_*\mcO_Y$ to the trace of the multiplication by $s$ on $\pi_*\mcO_Y$.
If $X$ is further assumed to be normal, and $X$ and $Y$ to be integral, the classical trace map $\mathrm{tr}_{Y/X}$ coincides with the previous construction (hence the same nomenclature for both trace maps). Indeed, a local basis of $\pi_* \mcO_Y$ as an $\mcO_X$-module gives a local basis of $K(Y)$ as a $K(X)$-vector space.\medskip

Let $\pi\colon Y \to X$ be a finite surjective morphism of Noetherian schemes. 
To give an isomorphism $\pi^!\mcO_X \cong \mcO_Y$ in $\Coh(Y)$ is equivalent to give an isomorphism of $\mcO_Y$-modules
$$\mcO_Y \xrightarrow{\cong} \sheafhom_{\mcO_X}(\pi_*\mcO_Y, \mcO_X),$$
or equivalently, if $\pi$ is further assumed to be flat, to produce an $\mcO_X$-linear map 
$\mathrm{Tr}_{Y/X} \colon \pi_*\mcO_Y \to \mcO_X$
such that the symmetric bilinear form $\mathrm{Tr}_{Y/X}(\alpha \cdot \beta)$ on the locally free sheaf $\pi_*\mcO_Y$ with values in~$\mcO_X$ is nonsingular. 
Such a nonsingular trace map $\mathrm{Tr}_{Y/X}$ then spans $\pi^!\mcO_X$ freely as an $\mcO_Y$-module.

If $\pi$ is finite flat, or if $Y$ is integral and $X$ integral normal, then the classical trace map provides a canonical morphism of $\mcO_Y$-modules
$$\mcO_Y \longrightarrow \sheafhom_{\mcO_X}(\pi_*\mcO_Y, \mcO_X), \quad 1 \mapsto \mathrm{tr}_{Y/X}.$$
We will write $\pi^!\mcO_X = \mcO_Y$ in $\Coh(Y)$ if this canonical morphism is an isomorphism.
\medskip

The classical trace map $\mathrm{tr}_{Y/X}\colon  \pi_* \mcO_Y \to \mcO_X$ for a finite flat morphism $\pi \colon Y \to X$ of schemes is nonsingular
if and only if $\pi$ is finite \'etale \cite[\href{https://stacks.math.columbia.edu/tag/0BVH}{Tag 0BVH}]{stacks-project}.
  In fact, the nonsingularity of $\mathrm{tr}_{Y/X}$ characterizes finite surjective quasi-\'etale morphisms, i.e., finite surjective morphisms $\pi\colon Y \to X$ such that there exists an open subset $U \subseteq X$ with $\codim_X (X\setminus U) \geq 2$ and such that $\pi\vert_U \colon \pi^{-1}(U) \to U$ is \'etale\,:

\begin{lemma}\label{lem:quasi-etale}
	Let $\pi\colon Y \to X$ be a finite surjective morphism of integral normal Noetherian schemes. Consider the following statements\,:
	\begin{enumerate}
				\item \label{item:quasi-etale_first_lemma} $\pi$ is quasi-\'etale.
		\item \label{item:classical_trace_generates_upper_shriek} 
	$\pi^!\mcO_X = \mcO_Y$ in $\Coh(Y)$,
		  i.e., the classical trace is nonsingular.
	\end{enumerate}
	Then $\labelcref{item:quasi-etale_first_lemma} \Rightarrow \labelcref{item:classical_trace_generates_upper_shriek}$, and $\labelcref{item:classical_trace_generates_upper_shriek} \Rightarrow \labelcref{item:quasi-etale_first_lemma}$ holds if $X$ is universally catenary, e.g., if $X$ is Cohen--Macaulay.
\end{lemma}
\begin{proof}
Assume $\pi$ is quasi-\'etale.
	Let $U\subseteq X$ be an open subset with $\codim_X (X\setminus U)\geq 2$ such that  $\pi\vert_U \colon \pi^{-1}(U) \to U$ is finite \'etale.
	The map $\mcO_Y \to \pi^!\mcO_X = \sheafhom_{\mcO_X}(\pi_*\mcO_Y, \mcO_X), 1 \mapsto \mathrm{tr}_{Y/X}$ is a map of quasi-coherent sheaves on $Y$ that restricts to an isomorphism on $\pi^{-1}(U)$, and hence is an isomorphism on $Y$ since both sheaves are reflexive after applying $\pi_*$.
Conversely, we claim that if $\pi\colon Y \to X$ is a finite surjective morphism with $Y$ integral, and $X$ normal and universally catenary,  then there exists an open subset $U\subseteq X$ with $\codim_X (X\setminus U)\geq 2$ such that  $\pi\vert_U \colon V\coloneqq \pi^{-1}(U) \to U$ is finite flat. 
This suffices to conclude since by \cite[\href{https://stacks.math.columbia.edu/tag/0BVH}{Tag 0BVH}]{stacks-project}  the classical trace for $\pi\vert_U$ is nonsingular if and only if $\pi\vert_U$ is \'etale.
Let $V\subseteq Y$ be the flat locus of $\pi$, which by \cite[\href{https://stacks.math.columbia.edu/tag/0398}{Tag 0398}]{stacks-project} is open, and let $Z \coloneqq Y\setminus V$. 
Since $X$ is universally catenary, any codimension-1 point of $Y$ is mapped to a codimension-1 point of $X$ by \cite[\href{https://stacks.math.columbia.edu/tag/02JT}{Tag 02JT}]{stacks-project}, and by \cite[Ch.~III, Prop.~9.7]{hartshorne_algebraic_geometry}, $V$ contains all codimension-1 points of $Y$. Hence $Z$ has codimension at least 2 in $Y$, and so $\pi(Z)$ defines a closed subscheme of codimension at least 2
in $X$. We may then choose $U\coloneqq X\setminus \pi(Z)$.
\end{proof}

\section{Preliminaries on splinters and globally $F$-regular varieties}
	\label{S:prelim_splinters_globF}
\subsection{Splinters} \label{SS:splinter}
We review the notion of \emph{splinter} for Noetherian schemes, the local constraints it imposes, as well as the global constraints it imposes on proper schemes over a field of positive characteristic.

\begin{definition}
	A Noetherian scheme $X$ is a \emph{splinter} if for any finite surjective morphism $f \colon Y \to X$ the canonical map $\mcO_X \to f_*\mcO_Y$ splits in the category $\Coh(X)$ of coherent sheaves on $X$. 
	A Noetherian scheme $X$ is a \emph{derived splinter} if for any proper surjective morphism $f\colon Y \to X$ the map $\mcO_X \to \mathrm{R} f_*\mcO_Y$ splits in the bounded derived category $\sfD^b(X)$ of coherent sheaves on $X$.
\end{definition}

Note that a derived splinter is a splinter, so that being a derived splinter is \emph{a priori} more restrictive than being a splinter.
 In the recent work \cite{bhatt_et_al_globally_+_regular_varieties_and_mmp_for_threefolds_in_mixed_char}, the notion of splinter has been extended to pairs. Precisely\,:

\begin{definition}[{\cite[Def.~6.1]{bhatt_et_al_globally_+_regular_varieties_and_mmp_for_threefolds_in_mixed_char}}] \label{def:glob+reg}
	 	 Let $(X,\Delta)$ be a pair consisting of a normal excellent Noetherian scheme $X$ and of an effective $\bbQ$-Weil divisor $\Delta$.
	 	 The pair $(X,\Delta)$ is called \emph{globally $+$-regular} if 
	 	   for any finite surjective morphism $f\colon Y \to X$ with $Y$ normal, 
	 the natural map $\mcO_X \to f_* \mcO_Y(\lfloor f^*\Delta\rfloor)$ splits in $\mathrm{Coh}(X)$.
	\end{definition}

 \begin{remark}  \label{rmk:splinter-vs-globally+}
A normal excellent Noetherian scheme $X$ is a splinter if and only if $(X,0)$ is globally $+$-regular. 
	The ``only if'' statement is obvious. For the converse, recall from \cite[\href{https://stacks.math.columbia.edu/tag/035S}{Tag 035S}]{stacks-project} that
	the normalization of an excellent scheme is a finite morphism.
	Thus, if $\pi\colon Y\to X$ is a finite surjective morphism, then $Y$ is excellent so that the normalization $\nu\colon Y^\nu \to Y$ is finite. Any splitting of $\mcO_X \to (\pi\circ \nu)_*\mcO_{Y^\nu}$ provides a splitting of $\mcO_X \to \pi_*\mcO_{Y}$.
\end{remark}

\begin{remark}
	Note that \cite[Def.~6.1]{bhatt_et_al_globally_+_regular_varieties_and_mmp_for_threefolds_in_mixed_char} only consider schemes whose closed points have residue fields of positive characteristic.
	A reason for this is outlined in  \cite[Rmk.~6.3]{bhatt_et_al_globally_+_regular_varieties_and_mmp_for_threefolds_in_mixed_char}.
	Since for our applications this condition is not necessary, we deviate from that convention.
	We also note that in this work the excellence condition on $X$ can be weakened to requiring $X$ to be universally catenary (in order to define pullbacks of $\bbQ$-Weil divisors along finite surjective morphisms) and Nagata (in order to dominate any finite cover of $X$ by a normal finite cover).
\end{remark}

In general, if a Noetherian scheme $X$ is a splinter, then it is a basic fact that any open $U \subseteq X$ is a splinter\,; see, e.g., \cref{lem:splinter_invariant_codim_2_surgery}\labelcref{item:open_of_splinter} below.
Hence, if $X$ is a splinter, then all of its local rings are splinters.
Moreover, if $X$ is in addition assumed to be affine,  
$X$ is a splinter if and only if all its local rings are splinters \cite[Lem.~2.1.3]{datta_tucker_on_some_permanence_properties_of_derived_splinters}.

In characteristic zero, a Noetherian scheme $X$ is a splinter if and only if it is normal \cite[Ex.~2.1]{bhatt_derived_splinters_in_positive_characteristic}, and it is a derived splinter if and only if it has rational singularities \cite[Thm.~2.12]{bhatt_derived_splinters_in_positive_characteristic}.
In particular, in characteristic zero, the splinter and derived splinter properties are distinct, and they
    both define local properties.

In positive characteristic, Bhatt showed that the splinter and the derived splinter properties agree~\cite[Thm.~1.4]{bhatt_derived_splinters_in_positive_characteristic} and observed that, in contrast to the affine setting,
the splinter property is not a local property for proper schemes. 
The following proposition summarizes the known local constraints on splinters in positive characteristic.

\begin{proposition}[{\cite[Ex.~2.1, Rmk.~2.5, Cor.~6.4]{bhatt_derived_splinters_in_positive_characteristic}}, {\cite[Rmk.~5.14]{bhatt2021cohenmacaulayness}}, \cite{singh_q_gorenstein_splinter_rings_of_characteristic_p_are_f_regular}, {\cite[\S 2.2]{smith_globally_f_regular_varieties_applications_to_vanishing_theorems_for_quotients_of_fano_varieties}}]
	\label{prop:bhatt_normal_CM}
	Let $X$ be a scheme of finite type over a field of positive characteristic. If $X$ is a splinter, 
then
	\begin{enumerate}
		\item $X$ is normal\,;
		\item $X$ is Cohen--Macaulay\,;
		\item $X$ is pseudo-rational\,;
		\item $X$ is $F$-rational.
	\end{enumerate}
Moreover, if $X$ is $\bbQ$-Gorenstein and $F$-finite, then its local rings are strongly $F$-regular.
\end{proposition}

Recall from \cite{hochster_huneke_tight_closure_and_strong_f_regularity} that a
 ring $R$ of positive characteristic is \emph{strongly $F$-regular} if it is $F$-finite and if for any $c\in R$ not belonging to any minimal prime ideal of $R$ there exists $e>0$ such that the inclusion of $R$-modules $R \hookrightarrow F_*^e R$ which sends 1 to $F^e_*c$ splits as a map of $R$-modules.
The ring $R$ is strongly $F$-regular if and only if its local rings are strongly $F$-regular.
If $R$ is strongly $F$-regular, then the affine scheme $X=\Spec R$ is a splinter\,; see, e.g., \cite[Rmk.~2.13(2)]{datta_tucker_on_some_permanence_properties_of_derived_splinters}.
\medskip

The splinter property also imposes strong constraints on the global geometry of proper varieties in positive characteristic. For example, from Bhatt's ``vanishing up to finite cover in positive characteristic'' \cite[Thm.~1.5]{bhatt_derived_splinters_in_positive_characteristic}, we have\,:

\begin{proposition}[\cite{bhatt_derived_splinters_in_positive_characteristic}]
	\label{prop:semiample}
	Let $X$ be a proper variety over a field of positive characteristic and let $\mcL$ be a semiample invertible sheaf on $X$. 
	If $X$ is a splinter, then $H^i(X, \mcL) = 0$ for all $i>0$. In particular, 	$H^i(X, \mcO_X)=0$ for all $i>0$.
\end{proposition}
\begin{proof}
	For a proper variety $X$ over a field of positive characteristic, there exists, by \cite[Prop.~7.2]{bhatt_derived_splinters_in_positive_characteristic}, for any $i>0$ a finite surjective morphism $\pi\colon Y \to X$ such that the induced map $H^i(X,\mcL) \to H^i(Y,\pi^*\mcL)$ is zero.
	If now $X$ is a splinter, the pullback map $\mcO_X \to \pi_*\mcO_Y$ admits a splitting $s$, i.e., we have
	$$ \id \colon \mcO_X \to \pi_*\mcO_Y \stackrel{s}{\to} \mcO_X.$$
	Tensoring with $\mcL$ and using the projection formula, we obtain
	$$ \id \colon H^i(X,\mcL) \xrightarrow{0}  H^i(Y,\pi^*\mcL) = H^i(X,\pi_* \pi^*\mcL) \to H^i(X,\mcL),$$
	where the equality in the middle uses that $\pi$ is finite, in particular affine.
	We conclude that $H^i(X,\mcL)=0$.
\end{proof}

As a direct consequence, we have the following useful constraint, which will be refined in \cref{prop:K_equvialenceodaira_dim_f_split_variety}, on proper splinters in positive characteristic\,:

\begin{lemma}\label{lem:splinter_trivial_canonical_sheaf}
	Let $X$ be a proper scheme over a field of positive characteristic, with positive-dimensional irreducible components. If $X$ is a splinter, then its canonical divisor $K_X$ is not effective, in particular its dualizing sheaf $\omega_X$ is nontrivial.
\end{lemma}
\begin{proof} Since a splinter is normal, by
	working on each connected component of $X$ separately, we can assume that $X$ is of pure positive  dimension, say~$n$. 
	By \cref{prop:bhatt_normal_CM}, $X$ is Cohen--Macaulay, so $ \omega_X^\bullet = \omega_X [n]$. By \cref{prop:semiample} and Serre duality for Cohen--Macaulay schemes, we obtain $H^0(X, \omega_X)^\vee \cong H^n(X, \mcO_X)=0$.
	This shows that the Weil divisor $K_X$ is not effective.
\end{proof}

\subsection{Globally $F$-regular varieties} Let $p$ be a prime number.
We recall the  notion of \emph{global $F$-regularity} for normal  varieties over an $F$-finite field of characteristic $p$, 
and 
review the local constraints it imposes, as well as the global constraints it imposes on proper varieties.

\begin{definition}[\cite{smith_globally_f_regular_varieties_applications_to_vanishing_theorems_for_quotients_of_fano_varieties, schwede_smith_globally_f_regular_and_log_Fano_varieties}]
	A normal scheme $X$ over $\bbF_p$ is called \emph{globally $F$-regular} if it is  $F$-finite and if for any effective Weil divisor $D$ on $X$ there exists a positive integer $e \in \bbZ_{>0}$ such that the map
	$$\mcO_X \to F_*^e \mcO_X \xrightarrow{F_*^e(\sigma_D)} F_*^e \mcO_X(D)$$
	of $\mcO_X$-modules splits. Here $\sigma_D \colon  \mcO_X \to \mcO_X(D)$ is the morphism determined by the Weil divisor $D$.
	
	\noindent A normal scheme $X$ over $\bbF_p$ is called \emph{$F$-split} if it is $F$-finite and if $\mcO_X \to F_*\mcO_X$ splits.
	
	\noindent A pair $(X,\Delta)$ consisting of a normal scheme $X$ over $\bbF_p$ and an effective $\bbQ$-Weil divisor~$\Delta$ is called \emph{globally $F$-regular}
	if $X$ is  $F$-finite and if for any effective Weil divisor $D$ on $X$ there exists an integer $e >0$ such that the natural map
	$$\mcO_X \to F_*^e \mcO_X(\lceil (p^e-1)\Delta\rceil + D)$$
	of $\mcO_X$-modules splits. In particular, $X$ is globally $F$-regular if and only if $(X,0)$ is globally $F$-regular.  
\end{definition}

The following well-known proposition gives local constraints on a normal variety to be globally $F$-regular and echoes \cref{prop:bhatt_normal_CM}. For the convenience of the reader, we outline the proof.

\begin{proposition}
	Let $X$ be a normal 
	 scheme over $\bbF_p$.
	If $X$ is globally $F$-regular, then its local rings are strongly $F$-regular.
\end{proposition}
\begin{proof}
	Fix a point $x\in X$.
	If $c \in \mcO_{X,x}$ is a nonzero element, $c$ defines an effective divisor in the neighborhood of $x$.
	By taking the Zariski closure, this divisor extends to a Weil divisor $D$ on $X$ such that the map $\mcO_X \to \mcO_X(D)$ localizes to the map $\mcO_{X,x} \to \mcO_{X,x}[c^{-1}]$ sending $1 \mapsto 1$.
	Since $X$ is globally $F$-regular, there exists $e>0$ such that $\mcO_X \to F_*^e \mcO_X(D)$ splits.
	Thus, localizing yields a splitting of the map $\mcO_{X,x} \to F_*^e (\mcO_{X,x}[c^{-1}] )$ sending $1 \mapsto 1$.
	Multiplying by $c$  yields a splitting of the map $\mcO_{X,x} \to F_*^e\mcO_{X,x}$ sending $1 \mapsto F_*^e c$.
\end{proof}

Global $F$-regularity is a local property for normal affine varieties. Indeed,
a normal affine variety over an $F$-finite field $k$ is globally $F$-regular 
if and only if it is strongly $F$-regular 
if and only if all its local rings are strongly $F$-regular\,; see \cite[Thm.~3.1]{hochster_huneke_tight_closure_and_strong_f_regularity}. 
By \cref{prop:bhatt_normal_CM} and the discussion that follows, we see that a $\bbQ$-Gorenstein normal affine variety over an $F$-finite field is globally $F$-regular if and only if it is a splinter. In fact, any normal globally $F$-regular variety is a splinter\,:

\begin{proposition}[{\cite[Prop.~8.9]{bhatt_derived_splinters_in_positive_characteristic}, \cite[Lem.~6.14]{bhatt_et_al_globally_+_regular_varieties_and_mmp_for_threefolds_in_mixed_char}}]
	\label{prop:gFr-splinter}
	Let $X$ be a
normal 
	excellent Noetherian scheme over $\bbF_p$ and let $\Delta$ be an effective $\bbQ$-Weil divisor. 
		 If $(X,\Delta)$ is globally $F$-regular, then $(X,\Delta)$ is globally $+$-regular. 
		 
		\noindent In particular, by \cref{rmk:splinter-vs-globally+}, assuming $X$ is a normal scheme of finite type over a field,
	 if $X$ is globally $F$-regular, then $X$ is a splinter.
\end{proposition}

As \cref{prop:semiample} in the splinter case,
global $F$-regularity imposes strong constraints on the global geometry of normal \emph{projective} varieties\,:

\begin{proposition}[{\cite[Cor.~4.3]{smith_globally_f_regular_varieties_applications_to_vanishing_theorems_for_quotients_of_fano_varieties}}, {	\cite[Thm.~1.1]{schwede_smith_globally_f_regular_and_log_Fano_varieties}}]
	\label{prop:gFr-big}
Let $X$ be a normal projective variety over a
field of positive characteristic.
 Assume that~$X$ is globally $F$-regular. Then\,:
\begin{enumerate}
	\item For all nef invertible sheaves $\mcL$ on $X$, $H^i(X, \mcL)=0$ for all $i>0$. 
	\item $X$ is log Fano\,; in particular, if $X$ is in addition $\bbQ$-Gorenstein, then $-K_X$ is big.
	\label{item:globF_log_Fano}
\end{enumerate}
\end{proposition}

	A  proper normal curve over an algebraically closed field $k$ of positive characteristic is globally $F$-regular if and only if it is splinter if and only if it is isomorphic to the projective line.
	It follows from the proof of \cite[Thm.~6.2]{kawakami_totaro_endomorphisms_of_varieties_and_bott_vanishing} that a smooth projective Fano variety over a perfect field of positive characteristic is globally $F$-regular if and only if it is a splinter if and only if it is $F$-split.
	The following conjecture stems
	from \cref{prop:gFr-big}\labelcref{item:globF_log_Fano} and the folklore expectation
	(see, e.g., \cite[Rmk.~6.16]{bhatt_et_al_globally_+_regular_varieties_and_mmp_for_threefolds_in_mixed_char}) that splinters should be globally $F$-regular.

	\begin{conjecture}[{\cite[Conj.~6.17]{bhatt_et_al_globally_+_regular_varieties_and_mmp_for_threefolds_in_mixed_char}}] 
		\label{conj:splinter_big_anticanonical_class}
		Let $X$ be a $\bbQ$-Gorenstein projective scheme over a field of positive characteristic. If $X$ is a splinter, then 
		$-K_X$ is big.
	\end{conjecture}

\section{Invariance under small birational maps}
\label{S:invariance_small_birat_maps}

We start by describing the behavior of the splinter property under open embeddings.
Recall that the normalization of a Nagata scheme is finite \cite[\href{https://stacks.math.columbia.edu/tag/035S}{Tag 035S}]{stacks-project}, and that, 
if $Y \to X$ is locally of finite type and $X$ is Nagata, then $Y$ is Nagata \cite[\href{https://stacks.math.columbia.edu/tag/0359}{Tag 0359}]{stacks-project}.
Moreover, any excellent scheme is Nagata \cite[\href{https://stacks.math.columbia.edu/tag/07QS}{Tag 07QS}]{stacks-project}.

	\begin{lemma}\label{lem:splinter_invariant_codim_2_surgery}
		Let $X$ be a Noetherian scheme
		and let $U\subseteq X$ be an open dense subset.
		\begin{enumerate}
			\item \label{item:open_of_splinter} If $X$ is a 
			splinter, then $U$ is a
			splinter.
			\item \label{item:splinter_determined_on_big_open} Assume that $X$ is normal and Nagata, and that $\codim_X (X \setminus U) \geq 2$. 
			If~$U$ is a splinter, then $X$ is a splinter.
		\end{enumerate}
	\end{lemma}
	\begin{proof}
		To prove \labelcref{item:open_of_splinter} consider a finite cover $f\colon Y \to U$.
		By Zariski's Main Theorem \cite[\href{https://stacks.math.columbia.edu/tag/05K0}{Tag 05K0}]{stacks-project},
		this cover extends to a  finite morphism $\bar{f}\colon \overline{Y} \to X$.
		Since $X$ is a splinter, we obtain a section $s$ such that the composition
		$$ \mcO_X \to \bar{f}_* \mcO_{\overline{Y}} \xrightarrow{s} \mcO_X$$
		is the identity. By restricting to $U$, we obtain the desired section of $\mcO_U \to f_*\mcO_Y$.
		
		For \labelcref{item:splinter_determined_on_big_open} consider a finite cover $f\colon Y \to X$.
		Since $X$ is Nagata, so is $Y$. By possibly replacing $Y$ by its normalization,
		we can assume that $Y$ is normal.
		The sheaf $f_*\mcO_Y$ satisfies the property $S_2$ by \cite[Prop.~5.7.9]{ega_iv_2} and is therefore reflexive. 
		Since $U$ is a splinter, we obtain a splitting of  $\mcO_U \to f_*\mcO_Y \vert_U$ and this extends to a splitting of $\mcO_X \to f_* \mcO_Y$ as $X$ is normal and all the involved sheaves are reflexive.
	\end{proof}

\begin{remark}\label{rmk:dense_open_glob_+_reg}
	Let $\Delta$ be an effective $\bbQ$-Weil divisor on a normal excellent Noetherian scheme $X$
	and let $U\subseteq X$ be an open dense subset. As the proof is analogous to the one of \cref{lem:splinter_invariant_codim_2_surgery}, we leave it to the reader to verify that if $(X, \Delta)$ is globally $+$-regular, then $(U, \Delta\vert_U)$ is globally $+$-regular. Conversely, if  $\codim_X (X \setminus U) \geq 2$ and 
	if $(U, \Delta\vert_U)$ is globally $+$-regular, then $(X, \Delta)$ is globally $+$-regular.
\end{remark}

	\begin{lemma}\label{lem:generic_fiber_splinter}
		Let $\pi \colon X \to Y$ be a surjective morphism of Noetherian schemes.
		Assume that $X$ is a  splinter and that $Y$ is integral with generic point $\eta$, then the generic fiber $X_\eta$ 
		 is a splinter.
	\end{lemma}
	\begin{proof}
		Let $f \colon Z \to X_\eta$ be a finite surjective morphism. 
		Since $f$ is of finite presentation, it spreads to a morphism of finite presentation $\mathcal{Z} \to X|_V$~\cite[Thm.~8.8.2]{ega_iv_3} for some dense affine open $V \subseteq Y$, and 
		the latter restricts on a dense affine open $U \subseteq V$ to a finite surjective morphism $\mathcal{Z}|_U \to X|_U$~\cite[Thm.~8.10.5]{ega_iv_3}.
		By \cref{lem:splinter_invariant_codim_2_surgery} the scheme $X_U$ is a splinter.
		Thus $\mcO_{X_U} \to {f_U}_* \mcO_{Z_U}$ admits a section. 
		By flat base change, restricting to the generic fiber yields the desired splitting.
	\end{proof}

   We now turn to the analogues of  \cref{lem:splinter_invariant_codim_2_surgery} and \cref{lem:generic_fiber_splinter} in the global $F$-regular setting.
	The following elementary lemma appears, e.g., in \cite[Lem.~1.5]{gongyo_li_patakfalvi_zsolt_schwede_tanaka_zong_on_rational_connectedness_of_globally_f_regular_threefolds}.

\begin{lemma}\label{lem:small_birational_map_globF}
	Let $(X,\Delta)$ be a pair consisting of a normal variety over a field of positive characteristic and of an effective $\bbQ$-Weil divisor $\Delta$, and let $U \subseteq X$ be an open subset.
	\begin{enumerate}
		\item \label{item:open_of_gFr} If $(X,\Delta)$ is globally $F$-regular, then $(U,\Delta\vert_U)$ is globally $F$-regular.
		\item \label{item:gFr_determined_on_big_open} Assume that $U\subseteq X$ is dense and $\codim_X (X \setminus U) \geq 2$. 
		If~$(U,\Delta\vert_U)$ is globally $F$-regular, then $(X,\Delta)$ is globally $F$-regular.
	\end{enumerate}
\end{lemma}
\begin{proof}
	For the proof of \labelcref{item:open_of_gFr}, consider an effective Weil divisor $D_0$ on $U$. Let $D$ be the Zariski closure of $D_0$ in $X$. 
	By assumption there exists $e>0$ such that $\mcO_X \to F_*^e \mcO_X(\lceil (p^e-1)\Delta\rceil + D)$ splits. Restricting to $U$ yields the desired splitting of $\mcO_U \to F_*^e \mcO_U(\lceil (p^e-1)\Delta\vert_U \rceil + D_0)$.
	
	To prove \labelcref{item:gFr_determined_on_big_open}, note that, since $F$ is finite, $F_*$ preserves $S_2$ sheaves by \cite[Prop.~5.7.9]{ega_iv_2}, so $F_*$ sends reflexive sheaves to reflexive sheaves.
	Furthermore, the restriction along $U\hookrightarrow X$ induces a bijection between Weil divisors of $X$ and Weil divisors of $U$.
	Statement \labelcref{item:gFr_determined_on_big_open} follows since, for any Weil divisor $D$, the map $\mcO_X \to F^e_* \mcO_X(\lceil (p^e-1)\Delta \rceil + D)$ of reflexive sheaves splits if and only if its restriction to $U$ splits.
\end{proof}

\begin{lemma}\label{lem:generic_fiber_globF}
	Let $\pi \colon X \to Y$ be a surjective morphism of normal varieties over
	 a field of positive characteristic.
	Assume that $(X,\Delta)$ is globally $F$-regular and that $Y$ is integral with generic point $\eta$. Then the generic fiber $X_\eta$ is normal and $(X_\eta, \Delta_\eta)$ is globally $F$-regular.
\end{lemma}
\begin{proof}
	By \cite[Lem.~3.5]{schwede_smith_globally_f_regular_and_log_Fano_varieties}, $X$ is globally $F$-regular. It follows from \cref{prop:gFr-splinter} that $X$ is a splinter and then 
	from \cref{lem:generic_fiber_splinter} that
	$X_\eta$ is a splinter, so $X_\eta$ is normal.
	Given an effective Weil divisor $D_0$ on $X_\eta$, we denote by $D$ its Zariski closure in $X$.
	Since $(X, \Delta)$ is globally $F$-regular, there exists $e>0$ such that $\mcO_X \to F_*^e \mcO_X(\lceil (p^e-1)\Delta \rceil +D)$ splits. 
	The lemma follows by restricting the splitting along $X_\eta \hookrightarrow X$.
\end{proof}

Finally, we draw as a consequence of the above that both the splinter property and the global $F$-regular property are invariant under \emph{small birational maps} of normal varieties.

\begin{definition}[Small birational map] \label{def:small_bir_map}
	A rational map $f\colon X \dashrightarrow Y$ between Noetherian schemes
	is said to be a \emph{small birational map} if there exist nonempty open subsets $U\subseteq X$ and $V\subseteq Y$ with $\codim_X(X\setminus U) \geq 2$ and $\codim_Y(Y\setminus V)\geq 2$ such that $f$ induces an isomorphism $U \stackrel{\simeq}{\longrightarrow} V$.
\end{definition}

\begin{proposition}\label{prop:small_birational_map}
	Let $f\colon X \dashrightarrow Y$ be a small birational map between normal schemes of finite type over a Noetherian scheme $S$. The following statements hold\,:
	\begin{enumerate}
		\item \label{item:splinter_small_birational_map} Assuming $S$ is Nagata, $X$ is a splinter if and only if $Y$ is a splinter.
		\item \label{item:globF_small_birational_map} Assuming $S= \Spec k$ with $k$ of positive characteristic,
		$X$ is globally $F$-regular if and only~if $Y$ is globally $F$-regular.
	\end{enumerate}
\end{proposition}
\begin{proof}
	Statement \labelcref{item:splinter_small_birational_map} follows directly from \cref{lem:splinter_invariant_codim_2_surgery} and statement \labelcref{item:globF_small_birational_map} from \cref{lem:small_birational_map_globF}.
\end{proof}

\section{Lifting and descending the splinter property}
\label{S:LD-splinter}

	\subsection{Descending the splinter property}\label{rmk:splittingcondition}
		If $\pi \colon Y \to X$ is a morphism of varieties such that $\mcO_X \to \pi_* \mcO_Y$ is an isomorphism, e.g., if $\pi$ is flat proper with geometrically connected and geometrically reduced fibers \cite[\href{https://stacks.math.columbia.edu/tag/0E0L}{Tag 0E0L}]{stacks-project}, or if $\pi\colon Y \to X$ is proper birational and $X$ is a normal variety, it is a formal consequence that the splinter property descends along $\pi$. Indeed, we have the following lemma.
		\begin{lemma}\label{lem:splinter_descent}
			Let $\pi\colon Y\to X$ be a morphism of Noetherian schemes. 
			\begin{enumerate}
		\item	If $Y$ is a splinter and the map $\mcO_X \to \pi_*\mcO_Y$ is split, then $X$ is a splinter.
		 \label{item:splinter_descent}
		\item \label{item:derived_splinter_descent} If $Y$ is a derived splinter and $\mcO_X \to \mathrm{R} \pi_*\mcO_Y$ is split, then $X$ is a derived splinter.
			\end{enumerate}
		Assume further that 
$X$ and $Y$ are normal and excellent,
		and let $\Delta$ be an effective
		$\bbQ$-Cartier divisor on $X$, or let $\Delta$ be an effective $\bbQ$-Weil divisor and assume that $\pi$ is finite.
		\begin{enumerate}		  \setcounter{enumi}{2}
			\item \label{item:glob_+_reg_descent} If $(Y, \pi^*\Delta)$ is globally $+$-regular and  $\mcO_X \to \pi_*\mcO_Y$ is split, then $(X, \Delta)$ is globally $+$-regular.
		\end{enumerate}
		\end{lemma}

		\begin{proof}
			We only prove \labelcref{item:glob_+_reg_descent} since \labelcref{item:splinter_descent,item:derived_splinter_descent} admit similar simpler proofs.
			First note that the condition that $\mcO_X \to \pi_*\mcO_Y$ splits forces $\pi$ to be dominant. We can assume that $X$ and $Y$ are integral.
			Thus, the pullback $\pi^* \Delta$ is well-defined if $\Delta$ is $\bbQ$-Cartier \cite[\href{https://stacks.math.columbia.edu/tag/02OO}{Tag 02OO}]{stacks-project}, or if $\pi$ is finite. 
			Let $f\colon Z \to X$ be a finite surjective morphism with $Z$ normal. Let $Z' \to  Y \times_X Z$ be the normalization of the fiber product, which is finite over $Y$ since $Y$ is 
			excellent. (This step is not needed for the proofs of \labelcref{item:splinter_descent,item:derived_splinter_descent} where one simply takes $Z'=   Y \times_X Z$.)
			We obtain a commutative square
			\[
			\begin{tikzcd}
				Z' \arrow[r, "\pi'"]  \arrow[d, "f'"]   
				&  Z\arrow[d, "f"]    \\
				Y\arrow[r,"\pi"]                &  X.
			\end{tikzcd}
			\]
			Since $(Y, \pi^*\Delta)$ is globally $+$-regular, we have a splitting
			$$\mcO_Y\to f_*' \mcO_{Z'}(\lfloor f'^*\pi^* \Delta \rfloor) \xrightarrow{t}\mcO_Y.$$
			Fixing a splitting $\mcO_X \to \pi_*\mcO_Y \xrightarrow{s} \mcO_X$, we obtain a factorization
			$$\id_{\mcO_X}\colon \mcO_X \to \pi_* \mcO_Y \to \pi_* f_*' \mcO_{Z'}(\lfloor f'^*\pi^* \Delta \rfloor) = f_* \pi'_* \mcO_{Z'}(\lfloor \pi'^* f^* \Delta \rfloor) \xrightarrow{\pi_* t} \pi_* \mcO_Y \xrightarrow{s} \mcO_X.$$
			Note that $\pi'^* \lfloor f^* \Delta \rfloor \leq \lfloor \pi'^* f^* \Delta \rfloor$. Hence, $\mcO_X \to f_* \pi'_* \mcO_{Z'}(\lfloor \pi'^* f^* \Delta \rfloor)$ factors through $\mcO_X \to f_* \mcO_Z(\lfloor f^*\Delta \rfloor) \to  f_* \pi'_* \pi'^*\mcO_{Z}(\lfloor f^* \Delta \rfloor)$ and the lemma follows.
		\end{proof}

		\begin{remark}\label{rmk:product}
			Assume $X$ and $Y$ are schemes over a field $k$. If $X\times_k Y$ is a splinter, then so are $X$ and $Y$. 
			Indeed, if $\pi\colon X\times_k Y \to X$ denotes the first projection, then $\pi_*\mcO_{X\times_k Y} = \mcO_X \otimes_k H^0(Y,\mcO_Y)$ and any splitting of the $k$-linear map $k \to H^0(Y,\mcO_Y), 1 \mapsto 1_Y$ provides a splitting of the natural map $\mcO_X \to \pi_*\mcO_{X\times_k Y}$.
		\end{remark}

\subsection{Lifting the splinter property} \label{SS:lift-splinter}
We refer to \cref{sec:upper_shriek_finite_morphism} for an overview regarding the exceptional inverse image functor for finite morphisms.
 We will say that a ring homomorphism $A \to B$ has \emph{reduced closed fibers} 
 if the fibers over closed points of the corresponding map on spectra are reduced\,; 
 this happens for instance when either $B$ is a field or $A=B$.
  The following lemma in case~\labelcref{item:etale} extends \cite[Thm.~A]{datta_tucker_on_some_permanence_properties_of_derived_splinters} and \cite[Prop.~6.20]{bhatt_et_al_globally_+_regular_varieties_and_mmp_for_threefolds_in_mixed_char}, where it is shown respectively that the splinter property lifts along \'etale morphisms of affine Noetherian schemes and along finite \'etale covers of normal excellent proper schemes over a Noetherian affine base.
 \cref{lem:splinter_ascend_pi_!} in cases~\labelcref{item:etale} and~\labelcref{item:reducedfibers} is key to the proof of Theorems~\labelcref{thm:intro_d-etale,thm:intro_d-torsor}, 
 while \cref{lem:splinter_ascend_pi_!} in case~\labelcref{item:sameH0} (whose proof is much more elementary than in cases~\labelcref{item:etale} and~\labelcref{item:reducedfibers}) is key to the proofs of Theorems~\labelcref{item:thm:triv_fund_grp,item:thm:negative_kod_dim,item:thm:global-1-forms}.
 Recall from \cref{ss:trace} that, for a finite surjective morphism $\pi\colon Y\to X$ of integral normal Noetherian schemes, we denote $\mathrm{tr}_{Y/X}$ the classical trace.
 Recall from \cref{lem:quasi-etale} that if $\pi$ is quasi-\'etale, then 	$\pi^!\mcO_X = \mcO_Y$ in $\Coh(Y)$
 (i.e., $\mathrm{tr}_{Y/X}$ freely generates $\pi^!\mcO_X \coloneqq \sheafhom_{\mcO_X}(\pi_*\mcO_Y, \mcO_X)$ as an $\mcO_Y$-module), and that the converse implication holds if $X$ is a splinter (since splinters are Cohen--Macaulay).

	\begin{lemma}[Lifting the splinter property]
		\label{lem:splinter_ascend_pi_!}
	Let $\pi\colon Y \to X$ be a finite surjective morphism of Noetherian schemes.
Assume either of the following conditions\,:
	\begin{enumerate}[label=\rm{(\alph*)}]
		\item \label{item:etale} $X$ and $Y$ are integral normal and Nagata, and $\pi^!\mcO_X  = \mcO_Y$ in $\Coh(Y)$, e.g., $\pi$ is quasi-\'etale.
		\item \label{item:sameH0} $\pi^!\mcO_X \cong \mcO_Y$ in $\Coh(Y)$, and either  $H^0(Y, \mcO_Y)$ is a field or $H^0(X,\mcO_X) = H^0(Y,\mcO_Y)$.
		\item \label{item:reducedfibers} $\pi^!\mcO_X \cong \mcO_Y$ in $\Coh(Y)$, $X$ is
		 Nagata, $Y$ is normal,
		  and $\pi^\sharp \colon H^0(X, \mcO_X) \to H^0(Y , \mcO_Y)$ has reduced closed fibers.  
	\end{enumerate}
	If $X$ is a splinter, then $Y$ is a splinter.
\end{lemma}

		\begin{proof}
	Let $f\colon Z \to Y$ be a finite surjective morphism.		
	The splitting of $f^\sharp \colon \mcO_Y \to f_*\mcO_Z$ is equivalent to the surjectivity of $$\Hom_{\mcO_Y}(f_*{\mcO_Z}, \mcO_Y) \xrightarrow{-\circ f^\sharp} \Hom_{\mcO_Y}(\mcO_Y, \mcO_Y).$$
	The adjunction $\pi_* \dashv \pi^!$ provides the following commutative diagram
	\begin{equation}\label{eq:diagram}
\begin{tikzcd}
\Hom_{\mcO_Y}(f_*{\mcO_Z}, \mcO_Y)  \dar{\cong} \arrow{rr}{-\circ f^\sharp} && \Hom_{\mcO_Y}(\mcO_Y, \mcO_Y) \dar{\cong} \arrow[bend left=75]{ddd}{\mathrm{Tr}_{Y/X}}\\
\Hom_{\mcO_Y}(f_*{\mcO_Z}, \pi^! \mcO_X)  \dar{=} \arrow{rr}{-\circ f^\sharp} && \Hom_{\mcO_Y}(\mcO_Y, \pi^! \mcO_X) \dar{=} \\
\Hom_{\mcO_X}(\pi_* f_*{\mcO_Z},\mcO_X)  \dar{-\circ (\pi\circ f)^\sharp} \arrow{rr}{-\circ \pi_* f^\sharp} && \Hom_{\mcO_X}(\pi_* \mcO_Y, \mcO_X) \dar{- \circ \pi^\sharp} \\
\Hom_{\mcO_X}(\mcO_X, \mcO_X) \arrow[equal]{rr} && \Hom_{\mcO_X}(\mcO_X,\mcO_X). 
\end{tikzcd}
			\end{equation}
Here, $\mathrm{Tr}_{Y/X}$ is simply defined as the composite of the right vertical arrows.
			In case \labelcref{item:etale}, we choose the isomorphism $\pi^!\mcO_X \cong \mcO_Y$ to be the one induced by the classical trace map  $\mathrm{tr}_{Y/X}$\,; 
			under the identifications $\Hom_{\mcO_Y}(\mcO_Y, \mcO_Y)=H^0(Y, \mcO_Y)$ and $\Hom_{\mcO_X}(\mcO_X, \mcO_X)=H^0(X, \mcO_X)$, the map $\mathrm{Tr}_{Y/X}$ agrees with $\mathrm{tr}_{Y/X}$.
			Since $X$ is a splinter, the bottom-left vertical arrow $-\circ (\pi \circ f)^\sharp$ is surjective.
			This implies on the one hand that  $-\circ f^\sharp$ is nonzero, and  on the other hand that $-\circ \pi^\sharp$ is surjective.	
			This is already sufficient to establish \cref{lem:splinter_ascend_pi_!} under condition~\labelcref{item:sameH0}.
			Assuming that $H^0(Y, \mcO_Y)$ is a field,
			that $-\circ f^\sharp$ is nonzero yields that $-\circ f^\sharp$ is surjective since it  is a map of $H^0(Y, \mcO_Y)$-modules.
			Assuming that  $H^0(X,\mcO_X) = H^0(Y,\mcO_Y)$, that $-\circ \pi^\sharp$ is surjective yields that  the composition of the right vertical arrows, which is $H^0(X,\mcO_X)$-linear, is bijective and hence that $-\circ f^\sharp$ is surjective.

			We now prove \cref{lem:splinter_ascend_pi_!} under conditions~\labelcref{item:etale} or~\labelcref{item:reducedfibers}. The proof works uniformly for both cases, but we refer to \cref{rmk:etale} for a direct proof in case $\pi$ is \'etale. 
			We write $R =H^0(X, \mcO_X)$ and $S = H^0(Y, \mcO_Y)$. 
			Recall that the map $\coprod \Spec R_z^h \to \Spec R$, where the coproduct runs through all closed points~$z$ of $\Spec R$ and where $R_z^h$ is the henselization of~$R_z$, is faithfully flat \cite[\href{https://stacks.math.columbia.edu/tag/07QM}{Tag 07QM}]{stacks-project}.
			We thus have to prove that for all closed points of $\Spec R$ the map $- \circ f^\sharp$ is surjective after base change to $ R_z^h $.
			Write $\pi^{h} \colon Y^h \to X^h$ and $f^h \colon Z^h \to Y^h$ for the pullbacks of $\pi$ and~$f$ along $R  \to R_z^h$, respectively.
			It is a priori not clear that $R$ is Noetherian\,;
			 however,
			 since $\Spec R_z^h \to \Spec R$ is affine, $X^h$ and $Y^h$ are quasi-compact and quasi-separated.
			By base change for the exceptional inverse image
			 \cite[Thm.~25.31]{goertz_wedhorn_ii}, 
			we have 
			$h^* \pi^! \mcO_X = (\pi^h)^\times \mcO_{X^h}$ in $\sfD_{\mathrm{QCoh}} (\mcO_{Y^h})$, where $h \colon Y^h \to Y$ is the natural map and $(\pi^h)^\times$ is the right adjoint to $\mathrm{R} \pi^h_* \colon \sfD_{\mathrm{QCoh}}^+ (\mcO_Y) \to \sfD_{\mathrm{QCoh}}^+ (\mcO_X)$. 
			Since $h$ is flat, we have $\mcH^0(h^*\pi^! \mcO_X) =h^* \mcH^0( \pi^! \mcO_X)$, and so $\mcH^0((\pi^{h})^\times \mcO_{X^{h}} ) \cong \mcO_{Y^{h}}$.
			As outlined in \cref{sec:upper_shriek_finite_morphism}, $(\pi^{h})^\times$ is the right derived functor of $\sheafhom_{\mcO_{X^h}}(\pi^h_* \mcO_{Y^h}, -)$ 
			and we thus have $\sheafhom_{\mcO_{X^h}}(\pi^h_* \mcO_{Y^h}, \mcO_{X^h}) \cong \mcO_{Y^h}$ as $\mcO_{Y^h}$-modules.
			On the other hand, for sheaves $\mcF, \mcG \in \Coh(X)$ we have $\Hom_{\mcO_X}(\mcF, \mcG) = H^0(X, \mathrm{R} \sheafhom_{\mcO_X}( \mcF, \mcG))$ \cite[\href{https://stacks.math.columbia.edu/tag/08DK}{Tag 08DK}]{stacks-project}.
			Therefore, by flat base change of quasi-coherent cohomology and the compatibility of $\mathrm{R}\sheafhom(-, -)$ with flat pullback \cite[Prop.~22.70]{goertz_wedhorn_ii}, 
			we obtain the same diagram as~\labelcref{eq:diagram} with $X$ replaced by $X^h$, $Y$ by $Y^h$,  $Z$ by $Z^h$, $\pi^!$ by $(\pi^h)^\times$, and where $-\circ (\pi^h\circ f^h)^\sharp$ is surjective. 
		 We are thus reduced to showing that $- \circ (f^h)^\sharp$ is surjective for every closed point $z \in \Spec R$.

		Since a splinter is normal, we have
			in both cases \labelcref{item:etale} and~\labelcref{item:reducedfibers} that $X$ and~$Y$ are normal. 
				Working componentwise, we may and do assume that $X$ and $Y$ are both integral.
			In particular, as discussed in \cref{ss:trace}, both $R$ and $S$ are integrally closed domains in $K(X)$ and $K(Y)$, respectively.
            Normality of rings is preserved under smooth base change~\cite[\href{https://stacks.math.columbia.edu/tag/033C}{Tag 033C}]{stacks-project} and under filtered colimits~\cite[\href{https://stacks.math.columbia.edu/tag/037D}{Tag 037D}]{stacks-project}, and hence $X^h$ and $Y^h$ are normal.
			Since $Y^h$ is quasi-compact, $Y^h$ has only finitely many connected components.
            Let $Y^h = Y_1 \amalg \dots \amalg Y_r$ be the decomposition into connected components and let $S\otimes_R R_z^h = H^0(Y^h,\mcO_{Y^h}) = S_1 \times \dots \times S_r$, where $S_i = H^0(Y_i, \mcO_{Y_i})$.
	        By \cref{lem:integral_on_global_sections}, $R \hookrightarrow S$ is integral, hence each $S_i$ is integral over $R_z^h$.

            Recall that $R_z^h$ is the filtered colimit of the pointed \'etale ring maps $(R_z, z) \to (T,\mathfrak{q})$ 
              such that $\kappa(z)= \kappa(\mathfrak{q})$.
            Each projector $S\otimes_R R_z^h \to S_i$ is defined by an idempotent in $S \otimes_R R_z^h$. 
            Thus, there exists a pointed \'etale ring map $(R_z, z) \to (T,\mathfrak{q})$ such that each projector $S\otimes_R R_z^h \to S_i$ is already defined in $ S \otimes_R T$,
            i.e., $ S \otimes_R T=  S_1' \times \dots \times S_r'$ with $S_i = S_i' \otimes_T R_z^h$.
          We then have $Y_T \coloneqq Y\times_{\Spec R} \Spec T = Y_1' \amalg \dots \amalg Y_r'$  with $Y_i = Y_i' \times _{\Spec T} S_i$.
           	Note that $Y_T$ is normal Noetherian so that each $Y_i'$ is an integral normal Noetherian scheme.
           	Using that for any normal $R$-algebra~$A$ and any \'etale ring morphism $T \to T'$ the base change map $A\otimes_R T \to A \otimes_R T'$ is injective (by openness of flat morphisms locally of finite presentation \cite[\href{https://stacks.math.columbia.edu/tag/01UA}{Tag 01UA}]{stacks-project}), and that a filtered colimit of integral domains is an integral domain, we see that any affine open cover of $Y'_i$ yields an affine open cover of $Y_i$ by integral domains and hence that $Y_i$ is integral.

			In particular, the rings $S_i$ are normal domains. 
            We claim that each $S_i$ is a local domain. Indeed, $R_z^h \to S_i$ is an integral ring map and $S_i$ contains no nontrivial idempotent elements. 
            Since $R_z^h$ is henselian, any finite $R_z^h$-subalgebra of $S_i$ is a finite product of local rings \cite[\href{https://stacks.math.columbia.edu/tag/04GG}{Tag 04GG (10)}]{stacks-project}.
            Since any such subalgebra contains no nontrivial idempotents it is a local ring. Hence, $S_i$ is a filtered colimit of local rings and because $S_i$ is integral over $R_z^h$, each morphism in the directed system is a morphism of local rings and so $S_i$ is a local ring.

			Moving forward, since we are assuming $X$ to be Nagata,  after possibly normalizing $Z$ and working on each component separately, we may and do assume that $Z$ is integral and normal.
			By further taking the normalization of $Z$ in a normal closure of $K(Z)$ over $K(X)$ we may and do assume that $K(Z)/ K(X)$ is normal.  
			We summarize the situation in the commutative diagram
			\[
			\begin{tikzcd}
				Z_1 \amalg \cdots \amalg Z_r    \arrow{rr}{f_1 \amalg \dots \amalg f_r} \arrow[swap]{drr}{g_1 \amalg \dots \amalg g_r} &&
				Y_1 \amalg \dots \amalg Y_r \dar{\pi_1 \amalg \dots \amalg \pi_r} \arrow{r} & \Spec S_1 \amalg \dots \amalg \Spec S_r \dar \\
				&  & X^{h} \arrow{r} & \Spec R_z^h 
			\end{tikzcd}
			\]
			where $f_i\colon Z_i \to Y_i$ is the base change of $Z^{h}$ along $S \otimes_R R_z^h \to S_i$ and $g_i = \pi_i \circ f_i \colon Z_i \to X^{h}$.
			
			\begin{claim} 
				The map $\mcO_{X^h} \to (g_i)_* \mcO_{Z_i}$ splits for every $1 \leq i \leq r$.
			\end{claim}

			\begin{proof}[Proof of the Claim]
				Let $G$ be the Galois group of $K(Z)$ over $K(X)$ and let $T \coloneqq H^0(Z, \mcO_Z)$ be the integral closure of $R$ in $K(Z)$.
				The field extension $K(Z)^G / K(X)$ is purely inseparable and 
				the ring homomorphism $R \to T^G$ to the invariant subring $T^G$, which is the integral closure of~$R$ in~$K(Z)^G$, is
				a homeomorphism on spectra.
				By a straightforward variant of
				\cite[\href{https://stacks.math.columbia.edu/tag/0BRK}{Tag 0BRK}]{stacks-project} for  domains that are integrally closed in a field, 
				$G$  acts transitively
				on the maximal ideals $\mathfrak{p}_1, \ldots, \mathfrak{p}_s$ of $T$ lying over $z$.
				As discussed previously in the case of the integral extension $R \to S$, $T \otimes_R R_z^h$ splits into a product of local rings $T_1 \times \dots \times T_s$.
				Let $Z_j'$ be the connected component of $Z^h$ with $H^0(Z_j', \mcO_{Z_j'}) = T_j$.
				The induced action of $G$ on $Z^h$ then permutes the morphisms $g_j' \colon Z_j' \to X^h$ and so induces isomorphisms $(g_j')_*\mcO_{Z_j'} \cong (g_k')_*\mcO_{Z_k'}$ of $\mcO_{X^{h}}$-algebras.
				Since $X$ is a splinter, the pullback $\mcO_X \to \pi_* f_* \mcO_Z$ splits. By base change, such a splitting induces a splitting
				$$\mcO_{X^{h}} \to \bigoplus_{j=1}^r (g_j')_*\mcO_{Z_j'} \xrightarrow{\sum_j s_j} \mcO_{X^{h}}.$$
				Since $\sum_j s_j$ sends the constant section $1_{X^h} \mapsto 1_{X^h}\in {R_z^h}$, there exists $j_0$ such that $s_{j_0}(1_{X^h}) \notin \mathfrak{m}_z^h$, where $\mathfrak{m}_z^h$ denotes the maximal ideal of $R_z^h$.
				Thus, $\mcO_{X^{h}} \to (g_{j_0}')_*\mcO_{Z_{j_0}'}$ splits.
				Using the isomorphisms $(g_j')_*\mcO_{Z_j'} \cong (g_k')_*\mcO_{Z_k'}$, we see that every $\mcO_{X^{h}} \to (g_j')_*\mcO_{Z_j'}$ splits.
				The claim then follows, as each $Z_i$ is a finite disjoint union of some connected components $Z_j'$.
			\end{proof}

            Now the map $-\circ (f^h)^\sharp$ decomposes as the direct sum\,:
			$$ - \circ \oplus_i f_i^\sharp\colon \bigoplus_{i=1}^r \Hom_{\mcO_{Y_i}}((f_i)_* \mcO_{Z_i}, \mcO_{Y_i}) = \Hom_{\mcO_{Y^{h}}} \left(f^{h}_* \mcO_{Z^{h}}, \mcO_{Y^{h}}\right) \to \Hom_{\mcO_{Y^{h}}}\left(\mcO_{Y^{h}},\mcO_{Y^{h}}\right)= \bigoplus_{i=1}^r S_i. $$
			By the claim above, $\mcO_{X^{h}} \to (g_i)_*\mcO_{Z_i}$ splits for every $i$.
			Thus, the left-vertical arrow in~\labelcref{eq:diagram} base changed to $R_z^h$ remains surjective after restricting to a component $Y_i$ of $Y^{h}$.
			Hence, it is enough to show that any morphism of rings $\varphi_i\colon T_i \to S_i$ such that the composite
			$\mathrm{Tr}_{Y_i/X^{h}} \circ \varphi_i \colon T_i \to R_z^h$ is surjective is also surjective. Denote $\mathfrak{m} \subseteq R$ and $\mathfrak{n_i} \subseteq S_i$ the maximal ideals.
			
            Under condition~\labelcref{item:reducedfibers},
			we still have that $R_z^h \to S_i$ has reduced closed fibers, and hence $\mathfrak{m}S_i = \mathfrak{n}_i$. 
			By $R_z^h$-linearity of
			the map $\mathrm{Tr}_{Y_i/X^{h}}$, we get $\mathrm{Tr}_{Y_i/X^{h}}(\mathfrak{n}_i) \subseteq \mathfrak{m}$.
			Thus, any element $s \in S_i$ such that $\mathrm{Tr}_{Y_i/X^{h}}(s)=1$ is a unit in $S_i$. Since $\varphi_i$ is $S_i$-linear, we conclude that $\varphi_i$ is surjective.

Assume now condition~\labelcref{item:etale}.
 Note that the restriction of the classical trace for $Y/X$ is the classical trace for $Y_i/X^{h}$\,; indeed,  they both agree on the base change along $R \to R^h_z$ of the flat locus $U\subseteq X$ of $\pi$ and then use that $X^{h}$ is integral.
 Note that for any integral extension of local rings $\varphi\colon (A,\mathfrak a) \hookrightarrow (B,\mathfrak b)$ we have $\mathfrak b = \sqrt{\mathfrak{a}B}$\,; this is indeed clear if $\varphi$ is finite and the general case follows by a limit argument.
By \cref{lem:speyer} and the discussion in \cref{ss:trace}, the trace $\mathrm{tr}_{Y_i/X^h}\colon S_i \to R^h_z$ sends $\mathfrak{n}_i = \sqrt{\mathfrak{m}S}$ into $\mathfrak{m}$. 
We can then conclude as before.
\end{proof}

\begin{remark}
	There are finite \'etale maps $\pi\colon Y \to X$ such that $\pi^\sharp \colon H^0(X, \mcO_X) \to H^0(Y , \mcO_Y)$ admits non-reduced closed fibers.
	Consider for instance a smooth projective surface $S$ over a field $k$ admitting a free action by a finite cyclic group $\bbZ/n\bbZ$ with $n>1$ invertible in $k$, and consider the action of $\bbZ/n\bbZ$ on $\bbA^1_k$ given by multiplication by the powers of a primitive $n$-th root of unity. 
	Then the diagonal action of $\bbZ/n\bbZ$ on $Y \coloneqq S\times_k \bbA^1_k$ is free and the quotient map $\pi$ is \'etale, but the induced map $\pi^\sharp$ on global sections is $k[T] \to k[T], \ T \mapsto T^n$, which has a non-reduced fiber over $0$. 
\end{remark}

\begin{remark}
	The proof of \cref{lem:splinter_ascend_pi_!}\labelcref{item:etale} and~\labelcref{item:reducedfibers} simplifies if $X$ is proper over a Noetherian ring as in this case (1) $R^h$ is Noetherian, so that $X^h$ is Noetherian,  and (2) $R\hookrightarrow S$ is finite, so that  $S^h= S\otimes_R R^h$ splits as a finite product of local domains by \cite[\href{https://stacks.math.columbia.edu/tag/04GG}{Tag 04GG (10)}]{stacks-project}.
\end{remark}

\begin{remark}
	Let $R \hookrightarrow S$ be a finite extension of normal Gorenstein Nagata Noetherian domains with reduced closed fibers.  If $R$ is a splinter, then $S$ is a splinter. 
	Indeed, we have to show that for any finite ring extension $S \hookrightarrow T$ the evaluation at 1 map $\Hom_S(T, S) \to S$ is surjective, or equivalently, that its base change to $R^h_{\mathfrak m}$ is surjective for any maximal ideal $\mathfrak m$ of $R$. Since $R$ is Nagata, we can assume that $T$
	is integrally closed.
	Writing $S\otimes_R R^h_{\mathfrak m} = S_1 \times \cdots \times S_r$, with each $S_i$ a henselian domain, 
	we are reduced to showing that each evaluation at 1 map $\Hom_{S_i}(T\otimes_SS_i, S_i) \to S_i$ is surjective. 
    Since a Noetherian local ring is Gorenstein if and only if its henselization is, and since
	the Picard group of a local ring is trivial, we have 
	 $\pi_i^! R^h_{\mathfrak m} \cong S_i$, where $\pi_i\colon R^h_{\mathfrak m} \hookrightarrow S_i$ is the natural finite ring extension. 
	By the Claim in the proof of \cref{lem:splinter_ascend_pi_!}, the composition 
	$\Hom_{S_i}(T\otimes_SS_i, S_i) \to S_i \xrightarrow{\mathrm{Tr}_i} R^h_{\mathfrak m}$ is surjective for all $i$, where  $\mathrm{Tr}_i$ is the $R^h_{\mathfrak m}$-linear map induced by the isomorphism $\pi_i^! R^h_{\mathfrak m} \cong S_i$. Since
	 henselization of local rings preserves the closed fiber, we have
	  $\mathfrak m S_i = \mathfrak n_i$, where $\mathfrak n_i$ is the maximal ideal of $S_i$, and we can conclude as in the proof of
	 \cref{lem:splinter_ascend_pi_!}\labelcref{item:reducedfibers}. 
\end{remark}

\begin{remark}[Lifting the splinter property along finite \'etale covers] \label{rmk:etale}
	As mentioned to us by Karl Schwede, \cref{lem:splinter_ascend_pi_!} admits the following  simpler proof in case $\pi \colon Y \to X$ is a finite \'etale morphism of Nagata Noetherian schemes. 
    Assume that $X$ is a splinter and let $f\colon Z \to Y$ be a finite surjective morphism. 
    To show that $f^\sharp \colon \mcO_Y \to f_*\mcO_Z$ splits, it suffices to show that $(f\circ f')^\sharp$ splits for a further finite cover $f'\colon Z' \to Z$. 
    By dominating $Z$ with the disjoint union of the reductions of its irreducible components, we may assume that $Z$ is integral.
    By further replacing $Z$ with its normalization in the normal closure of $K(Z)$
    over $K(X)$, which remains finite over $X$ since $X$ is assumed Nagata, we may and do assume that $Z$ is normal and that $K(Z)/K(X)$ is a finite normal extension of fields.
In that situation,  the base change $g_Y \colon Z\times_X Y \to Y$ of $g = \pi \circ f \colon Z  \to X$
 identifies with the morphism $\coprod_\sigma g_\sigma \colon \coprod_\sigma Z_\sigma \to Y$, 
    where the coproduct runs through all embeddings $\sigma\colon  K(Y) \hookrightarrow K(Z)$ that are the identity on $K(X)$ and
	where $g_\sigma$ is the integral closure of $Y$ with respect to the embedding~$\sigma$.
    A splitting of $\mcO_X \to g_*\mcO_Z$ provides by base change a splitting of $\mcO_Y \to (g_Y)_*\mcO_W$. 
    If $H^0(Y,\mcO_Y)$ is a local ring, then arguing as in the proof of the claim there exists~$\sigma$ such that $\mcO_Y \to (g_\sigma)_* \mcO_{Z_\sigma}$ splits, and using the action of $\mathrm{Gal}(K(Z)/K(X))$, which permutes the $g_\sigma$'s transitively,
    we find that $\mcO_Y \to (g_\sigma)_* \mcO_{Z_\sigma}$ splits for all~$\sigma$.
    In general,
    we can reduce to this case by using that $\mcO_Y \to (g_\sigma)_* \mcO_{Z_\sigma}$ splits if and only if its base change to the localizations of all closed points of $\Spec H^0(Y, \mcO_Y)$ splits\,; see the proof of \cref{lem:bhatt_et_al_enough_to_check_at_closed_points} below.
    In particular, since $g_{\id} = f \colon Z_{\id} = Z \to Y$, we find that $f^\sharp$ splits, as desired.
\end{remark}

\begin{remark}[Lifting the $F$-split property] \label{rem:F-split_ascend}
	Let $\pi\colon Y \to X$ be a finite surjective morphism of Noetherian schemes over $\bbF_p$. Assume either that $\pi$ is \'etale with $X$ normal, or that condition \labelcref{item:sameH0} is fulfilled with the additional assumption that $\mcO_X \to \pi_*\mcO_Y$ splits. 
If $X$ is $F$-split, then $Y$ is $F$-split. 

In the \'etale case, this is \cite[Lem.~11.1]{patakfalvi2020beauvillebogomolov}.
In the other case, we argue as in the proof of \cref{lem:splinter_ascend_pi_!}\labelcref{item:sameH0} with $f\colon Z \to Y$ replaced by the Frobenius $F\colon Y \to Y$. 
If $s\colon \pi_*\mcO_Y \to \mcO_X$ is a splitting of the map $\mcO_X \to \pi_*\mcO_Y$, then, with a Frobenius splitting of $X$, we obtain a diagram
$$\mcO_X \to \pi_*F_*\mcO_Y = F_*\pi_* \mcO_Y \xrightarrow{F_*(s)} F_* \mcO_X \to \mcO_X,$$
where the composition is the identity. This proves that the bottom-left vertical arrow $-\circ (\pi \circ F)^\sharp$ of \labelcref{eq:diagram} is surjective. 
As in the splinter case, we deduce that $-\circ F^\sharp$ is surjective, i.e., that $Y$ is $F$-split.
\end{remark}

\begin{lemma}[Lifting the derived splinter property]
		\label{lem:Dsplinter_ascend_pi_!} 
	Let $\pi\colon Y \to X$ be a proper surjective morphism
	 of Noetherian schemes.
Assume  $\pi^!\mcO_X \cong \mcO_Y$ in  $\sfD_{\mathrm{Coh}} (\mcO_Y)$, and either  $H^0(Y, \mcO_Y)$ is a field or $H^0(X,\mcO_X) = H^0(Y,\mcO_Y)$.
If $X$ is a derived splinter, then $Y$ is a derived splinter.
\end{lemma}
\begin{proof} 
This is proved as in \cref{lem:splinter_ascend_pi_!}\labelcref{item:sameH0}
by using 
the adjunction $\mathrm{R}\pi_* \dashv \pi^!$ for a proper morphism $\pi\colon Y \to X$ instead of the adjunction 
$\pi_* \dashv \pi^!$ for a finite morphism $\pi\colon Y \to X$.
\end{proof}

A morphism of schemes $\pi\colon Y \to X$ is said to be a \emph{crepant morphism} if it is proper, birational, and is such that $\pi^!\mcO_X \cong \mcO_Y$. Note that, by \cref{lem:k_and_o_equvalence},
 if $X$ is assumed to be Gorenstein, the latter condition is equivalent to $\pi^*\omega_X = \omega_Y$.

\begin{proposition}[The splinter property in positive characteristic lifts along crepant morphisms]
	\label{rmk:crepant}
Let $\pi \colon Y \to X$ be a crepant morphism of Noetherian schemes.
 \begin{enumerate}
 	\item  \label{item:crepant-splinter}  
 	If $\pi \colon Y \to X$ is defined over~$\bbF_p$
 	 and
 	if $X$ is a splinter, then $Y$ is a splinter. 
\item \label{item:crepant-Dsplinter} 
If $X$ is a derived splinter, then $Y$ is a derived splinter.
\item \label{item:crepant-Fsplit} 
If $X$ is $F$-split and $Y$ is integral, then $Y$ is $F$-split.
 \end{enumerate}
\end{proposition}
\begin{proof}
   If $X$ is a derived splinter, then it is normal so that $\pi_*\mcO_Y = \mcO_X$ 
   and hence $\pi^\sharp  \colon H^0(X, \mcO_X) \to H^0(Y , \mcO_Y)$ is an isomorphism. 
   We can then apply \cref{lem:Dsplinter_ascend_pi_!}.
	The statement \labelcref{item:crepant-splinter}  follows from the fact \cite[Thm.~1.4]{bhatt_derived_splinters_in_positive_characteristic} that splinters in positive characteristic agree with derived splinters.
	
	For \labelcref{item:crepant-Fsplit}, choose a splitting $\mcO_X \to F_*\mcO_X \xrightarrow{s} \mcO_X$.
	Let $\tilde{s}$ be the image of $s$ under 
	\begin{align*}
		\Hom_{\mcO_X}(F_* \mcO_X, \mcO_X) = &\Hom_{\mcO_X}(\mcO_X,F^! \mcO_X) \to \Hom_{\mcO_Y}(\pi^!\mcO_X, \pi^!F^!\mcO_X) \\
		= &\Hom_{\mcO_Y}(\pi^!\mcO_X, F^! \pi^!\mcO_X) \cong  \Hom_{\mcO_Y}(\mcO_Y,F^! \mcO_Y) =  \Hom_{\mcO_Y}(F_*\mcO_Y, \mcO_Y).
	\end{align*}
	Then $\mcO_Y \to F_* \mcO_Y \xrightarrow{\tilde{s}} \mcO_Y$ is the identity. 
	Indeed, since $Y$ is integral, this can be verified after restricting to the dense open subset where $\pi$ is an isomorphism so that the claim follows from the fact that $s$ is a splitting of $\mcO_X \to F_* \mcO_X$.
\end{proof}

\begin{remark} 
First, \cref{rmk:crepant}\labelcref{item:crepant-Dsplinter} should be compared to \cref{cor:splinters-O-eq} below.
Second, \cref{rmk:crepant}\labelcref{item:crepant-Fsplit} notably removes the Gorenstein assumption in \cite[Lem.~1.3.13]{brion_kumar_frobenius_spliting_methods_in_gemoetry_and_representation_theory}
 where it is shown that if $\pi\colon Y \to X$ is a crepant morphism of normal $F$-finite schemes over $\bbF_p$ with $X$ Gorenstein and $F$-split, then $Y$ is $F$-split.
Third, regarding \cref{rmk:crepant}\labelcref{item:crepant-splinter}, we have the following related result.
Let $\pi\colon Y \to X$ be a birational morphism of integral normal schemes proper over a complete Noetherian local domain  with positive characteristic residue field. 
Assume that $X$ is $\bbQ$-Gorenstein and that $\pi^*K_X = K_Y$ in $\mathrm{Pic}(Y)\otimes \bbQ$. 
It is shown in \cite[Lem.~4.19]{bhatt_et_al_globally_+_regular_varieties_and_mmp_for_threefolds_in_mixed_char} that
 if $X$ is a splinter, then $Y$ is a splinter.
\end{remark}

	For the sake of completeness, we mention that \cref{lem:splinter_ascend_pi_!} also holds for globally $+$-regular pairs\,:
	
\begin{lemma}[Lifting global $+$-regularity]
	\label{lem:lifting_glob_+_reg}
		Let $\pi\colon Y \to X$ be a finite surjective morphism of normal excellent
		Noetherian schemes and let $\Delta$ be an effective $\bbQ$-Weil divisor on $X$.
	 Assume either of the following conditions\,:
		\begin{enumerate}[label=\rm{(\alph*)}]
		\item \label{item:globplus_quasietale}  $\pi$ is quasi-\'etale.
		\item  $\pi^!\mcO_X \cong \mcO_Y$ in $\Coh(Y)$,  and $\pi^\sharp \colon H^0(X, \mcO_X) \to H^0(Y , \mcO_Y)$ has reduced closed fibers.  
	\end{enumerate}
	 If $(X,\Delta)$ is globally $+$-regular, then $(Y, \pi^*\Delta)$ is globally $+$-regular.
\end{lemma}
	\begin{proof}
		Let $f\colon Z \to Y$ be a finite surjective morphism with $Z$ normal.
		We have to show that the map $f^\sharp_{\pi^*\Delta} \colon \mcO_Y \to f_*\mcO_Z(\lfloor f^* \pi^*\Delta \rfloor)$ splits, or equivalently, that
		$$\Hom_{\mcO_Y}(f_*{\mcO_Z}(\lfloor f^* \pi^*\Delta \rfloor), \mcO_Y) \xrightarrow{-\circ f^\sharp_{\pi^*\Delta}} \Hom_{\mcO_Y}(\mcO_Y, \mcO_Y)$$
		is surjective.
		Note that the map $(\pi \circ f)^\sharp_\Delta\colon \mcO_X \to \pi_* f_*\mcO_Z(\lfloor f^* \pi^*\Delta \rfloor)$ factors through $\mcO_X \to \pi_*\mcO_Y$.
		Hence, the following diagram commutes\,:
		\[
		\begin{tikzcd}
			\Hom_{\mcO_Y}(f_*{\mcO_Z}(\lfloor f^* \pi^*\Delta \rfloor), \mcO_Y)  \dar{\cong} \arrow{rr}{-\circ f^\sharp_{\pi^*\Delta}} && \Hom_{\mcO_Y}(\mcO_Y, \mcO_Y) \dar{\cong} \\
			\Hom_{\mcO_Y}(f_*{\mcO_Z}(\lfloor f^* \pi^*\Delta \rfloor), \pi^! \mcO_X)  \dar{=} \arrow{rr}{-\circ f^\sharp_{\pi^*\Delta}} && \Hom_{\mcO_Y}(\mcO_Y, \pi^! \mcO_X) \dar{=} \\
			\Hom_{\mcO_X}(\pi_* f_*{\mcO_Z}(\lfloor f^* \pi^*\Delta \rfloor),\mcO_X)  \dar{-\circ (\pi \circ f)^\sharp_\Delta} \arrow{rr}{-\circ \pi_* f^\sharp_{\pi^*\Delta}} && \Hom_{\mcO_X}(\pi_* \mcO_Y, \mcO_X) \dar{- \circ \pi^\sharp} \\
			\Hom_{\mcO_X}(\mcO_X, \mcO_X) \arrow[equal]{rr} && \Hom_{\mcO_X}(\mcO_X,\mcO_X). 
		\end{tikzcd}
		\]
		We then proceed as in the proof of \cref{lem:splinter_ascend_pi_!}.
	\end{proof}

	\subsection{The splinter property and base change of field} \label{SS:splinters-base change}
The following proposition establishes the invariance of the splinter property for proper schemes over a field under algebraic base change. 
We refer to \cref{rmk:bc-splinter} for base change along any field extension.

		\begin{proposition}\label{lem:splinter_base change}
			Let $X$ be a
			  scheme such that $H^0(X, \mcO_X)$ is a field, e.g., $X$ is proper, connected, and reduced over a field $k$.
			Then, for any algebraic field extension~$L$ of $H^0(X, \mcO_X)$, the scheme $X_L \coloneqq  X\times_{H^0(X,\mcO_X)} L$ is a splinter if and only if $X$ is a splinter.
		\end{proposition}
		\begin{proof}
			The ``only if'' part of the proposition follows from \cref{lem:splinter_descent}. 
			Let $K\coloneqq H^0(X, \mcO_X)$.
			First assume $L$ is a finite extension of $K$ and let $h \colon \Spec L \to \Spec K$. Then 
			$$h^!\mcO_{\Spec K} =\Hom_K(L, K) \cong L = \mcO_{\Spec L}$$
			as $L$-vector spaces.
			By base change for the exceptional inverse image \cite[\href{https://stacks.math.columbia.edu/tag/0E9U}{Tag 0E9U}]{stacks-project}, there exists an isomorphism $\pi^!\mcO_X \cong \mcO_{X_L}$, where $\pi\colon X_L \to X$ is the base change of $h$ along $X \to \Spec K$.
			We conclude by \cref{lem:splinter_ascend_pi_!} that $X_L$ is a splinter, since 
			$$H^0(X_L, \mcO_{X_L})=H^0(X, \mcO_X) \times_K L=L$$
			by flat base change.
			Now assume $K \to L$ is any algebraic field extension and take a finite cover $f \colon Y \to X_L$. Since $f$ is defined by finitely many equations and $L$ is algebraic over $K$, we can find a finite cover $f' \colon Y' \to X_{L'}$ such that $K \subseteq L' \subseteq L$ is an intermediate extension, finite over~$K$, 
			and $f=f'\times_{L'}L$ is the base change of $f'$. By the previous argument $X_{L'}$ is a splinter, thus $\mcO_{X_{L'}} \to f'_* \mcO_{Y'}$ admits a section $s\colon f'_*\mcO_{Y'} \to \mcO_{X_{L'}}$. Now pulling back to $X_L$ and using flat base change, we obtain the desired section of $\mcO_{X_L} \to f_*\mcO_Y$.
		\end{proof}
	
	\begin{corollary}\label{cor:splinter_bc}
			Let $X$ be a connected proper scheme over a field $k$ such that $H^0(X,\mcO_X)$ is a separable extension of $k$. 
			Then, for any algebraic field extension~$K$ of $k$,
			the scheme $X_K \coloneqq  X\times_k K$ is a splinter if and only if $X$ is a splinter.
	\end{corollary}
	\begin{proof}
			The ``only if'' part of the corollary follows from \cref{lem:splinter_descent}. 
		Let $\bar k$ be an algebraic closure of $k$. By the assumption on $H^0(X,\mcO_X)$, $X_{\bar k} \coloneqq X\times_k \bar k$ is isomorphic to the disjoint union of $\dim_k H^0(X,\mcO_X)$ copies of $X \times_{H^0(X,\mcO_X)} \bar k$. By the ``if'' part of \cref{lem:splinter_base change}, if $X$ is a splinter, then $X_{\bar k}$ is a splinter. By \cref{lem:splinter_descent}, we get that $X_K$ is a splinter.
	\end{proof}

		\begin{corollary}\label{cor:splinter_geom_normal}
			Let $X$ be a connected proper scheme over a field $k$. If $X$ is a splinter, then $H^0(X, \mcO_X)$ is a field, and $X$, considered as a scheme over $\Spec H^0(X, \mcO_X)$, is geometrically normal.
		\end{corollary}
		\begin{proof}
			By \cref{prop:bhatt_normal_CM}, a splinter is normal.  In particular,  $H^0(X, \mcO_X)$ is a field.
		The corollary then follows from \cref{lem:splinter_base change}.
		\end{proof}
		
	\begin{remark}[Algebraic base change for globally $+$-regular pairs]
		 Let $X$ be a connected normal proper scheme over a field $k$ such that $H^0(X, \mcO_X)$ is a field, and let $\Delta$ be an effective $\bbQ$-Weil divisor on $X$.
		Let $L$ be an algebraic field extension of $H^0(X, \mcO_X)$.
The arguments of the proof of \cref{lem:splinter_base change} show that
if	 $(X, \Delta)$ is globally $+$-regular, then $X_L  \coloneqq  X\times_{H^0(X,\mcO_X)} L$ is normal and  the pair $(X_L,  \Delta_L)$ is globally $+$-regular.

\noindent Assume in addition that $H^0(X,\mcO_X)$ is a separable extension of $k$, and let $K$ be any algebraic extension of $k$. As in \cref{cor:splinter_bc}, we have that if	$(X, \Delta)$ is globally $+$-regular, then $X_K \coloneqq  X\times_k K$ is normal and  the pair $(X_K,  \Delta_K)$ is globally $+$-regular.
	\end{remark}

	\section{Lifting and descending global $F$-regularity}
	\label{S:LD-globF}
	
	The first aim of this section is to show that the results of \cref{S:LD-splinter} regarding splinters, notably \cref{lem:splinter_ascend_pi_!}, extend to the globally $F$-regular setting\,; see \cref{lem:gFr_ascend_pi_!}.
	The second aim is to show how the criterion of Schwede--Smith (recalled in \cref{thm:schwede_smith_f_reg_one_divisor}) can be used to establish
	further results in the global $F$-regular setting\,; for instance,
	 we show in \cref{prop:product-gFr} that global $F$-regularity is stable under product and, combined with \cref{lem:gFr_ascend_pi_!}, use it to recover in \cref{lem:gFr_base change} a result of \cite{gongyo_li_patakfalvi_zsolt_schwede_tanaka_zong_on_rational_connectedness_of_globally_f_regular_threefolds} stating that global $F$-regularity for normal proper schemes over a field is stable under base change of $F$-finite fields. 
	Except for \cref{thm:schwede_smith_f_reg_one_divisor}, which is due to Schwede--Smith, the results of this section will not be used in the rest of the paper.

	\subsection{A criterion for global $F$-regularity}
	The following criterion of Schwede--Smith makes it possible in practice to reduce checking that a variety is globally $F$-regular to simply check that $\mcO_X \to F_*^e \mcO_X(D)$ splits for one specific Weil divisor $D$.
	
	\begin{theorem}[{\cite[Thm.~3.9]{schwede_smith_globally_f_regular_and_log_Fano_varieties}}]
		\label{thm:schwede_smith_f_reg_one_divisor} 
		Let $X$ be a normal variety over an $F$-finite field of positive characteristic. 
		Then $X$ is globally $F$-regular if and only if there exists an effective Weil divisor $D$ on $X$ such that
		\begin{enumerate}
			\item there exists an $e>0$ such that the natural map $\mcO_X \to F_*^e \mcO_X(D)$ splits, and
			\item the variety $X\setminus D$ is globally $F$-regular.
		\end{enumerate}
	\end{theorem}

	\begin{remark} \label{rmk:SS-criterion}
		Suppose $X$ is a normal projective variety over an $F$-finite field of positive characteristic. 
	If $D$ is an ample divisor on $X$, the variety $X\setminus D$ is affine and therefore globally $F$-regular if and only if its local rings are strongly $F$-regular. 
	Since regular local rings are strongly $F$-regular, a smooth projective variety $X$ over an $F$-finite field of positive characteristic is globally $F$-regular if and only if $\mcO_X \to F_*^e \mcO_X(D)$ splits for some ample divisor $D$.
	\end{remark}

	\subsection{Descending global $F$-regularity}
	The following \cref{lem:gFr_descent}, which is due to Schwede--Smith, holds in particular when the map $\mcO_X \to \pi_*\mcO_Y$ is an isomorphism, e.g., when $\pi\colon Y \to X$ is flat proper with geometrically connected and geometrically reduced fibers \cite[\href{https://stacks.math.columbia.edu/tag/0E0L}{Tag 0E0L}]{stacks-project}, or when $\pi \colon Y \to X$ is proper birational and $X$ is a normal variety.

	\begin{lemma}[{\cite[Cor.~6.4]{schwede_smith_globally_f_regular_and_log_Fano_varieties}}]
		\label{lem:gFr_descent}
		Let $\pi\colon Y\to X$ be a morphism of varieties over 
		a field
		 of positive characteristic. 
		If $Y$ is 
		normal 
		globally $F$-regular and if the map $\mcO_X \to \pi_*\mcO_Y$ is split, then $X$ is 
		normal 
		globally $F$-regular.
	\end{lemma}
	\begin{proof}
  By \cref{prop:gFr-splinter}, $Y$ is a splinter and, by \cref{lem:splinter_descent}, $X$ is a splinter. Hence, by \cref{prop:bhatt_normal_CM}, $X$ is normal. We can now apply \cite[Cor.~6.4]{schwede_smith_globally_f_regular_and_log_Fano_varieties}.
	\end{proof}

	\subsection{Lifting global $F$-regularity}
	The following \cref{lem:gFr_ascend_pi_!} is the analogue of \cref{lem:splinter_ascend_pi_!} and \cref{rmk:crepant}.
	 \cref{lem:gFr_ascend_pi_!}\labelcref{item:case_0_pi_finite_etale} was previously established in \cite[Lem.~11.1]{patakfalvi2020beauvillebogomolov} and,
	 	under the more restrictive assumptions that $X$ is projective and Gorenstein and that $\pi$ is birational,
	   \cref{lem:gFr_ascend_pi_!}\labelcref{item:case_2_pi_proper} was established in \cite[Lem.~3.3]{gongyo_takagi_surfaces_of_globally_f_regular_and_f_split_type}.

	\begin{proposition}\label{lem:gFr_ascend_pi_!}
		Let $\pi \colon Y \to X$ be a proper surjective morphism of varieties over a field $k$ of characteristic $p>0$.
		Assume either of the following conditions\,:
		\begin{enumerate}[label=\rm{(\alph*)}]
			\item \label{item:case_0_pi_finite_etale} $\pi$ is finite quasi-\'etale, and $Y$ is normal.
			\item \label{item:case_1_pi_finite} $\pi$ is finite, $\pi^!\mcO_X \cong \mcO_Y$ in $\Coh(Y)$, and either $H^0(Y, \mcO_Y)$ is a field or $H^0(X,\mcO_X) = H^0(Y,\mcO_Y)$.
			\item \label{item:globF_reducedfibers} $\pi$ is finite, $\pi^!\mcO_X \cong \mcO_Y$ in $\Coh(Y)$,  $Y$ is normal, 
			and $\pi^\sharp \colon H^0(X, \mcO_X) \to H^0(Y , \mcO_Y)$ has reduced closed fibers. 
			\item \label{item:case_2_pi_proper} $X$ is quasi-projective, $\pi^!\mcO_X \cong \mcO_Y$ in $\sfD_{\mathrm{Coh}}(\mcO_Y)$, and either $H^0(Y, \mcO_Y)$ is a field or $H^0(X,\mcO_X) = H^0(Y,\mcO_Y)$.
		\end{enumerate}
		If $X$ is a 
		normal 
		globally $F$-regular variety, then $Y$ is a
		 normal 
		 globally $F$-regular variety.	
	\end{proposition}
	\begin{proof}
		We first prove \labelcref{item:case_0_pi_finite_etale}, \labelcref{item:case_1_pi_finite} and \labelcref{item:globF_reducedfibers}.
		By \cref{prop:gFr-splinter}, $X$ is a splinter, and it follows from  \cref{lem:splinter_ascend_pi_!} that $Y$ is a splinter and hence is normal.
		By normality of $X$, the complement of $X_{\mathrm{reg}}$ has codimension at least 2, and by finiteness of $\pi$, the complement of $\pi^{-1}(X_{\mathrm{reg}}) $ in $Y$ has codimension at least~$2$. 
		Therefore, by \cref{lem:small_birational_map_globF}, $Y$ is globally $F$-regular if and only if $\pi^{-1}(X_{\mathrm{reg}}) \subseteq Y$ is globally $F$-regular.
		Replacing $X$ by $X_{\mathrm{reg}}$ we can and do assume  that $X$ is regular.
		Since $Y$ is normal, $Y_{\mathrm{sing}}$ is a proper closed subset of $Y$ and since $\pi$ is finite, $\pi(Y_{\mathrm{sing}}) \subseteq X$ is also a proper closed subset.
		Let $U \subseteq X$ be an affine open subset in the complement of $\pi(Y_{\mathrm{sing}})$.
		By \cite[\href{https://stacks.math.columbia.edu/tag/0BCU}{Tag 0BCU}]{stacks-project},  $D\coloneqq X \setminus U$ has codimension~$1$, so defines a Weil divisor on $X$.
		Since $X$ is regular, $D$ is further Cartier and $\pi^* \mcO_X(D)$ is invertible.
		Note that the pullback  $\pi^*D = \pi^{-1}(D)$ of $D$ is  a Cartier divisor \cite[\href{https://stacks.math.columbia.edu/tag/02OO}{Tag 02OO}]{stacks-project}.
		Let $\sigma_D \colon \mcO_X\to \mcO_X(D)$ be the global section defined by the divisor $D$.
		Then the pullback $\pi^* \sigma_D$ defines a global section of $\pi^*\mcO_X(D)$ whose zero locus is precisely $\pi^{-1}(D)$.
		Thus $\pi^*\sigma_D$ is the global section of $\pi^*\mcO_X(D)=\mcO_Y(\pi^*D)$ defined by the divisor $\pi^*D$ \cite[\href{https://stacks.math.columbia.edu/tag/0C4S}{Tag 0C4S}]{stacks-project}.
		
		Now $\pi^{-1}(U) = Y \setminus \pi^*D$ is an affine open subset contained in the regular locus of $Y$, so is strongly $F$-regular.
		By \cref{thm:schwede_smith_f_reg_one_divisor}, it is enough to show that there exists an $e>0$ such that the map $\mcO_Y \to F_*^e\mcO_Y(\pi^* D)$ splits.
		Since $X$ is globally $F$-regular, there exists an integer $e>0$ and a splitting $s$ such that 
		$$ \id_{\mcO_X} \colon \mcO_X \to F_*^e\mcO_X \xrightarrow{F_*^e(\sigma_D)} F_*^e \mcO_X(D) \xrightarrow{s} \mcO_X.$$
		Since $X$ is a splinter,  we have a splitting
		$$ \id_{\mcO_X} \colon \mcO_X \xrightarrow{\pi^\sharp} \pi_* \mcO_Y \xrightarrow{t} \mcO_X.$$
		Note that by the projection formula, $t$ induces a splitting of $\mcO_X(D) \to \pi_* \pi^* \mcO_X(D) = \pi_*\mcO_Y(\pi^*D)$ which by abuse we still denote by~$t$.
        The composite map
		\begin{equation}\label{eq:splitting}
			 \mcO_X \xrightarrow{\pi^\sharp} \pi_* \mcO_Y \xrightarrow{\pi_*F^{e\sharp} } \pi_*F^e_*\mcO_Y \xrightarrow{	\pi_* F_*^e(\pi^*\sigma_D) } \pi_* F_*^e \mcO_Y (\pi^*D)
		\end{equation}
	then admits 
	$$ \pi_* F_*^e \mcO_Y (\pi^*D) =  F_*^e \pi_* \mcO_Y (\pi^*D) \xrightarrow{F^e_*(t)} F^e_*\mcO_X(D) \xrightarrow{s} \mcO_X$$ 
	as a splitting.
		As in \cref{lem:splinter_ascend_pi_!}, we consider the diagram
		\[
		\begin{tikzcd}
			\Hom_{\mcO_Y}(F_*^e \mcO_Y(\pi^*D) , \mcO_Y)  \dar{\cong} \arrow{rrr}{-\circ \pi^*\sigma_D} &&& \Hom_{\mcO_Y}(\mcO_Y, \mcO_Y) \dar{\cong} \\
			\Hom_{\mcO_Y}(F_*^e \mcO_Y(\pi^*D), \pi^! \mcO_X)  \dar{=} \arrow{rrr}{-\circ \pi^*\sigma_D} &&& \Hom_{\mcO_Y}(\mcO_Y, \pi^! \mcO_X) \dar{=} \\
			\Hom_{\mcO_X}(\pi_* F_*^e \mcO_Y(\pi^*D),\mcO_X)  \dar{-\circ \pi_* F_*^e(\pi^*\sigma_D) \circ \pi_*F^{e\sharp} \circ \pi^\sharp} \arrow{rrr}{-\circ \pi_* F_*^e(\pi^*\sigma_D) \circ \pi_*F^{e\sharp}}  &&& \Hom_{\mcO_X}(\pi_* \mcO_Y, \mcO_X) \dar{- \circ \pi^\sharp} \\
			\Hom_{\mcO_X}(\mcO_X, \mcO_X) 	\arrow[equal]{rrr} &&& \Hom_{\mcO_X}(\mcO_X,\mcO_X). 
		\end{tikzcd}
		\]
		Since the splitting of \cref{eq:splitting} is equivalent to the left-vertical map $-\circ \pi_* F_*^e(\pi^*\sigma_D) \circ \pi_* F^{e\sharp} \circ \pi^\sharp$ being surjective, in case \labelcref{item:case_1_pi_finite} we conclude, as in the proof of \cref{lem:splinter_ascend_pi_!}, that the map $-\circ \pi^*\sigma_D$ is surjective, or equivalently, that $\mcO_Y \to F_*^e \mcO_Y (\pi^*D)$ splits.
		In cases \labelcref{item:case_0_pi_finite_etale,item:globF_reducedfibers}, 
		we argue again as in the proof of \cref{lem:splinter_ascend_pi_!}, whose notation we take up.
		Let $R = H^0(X, \mcO_X)$ and $z\in \Spec R$ be a closed point.
		Since henselization is a filtered colimit of \'etale ring maps, the base change of the absolute Frobenius $F \colon Y \to Y$ along $Y^{h} \to Y$ identifies with the absolute Frobenius of $Y^{h}$.
		Let $\pi'\colon Y' \to Y$ be the normalization of $Y$ in a normal closure of $K(Y)$ over~$K(X)$.
		Arguing as above we obtain a splitting of the map $\mcO_X\to \pi_*\pi_*' F_*^e\mcO_{Y'} (\pi'^* \pi^* D)$.
		This provides a splitting of the base change $\mcO_{X^h}\to \pi_*^h\pi'^h_* F_*^e\mcO_{Y'^h} (\pi'^{h*} \pi^{h*} D\vert_{X^h})$.
		We have a decomposition into connected components $Y^h = Y_1 \amalg \dots \amalg Y_r $ and $Y'^h = Y'_1 \amalg \dots \amalg Y'_s$, where the induced action of $\mathrm{Gal}(K(Y') / K(X))$ permutes the morphisms $Y_j' \to X^h$.
		Hence, we obtain a splitting of $\mcO_{X^h}\to \pi_*^h\pi'^h_* F_*^e\mcO_{{Y'}^h} (\pi'^{h*} \pi^{h*} D\vert_{X^h})\vert_{Y'_j}$ for every $j$.
		As every $Y_i$ is dominated by one $Y_j'$, we obtain that 
		$\mcO_{X^h}\to \pi_*^h F_*^e\mcO_{Y^h} ( \pi^{h*} D\vert_{X^h})\vert_{Y_i}$
		splits for every $i$.		
		We then conclude as in the proof of \cref{lem:splinter_ascend_pi_!}.

		We now prove \labelcref{item:case_2_pi_proper}.
		By \cref{prop:gFr-splinter}, $X$ is a splinter, so a derived splinters since $k$ is of positive characteristic \cite[Thm.~1.4]{bhatt_derived_splinters_in_positive_characteristic}. 
		By \cref{rmk:uppershriek_condition}, $\pi$ is generically finite.
		Let $U\subseteq X$ be an affine regular dense open subset such that the restriction of $\pi$ to $U$ is finite.
		After possibly further shrinking $U$ we can assume 
		that $\pi^{-1}(U)$ is also regular.
		Since restriction to open subschemes commutes with exceptional inverse image functors \cite[\href{https://stacks.math.columbia.edu/tag/0G4J}{Tag 0G4J}]{stacks-project}, we have $\pi\vert_U^!\mcO_U \cong \mcO_{\pi^{-1}(U)}$.
		Since $X$ is quasi-projective, any Weil divisor on $X$ is dominated by a Cartier divisor. 
		Thus, by possibly further shrinking $U$, we may and do assume that $D\coloneqq X\setminus U$ is a Cartier divisor. 
		Since $\pi^{-1}(U)$ is affine an regular, it is enough to show, by \cref{thm:schwede_smith_f_reg_one_divisor}, that there exists an $e>0$ such that the map $\mcO_Y \to F_*^e\mcO_Y(\pi^* D)$ splits.
		By arguing as in case \labelcref{item:case_1_pi_finite} with $\pi_*$ replaced by $\mathrm{R}\pi_*$,
		we conclude that a splitting of $\mcO_X \to F_*^e \mcO_X(D)$ induces a splitting of $\mcO_Y \to F_*^e\mcO_Y(\pi^* D)$.
	\end{proof}

	\begin{remark}
		\label{rmk:lift-gFr}
		By using the version of \cref{thm:schwede_smith_f_reg_one_divisor} for pairs, i.e., the original \cite[Thm.~3.9]{schwede_smith_globally_f_regular_and_log_Fano_varieties}, we leave it to the reader to show the following version of \cref{lem:gFr_ascend_pi_!} for pairs. 
		Let  $\pi\colon Y \to X$ be a finite 
		surjective morphism of varieties over a field of positive characteristic that satisfies \labelcref{item:case_0_pi_finite_etale}, \labelcref{item:case_1_pi_finite}, or~\labelcref{item:globF_reducedfibers}, and let $\Delta$ be an effective $\bbQ$-Weil divisor on $X$.
		If $X$ is normal and if $(X, \Delta)$ is globally $F$-regular, then $Y$ is normal
	 	and $(Y,\pi^*\Delta)$ is globally $F$-regular.
	\end{remark}

	\subsection{Products of globally $F$-regular varieties}
	As far as we know, it is unknown whether the splinter property is stable under product. 
	On the other hand, global $F$-regularity for products is more tractable since the Frobenius of a product is the product of the Frobenii and since one may use the criterion of \cref{thm:schwede_smith_f_reg_one_divisor} to check global $F$-regularity for one specific divisor.
	The following \cref{prop:product-gFr} generalizes \cite[Thm.~5.2]{hashimoto_surjectivity_of_multiplication_and_f_regularity_of_multigraded_rings}, where the case of products of projective globally $F$-regular varieties was dealt with by taking affine cones.

	\begin{proposition}\label{prop:product-gFr}
		Let $X$ and $Y$ be normal varieties over a perfect field~$k$ of positive characteristic.
		Then, $X$ and $Y$ are globally $F$-regular if and only if their product $X \times_k Y$ is globally $F$-regular.
	\end{proposition}
	\begin{proof}		
		Denote by $\pi_X$ and $\pi_Y$ the natural projections from $X\times_k Y$ to $X$ and $Y$, respectively.
		Since $X$ and $Y$ are normal and $k$ is perfect, their product $X\times_k Y$ is normal, see \cite[\href{https://stacks.math.columbia.edu/tag/038L}{Tag 038L}]{stacks-project}.
		Moreover, $X \setminus X_{\mathrm{reg}}$ and $Y \setminus Y_{\mathrm{reg}}$ both have codimension $\geq 2$ and thus $X \times_k Y \setminus X_{\mathrm{reg}} \times_k Y_{\mathrm{reg}}$ has codimension $\geq 2$.
		Therefore, by \cref{lem:small_birational_map_globF}, 
		we can and do assume without loss of generality that $X$ and $Y$ are smooth over $k$.
		
		Assume first that $X\times_k Y$ is globally $F$-regular. As in \cref{rmk:product}, since $\pi_*\mcO_{X\times_k Y}  = \linebreak \mcO_X \otimes_k H^0(Y,\mcO_Y)$, any splitting of the $k$-linear map $k \to H^0(Y,\mcO_Y), 1 \mapsto 1_Y$ provides a splitting to the natural map $\mcO_X \to \pi_*\mcO_{X\times_k Y}$. By \cref{lem:gFr_descent}, it follows that $X$ is globally $F$-regular.
		
		For the converse, we first note that there exist effective Cartier divisors $D$ on $X$ and $E$ on $Y$ such that $X\setminus D$ and $Y \setminus E$ are affine.
		Indeed, since $X$ and $Y$ are normal and $k$ is perfect, they admit dense affine open subsets, and then use the fact that the complement of a dense affine open subset is a divisor by \cite[\href{https://stacks.math.columbia.edu/tag/0BCU}{Tag 0BCU}]{stacks-project}.
		The divisors obtained this way are a priori Weil divisors, but since $X$ and $Y$ are smooth, they are actually Cartier divisors.
		Since $k$ is assumed to be perfect,
		$$X \setminus D \times_k Y \setminus E = X\times_k Y \setminus \pi_X^*D \cup \pi_Y^*E$$
		is smooth and affine, so in particular strongly $F$-regular.
		By \cref{thm:schwede_smith_f_reg_one_divisor} it is enough to show that the map
		$$ \mcO_{X\times Y} \to  F_*^e \mcO_{X\times Y} \to F_*^e \mcO_{X\times Y}(\pi_X^*D + \pi_Y^*E)$$
		splits for some $e>0$.
		Since $X$ is globally $F$-regular, we can find an $e>0$ such that $\mcO_X \to F_*^e \mcO_X(D)$ splits and similarly for $Y$.
		As remarked in \cite[p.~558]{smith_globally_f_regular_varieties_applications_to_vanishing_theorems_for_quotients_of_fano_varieties}, if $\mcO_X \to F_*^e \mcO_X(D)$ splits for some $e>0$, then $\mcO_X \to F_*^{e'} \mcO_X(D)$ splits for all $e'\geq e$. Thus, there exists an integer $e>0$ such that both $\mcO_X \to F_*^e \mcO_X(D)$ and $\mcO_Y \to F_*^e \mcO_Y(E)$ split.
		The morphism
		$$\sigma_{\pi_X^*D + \pi_Y^* E} \colon \mcO_{X\times Y} \to \mcO_{X\times Y}(\pi_X^*D + \pi_Y^* E)$$
		can be identified with the tensor product $\pi_X^*\sigma_D \otimes \pi_Y^*\sigma_E$, where
		\begin{equation*}
		\sigma_D\colon \mcO_X \to \mcO_X(D)\quad \text{and} \quad \sigma_E \colon \mcO_Y \to \mcO_Y(E)
		\end{equation*}
		denote the corresponding morphisms on $X$ and $Y$.
		Pushing forward along Frobenius, we obtain 
		$$F^e_*\sigma_{\pi_X^*D + \pi_Y^* E} = \pi_X^* F^e_*\sigma_D \otimes \pi_Y^* F^e_*\sigma_E.$$
		We conclude, by taking the tensor product of the sections of $\mcO_X \to F_*^e \mcO_X(D)$ and $\mcO_Y \to F_*^e \mcO_Y(E)$, that
		$$\mcO_{X\times Y} \to F_*^e \mcO_{X\times Y}(\pi_X^*D + \pi_Y^* E)$$
		splits. Hence $X\times_k Y$ is globally $F$-regular.
	\end{proof}

	\begin{remark}\label{rmk:prod-pairs}
		By using the version of \cref{thm:schwede_smith_f_reg_one_divisor} for pairs, i.e., the original \cite[Thm.~3.9]{schwede_smith_globally_f_regular_and_log_Fano_varieties}, we leave it to the reader to show the following version of \cref{prop:product-gFr} for pairs\,: 
		Let $X$ and $Y$ be normal varieties over a perfect field $k$ of positive characteristic, and denote by $\pi_X\colon X\times_k Y \to X$ and $\pi_Y\colon X\times_k Y \to Y$ the natural projections. Let $\Delta_X$ and $\Delta_Y$ be effective $\bbQ$-Weil divisors on $X$ and $Y$, respectively.
		Then, $(X,\Delta_X)$ and $(Y, \Delta_Y)$ are globally $F$-regular if and only if $(X\times_k Y, \pi_X^*\Delta_X + \pi_Y^*\Delta_Y)$ is globally $F$-regular.
	\end{remark}

	\subsection{Global $F$-regularity and base change of field}
	By using the lifting \cref{lem:gFr_ascend_pi_!}, it is possible to show that the base change results  for splinters along algebraic field extensions of \cref{SS:splinters-base change} also hold for normal global $F$-regular varieties.
	However, by using the criterion of Schwede--Smith \cite[Thm.~3.9]{schwede_smith_globally_f_regular_and_log_Fano_varieties},
	Gongyo--Li--Patakfalvi--Schwede--Tanaka--Zong \cite{gongyo_li_patakfalvi_zsolt_schwede_tanaka_zong_on_rational_connectedness_of_globally_f_regular_threefolds} have established a more general base change results that deals with not necessarily algebraic extensions.

	\begin{proposition}[Gongyo--Li--Patakfalvi--Schwede--Tanaka--Zong \cite{gongyo_li_patakfalvi_zsolt_schwede_tanaka_zong_on_rational_connectedness_of_globally_f_regular_threefolds}]
		\label{lem:gFr_base change}

		Let $X$ be a normal proper scheme over 
		a field $k$ of positive characteristic. 
		Assume that $(X,\Delta)$ is globally $F$-regular. 
		Then, for any $F$-finite field extension $L$ of $k$ with a morphism $\Spec L \to \Spec H^0(X,\mcO_X)$, the scheme $X_L \coloneqq  X\times_{H^0(X,\mcO_X)} L$ is normal and the pair $(X_L,\Delta_L)$ is globally $F$-regular.
	\end{proposition}	
	\begin{proof} Let us provide an alternate proof.
		 We may and do assume $X$ is connected.
		Since any divisor on $X_L$ is defined over a finitely generated field extension of the field $H^0(X,\mcO_X)$, we may assume that $L$ is a simple extension of $H^0(X,\mcO_X)$.
		If $L$ is algebraic, we can apply \cref{lem:gFr_ascend_pi_!}, 
		while if $L$ is purely transcendental, we can apply \cref{prop:product-gFr} (or  \cref{rmk:prod-pairs} in case $\Delta \neq 0$) and \cref{lem:generic_fiber_globF} to $X\times_{H^0(X,\mcO_X)} \bbA^1 \to \bbA^1$. Note that in this situation it is not necessary to assume that $H^0(X, \mcO_X)$ is perfect as $X_{\mathrm{reg}} \times_{H^0(X, \mcO_X)} \bbA^1$ is regular.
	\end{proof}

	\begin{remark} \label{rmk:bc-splinter}
		Our proof of 
\cref{lem:gFr_base change} shows that one could extend \cref{lem:splinter_base change} concerned with base change of splinters along algebraic field extensions to arbitrary field extensions 
if one could establish that the splinter property is stable under taking product with the affine line~$\bbA^1$.
	\end{remark}
	
	\begin{corollary}[{\cite[Cor.~2.8]{gongyo_li_patakfalvi_zsolt_schwede_tanaka_zong_on_rational_connectedness_of_globally_f_regular_threefolds}}]
		Let $X$ be a connected normal proper scheme over an $F$-finite field $k$. Assume that $H^0(X,\mcO_X)$ is a separable extension of $k$ and that $(X,\Delta)$ is globally $F$-regular.
		Then, for any $F$-finite field extension~$K$ of $k$,
		the scheme $X_K \coloneqq  X\times_k K$ is normal and the pair $(X_K,\Delta_K)$ is globally $F$-regular.
	\end{corollary}

\section{Finite torsors over splinters}
	\label{S:finite_torsors_over_splitners-generalcase}

	 We say that a morphism $\pi\colon Y\to X$ of schemes over a 
	  scheme $S$ is a \emph{finite torsor} over $S$ if it is a
	  fppf torsor under a finite  locally free group scheme $G$ over $S$.
	  (The reason for not restricting ourselves to finite flat group schemes over a Noetherian scheme is that $H^0(X,\mcO_X)$ might fail to be Noetherian although $X$ is Noetherian.)
	The aim of this section is to prove Theorems~\labelcref{thm:intro_d-etale} and~\labelcref{thm:intro_d-torsor}.
   Recall from \cref{ss:trace} that  a finite \'etale morphism $\pi \colon Y \to X$ satisfies $\pi^!\mcO_X \cong \mcO_Y$.
   	First, in order to apply our lifting \cref{lem:splinter_ascend_pi_!} to finite torsors over  splinters, we have\,:

	\begin{lemma}
		\label{lem:covers}
		Let $\pi\colon Y \to X$ be a 
		morphism of Noetherian schemes over a ring $R$ such that $\Pic(\Spec R) = 0$\,; e.g., $R$ is a local ring or a UFD. 
		If  $\pi$ is a finite torsor over $R$,
		then $\pi^!\mcO_X \cong \mcO_Y$.
	\end{lemma}
	\begin{proof}
      As discussed in \cref{ss:trace},
		in order to show that $\pi^!\mcO_X \cong \mcO_Y$,  it is equivalent to produce an $\mcO_X$-linear map 
		$\mathrm{Tr}_{Y/X} \colon \pi_*\mcO_Y \to \mcO_X$
		such that the symmetric bilinear form $\mathrm{Tr}_{Y/X}(\alpha \cdot \beta)$ on the locally free sheaf $\pi_*\mcO_Y$ with values in $\mcO_X$ is nonsingular.
	In 	case $\pi$ is \'etale, this is achieved by considering the classical trace.
		In general,  the classical trace need not be nonsingular, but the existence of a nonsingular trace is claimed in \cite[p.~222]{bombieri_mumford_enqiques_classification_of_surfaces_in_char_p_iii} in the special case where $R$ is a field and 
		a proof is provided in  \cite[Thm.~3.9]{carvajal_rojas_finite_torsors_over_strongly_f_regular_singularities}. 
		(Note from e.g.~\cite[Prop.~2.6.4 \&~2.6.5(i)]{brion_some_structure_theorems_for_algebraic_groups} that any finite $G$-torsor is a finite $G$-quotient in the sense of \cite[Rmk.~2.3]{carvajal_rojas_finite_torsors_over_strongly_f_regular_singularities}.)
		We now consider the more general case where $G$ is a finite locally free group scheme over 
		a ring $R$.
		In that case, $H\coloneqq H^0(G, \mcO_G)$ is a finitely generated projective
		Hopf algebra with antipode.
		The dual Hopf algebra $H^\vee = \Hom_R(H, R)$ is also a finitely generated projective Hopf algebra with antipode.
		If $\Pic(\Spec R)=0$, then $H^\vee$ admits the additional structure of a Frobenius algebra  \cite[Thm.~1]{pareigis_when_hopf_algebras_are_frobenius_algebras}.
	By \cite[Thm.~3 \& Discussion on p.~596]{pareigis_when_hopf_algebras_are_frobenius_algebras} the $R$-submodule $\int_{H^\vee}^l \subseteq H^\vee$ of left integrals is freely generated by a nonsingular
		left integral $\mathrm{Tr}_G \in H^\vee$.
	If $\pi \colon Y \to X$ is a $G$-torsor over $R$, 
by pulling back $\mathrm{Tr}_G$ along $Y \to \Spec R$ and composing with $ \pi_*\mcO_Y\to \pi_*a_* \mcO_{Y\times_R G}$, where $a\colon Y\times_R G \to Y$ is the action of $G$ on $Y$, one obtains as in \cite[\S3.1, p.~12]{carvajal_rojas_finite_torsors_over_strongly_f_regular_singularities}  an $\mcO_X$-linear map $\pi_*\mcO_Y \to \pi_*\mcO_Y$ that factors through $\mcO_X = (\pi_*\mcO_Y)^G$ to an $\mcO_X$-linear map $\mathrm{Tr}_{Y/X} \colon \pi_*\mcO_Y \to \mcO_X$.
		Since $\mathrm{Tr}_G$ is nonsingular, arguing as in \cite[Thm.~3.9]{carvajal_rojas_finite_torsors_over_strongly_f_regular_singularities} shows that the bilinear form $(\alpha, \beta) \mapsto \mathrm{Tr}_{Y/X}(\alpha\cdot \beta)$ is nonsingular.
	\end{proof}
	
	\begin{remark}
	Let $\pi \colon Y \to X$ be a finite torsor over a ring $R$. 
	Assume that  $\Pic(\Spec R) =0$ and that $X$ is Noetherian and admits a dualizing complex. As explained in \cref{rmk:lift_Gorenstein}, \cref{lem:covers} implies that $\pi^*\omega_X^\bullet \cong \omega_Y^\bullet$.
In particular, if $X$ is Gorenstein, then $Y$ is Gorenstein. Likewise, if $X$ is Cohen--Macaulay, then $Y$ is Cohen--Macaulay.
	\end{remark}

	We have the following lemma from \cite{bhatt_et_al_globally_+_regular_varieties_and_mmp_for_threefolds_in_mixed_char}\,:
	\begin{lemma}[{\cite[Lem.~6.6]{bhatt_et_al_globally_+_regular_varieties_and_mmp_for_threefolds_in_mixed_char}}]\label{lem:bhatt_et_al_enough_to_check_at_closed_points}
		Let $X \to \Spec R$ be a		
		Noetherian scheme over a ring~$R$.
		Then the following are equivalent\,:
		\begin{enumerate}
			\item The scheme $X$ is a splinter.
			\item For each closed point $z \in \Spec R$ the base change to the localization $X_{R_z}$ is a splinter.
		\end{enumerate}
		Further, assume $X$ is in addition normal and excellent and let $\Delta$ be an effective $\bbQ$-Weil divisor on~$X$.
		Then the following are equivalent\,:
		\begin{enumerate}
			\item \label{item:localization_i} The pair $(X , \Delta)$ is globally $+$-regular.
			\item \label{item:localization_ii} For each closed point $z \in \Spec R$ the base change to the localization $(X_{R_z}, \Delta_{R_z})$ is globally $+$-regular.
		\end{enumerate}
	\end{lemma}
	\begin{proof}
		Our assumptions are less restrictive than the setup of \cite[\S 6]{bhatt_et_al_globally_+_regular_varieties_and_mmp_for_threefolds_in_mixed_char}, but the proof of \cite[Lem.~6.6]{bhatt_et_al_globally_+_regular_varieties_and_mmp_for_threefolds_in_mixed_char} works as we outline below.
Note that in any case $X$ is normal. 
		Working on each connected component of $X$ separately, we can and do assume that $X$ is integral.
		Let $f\colon Y \to X$ be a finite cover (with $Y$ normal if we assume $(X,\Delta)$ globally $+$-regular).
		By flat base change, and as argued in the proof of \cite[Lem.~6.6]{bhatt_et_al_globally_+_regular_varieties_and_mmp_for_threefolds_in_mixed_char},   the evaluation at~1 map
		$$\Hom_{\mcO_X}(f_* \mcO_Y(\lfloor f^* \Delta \rfloor), \mcO_X)\to H^0(X, \mcO_X)$$
		is surjective if and only if the evaluation-at-$1$ map
		$$\Hom_{\mcO_{X_{R_z}}}(f_* \mcO_{R_z}(\lfloor f^* \Delta\vert_{X_{R_z}} \rfloor), \mcO_{X_{R_z}} )\to H^0({X_{R_z}}, \mcO_{X_{R_z}})$$ is surjective
		for every closed point $z\in \Spec R$. This shows the implications $\labelcref{item:localization_ii} \Rightarrow \labelcref{item:localization_i}$.
		For the converse, it is enough to observe that any finite surjective morphism $Y' \to X_{R_z}$  is the localization of a finite surjective morphism $Y \to X$, and in case $X$ is excellent that any finite surjective morphism $ Y' \to X_{R_z}$ with $Y'$ normal and integral is the localization of a finite surjective morphism $Y \to X$ with $Y$ normal.
		To see this, use that $Y' \to X_{R_z}$ is of finite presentation and hence spreads to a morphism of finite presentation $Y'' \to X|_V$ for some dense affine open $V \subseteq \Spec R$~\cite[Thm.~8.8.2]{ega_iv_3}. 
		The latter restricts on a dense affine open $U \subseteq V$ to a finite surjective morphism $Y''|_U \to X|_U$~\cite[Thm.~8.10.5]{ega_iv_3}, which then by Zariski's Main Theorem \cite[\href{https://stacks.math.columbia.edu/tag/05K0}{Tag 05K0}]{stacks-project} extends to a finite surjective morphism $Y \to X$.
		In case $X$ is excellent, $X$ is in particular Nagata, 
	 so is $X_{R_z}$ \cite[\href{https://stacks.math.columbia.edu/tag/032U}{Tag 032U}]{stacks-project}.
	Therefore, a finite surjective morphism $Y \to X$ with $Y$ normal is provided by taking
	the normalization of $\mcO_X$ in the fraction field $K(Y')$, see \cite[\href{https://stacks.math.columbia.edu/tag/0AVK}{Tag 0AVK}]{stacks-project}.
	\end{proof}

The following proposition  extends \cref{lem:splinter_ascend_pi_!}\labelcref{item:etale} from lifting the splinter property along finite \'etale covers to finite torsors.

	\begin{proposition}[Lifting the splinter property along finite torsors]
		\label{prop:torsor}
	Let $\pi\colon Y \to X$ be a
	morphism of Noetherian 
	schemes.
	Assume that $\pi$ is a finite torsor over a  ring~$R$,
	and that it satisfies either of the following conditions\,:
	\begin{enumerate}[label=\rm{(\alph*)}]
		\item  \label{item:i} $H^0(Y,\mcO_Y)$ is a field.
		\item \label{ii}   $H^0(X,\mcO_X) = H^0(Y,\mcO_Y)$.
		\item \label{item:iii} $X$ is Nagata, $Y$ is normal, and $\pi^\sharp \colon H^0(X, \mcO_X) \to H^0(Y , \mcO_Y)$ has reduced closed fibers.
	\end{enumerate}
	If $X$ is a splinter,
	then $Y$ is a splinter. 	
	\end{proposition}
	\begin{proof}
		Let $z \in \Spec R$ be a closed point. 
		Note that all three conditions \labelcref{item:i,ii,item:iii} are stable under base change along $R \to R_z$, so that
		by \cref{lem:bhatt_et_al_enough_to_check_at_closed_points}
		we can reduce to the case where $R$ is  local.
		From \cref{lem:covers} we know that $\pi^!\mcO_X \cong \mcO_{Y}$ for any finite torsor morphism $\pi\colon Y \to X$ over a local ring.
		\cref{lem:splinter_ascend_pi_!} shows that $Y$ is a splinter. 
	\end{proof}

\begin{remark}
	\cref{prop:torsor} does not hold without any assumptions on global sections. Consider indeed the trivial torsor given by a finite non-reduced group scheme $G$ over a field $k$\,: $\Spec k$ is a splinter but $G$ fails to be a splinter as it is non-reduced.
\end{remark}

\begin{proposition}[Lifting the derived splinter property along finite torsors]
Let $\pi\colon Y \to X$ be a
morphism of Noetherian 
schemes.
Assume that $\pi$ is a finite torsor over a  ring~$R$,
and that either $H^0(Y,\mcO_Y)$ is a field or $H^0(X,\mcO_X) = H^0(Y,\mcO_Y)$.
	If $X$ is a derived splinter, then $Y$ is a derived splinter.
\end{proposition}
\begin{proof}
	One argues as in the proof of \cref{prop:torsor} by reducing to $R$ local (using Nagata compactification \cite[\href{https://stacks.math.columbia.edu/tag/0F41}{Tag 0F41}]{stacks-project} instead of Zariski's Main Theorem in the proof of \cref{lem:bhatt_et_al_enough_to_check_at_closed_points})
and by replacing the use of \cref{lem:splinter_ascend_pi_!} with \cref{lem:Dsplinter_ascend_pi_!}.
\end{proof}
	 
	We say that a morphism of schemes $\pi\colon Y \to X$
is  a \emph{quasi-torsor} over a ring~$R$ 
if it is a morphism over $R$ and there exists $U \subseteq X$ open with $\codim_X (X\setminus U)\geq 2$ such that $f\vert_{f^{-1}(U)}$ is  a  torsor under a finite locally free group scheme $G$ over $R$. 
	The following theorem, which generalizes
 \cite[Prop.~6.20]{bhatt_et_al_globally_+_regular_varieties_and_mmp_for_threefolds_in_mixed_char},  
establishes Theorems~\labelcref{thm:intro_d-etale} and~\labelcref{thm:intro_d-torsor}.

	\begin{theorem}[Theorems~\labelcref{thm:intro_d-etale,thm:intro_d-torsor}]
		\label{prop:quasitorsor}
		Let $\pi\colon Y \to X$ be a
		finite surjective morphism of normal Nagata Noetherian 
		schemes.
		Assume that $\pi$ satisfies either of the following conditions\,:
		\begin{enumerate}[label=\rm{(\alph*)}]
			\item \label{item:finite_quasietale} $\pi$ is  quasi-\'etale.
			\item \label{item:finite_quasitorsor_reducedfibers} 
				$\pi$ is a quasi-torsor over a ring $R$, 
				and $\pi^\sharp \colon H^0(X, \mcO_X) \to H^0(Y , \mcO_Y)$ has reduced closed fibers.  
		\end{enumerate}
		If $X$ is a splinter, then $Y$ is a splinter.
		
	 Assume in addition that $X$ is excellent, and let  $\Delta$ be an effective $\bbQ$-Weil divisor on $X$.
		If $(X, \Delta)$ is globally $+$-regular,
		then $(Y, \pi^* \Delta)$ is globally $+$-regular.
	\end{theorem}
	\begin{proof}
		The theorem in case~\labelcref{item:finite_quasietale} is  \cref{lem:splinter_ascend_pi_!}\labelcref{item:etale} and \cref{lem:lifting_glob_+_reg}\labelcref{item:globplus_quasietale}.
		We now prove case~ \labelcref{item:finite_quasitorsor_reducedfibers}.
		Using \cref{lem:bhatt_et_al_enough_to_check_at_closed_points} we can and do reduce to the case where $R$ is local so that $\Pic( \Spec R)=0$.
		By assumption, there exists an open subset $U\subseteq X$ such that $\codim_X (X\setminus U)\geq 2$ and such that $\pi\vert_U \colon V\coloneqq \pi^{-1}(U) \to U$ is 
		a finite torsor over $R$. 
 By \cref{lem:covers}, 	we have $\pi_U^! \mcO_U \cong \mcO_V$.
		Since $Y$ is assumed to be normal, we can work on each connected component of $Y$ separately and assume without loss of generality that $Y$ is connected.
		Since $X$ is a splinter, it is Cohen--Macaulay and hence
		universally catenary \cite[\href{https://stacks.math.columbia.edu/tag/02II}{Tag 02II}]{stacks-project}.
		From the fact that $\pi$ is finite surjective, it follows from \cite[\href{https://stacks.math.columbia.edu/tag/02II}{Tag 02II}]{stacks-project}
		that $Y \setminus V$ has codimension at least $2$ in~$Y$.
		Since $X$ and $Y$ are normal, 
		 we have
		$H^0(U, \mcO_U) = H^0(X, \mcO_X)$ and $H^0(V, \mcO_V) = H^0(Y, \mcO_Y)$.
		If $X$ is a splinter, \cref{lem:splinter_ascend_pi_!}\labelcref{item:reducedfibers} shows that $V$ is a splinter and we conclude with \cref{lem:splinter_invariant_codim_2_surgery} that $Y$ is a splinter.
		 On the other hand, if $(X,\Delta)$ is globally $+$-regular, \cref{lem:lifting_glob_+_reg} shows that $(V, \pi^* \Delta \vert_V)$ is globally $+$-regular and
		 we conclude with \cref{rmk:dense_open_glob_+_reg} that $(Y,\pi^*\Delta)$ is globally $+$-regular.
	\end{proof}

\begin{remark}
By replacing the use of  \cref{lem:lifting_glob_+_reg} with \cref{lem:gFr_ascend_pi_!} (or rather \cref{rmk:lift-gFr}) and the use of 
 \cref{rmk:dense_open_glob_+_reg} with \cref{lem:small_birational_map_globF} in the proof of \cref{prop:quasitorsor} in the global $+$-regular case, one obtains the following statement. 
 Let $\pi\colon Y \to X$ be a finite 
 morphism of normal varieties over 
 a field of positive characteristic.
 Assume that $\pi$ is quasi-\'etale, or that $\pi$ is a quasi-torsor over a ring such that $\pi^\sharp \colon H^0(X, \mcO_X) \to H^0(Y , \mcO_Y)$ has reduced closed fibers.  
 Let $\Delta$ be an effective $\bbQ$-Weil divisor on $X$. 
 If $(X,\Delta)$ is globally $F$-regular, then $(Y,\pi^* \Delta)$ is globally $F$-regular. 
	This generalizes \cite[Lem.~11.1]{patakfalvi2020beauvillebogomolov} where the quasi-\'etale case was treated.
\end{remark}

	\section{Finite torsors over proper splinters over a field}
	\label{S:finite_torsors_over_splitners}
	We now focus on finite torsors over proper splinters over a field and 
	 prove Theorem~\ref{item:thm:triv_fund_grp}.
		The following lemma 
	can be found in \cite[Thm.~2, p.~121]{mumford-abelian} (we thank Michel Brion for bringing this reference to our attention). We provide an alternate proof based on Hirzebruch--Riemann--Roch for (not necessarily smooth) proper schemes over a field.
	
	\begin{lemma}\label{lem:chi_formula_cover}
		Let $X$ be a proper scheme over a field $k$ and let $\pi\colon Y \to X$ be a morphism of schemes over $k$. Assume that $\pi$ satisfies either of the following conditions\,:
		\begin{enumerate}[label=\rm{(\alph*)}]
			\item $\pi$ is finite \'etale.
			\item $\pi$ is a finite torsor.\label{item:torsor_finite_group_scheme}
		\end{enumerate}
		Then $\chi(\mcO_Y) = \deg(\pi)\chi(\mcO_X)$.
	\end{lemma}
	\begin{proof}
		We first establish \labelcref{item:torsor_finite_group_scheme}.
		Recall that there is a Hirzebruch--Riemann--Roch formula
		$$\chi(\mcE) = \int_X \ch(\mcE)\cap \mathrm{td}(X)$$
		for any vector bundle $\mcE$ on a proper scheme $X$ over a field\,; see \cite[Cor.~18.3.1]{fulton_intersection_theory}.
		In particular, the Euler characteristic only depends on the class of the Chern character $\ch(\mcE) \in \mathrm{A}^*(X)_\bbQ$, where $\mathrm{A}^*(X)_\bbQ$ denotes the Chow cohomology \cite[\href{https://stacks.math.columbia.edu/tag/0FDV}{Tag 0FDV}]{stacks-project}  with $\bbQ$-coefficients.
		Assume $\pi\colon Y \to X$ is a torsor under a finite group scheme $G$ over $k$.
		By definition of a $G$-torsor, the product $Y \times_X Y \to Y$ is isomorphic to $G \times_k Y \to Y$ as schemes over~$Y$.
		This gives an isomorphism 
		$$\pi_* \mcO_Y \otimes \pi_* \mcO_Y \cong \pi_*\mcO_Y^{\oplus n},$$
		where $n=\deg(\pi)$ is the order of $G$. 	
		Since $\pi$ is finite flat, $\pi_* \mcO_Y$ is a vector bundle (of rank $n$) on ~$X$,
		and hence by \cite[\href{https://stacks.math.columbia.edu/tag/02UM}{Tag 02UM}]{stacks-project} we have the identity
		\begin{equation*}
		\ch(\pi_* \mcO_Y)\cdot \ch( \pi_* \mcO_Y) = n\, \ch(\pi_* \mcO_Y) \quad \mbox{in } \mathrm{A}^*(X)_\bbQ.
		\end{equation*}
		Since $\ch(\pi_*\mcO_Y)$ is a unit in $\mathrm{A}^*(X)_\bbQ$, 
		we obtain $\ch(\pi_*\mcO_Y) = n \ch( \mcO_X)$. By Hirzebruch--Riemann--Roch, the equality $\chi(\pi_*\mcO_Y) = \chi(\mcO_Y)= n \chi(\mcO_X)$ follows.
		If $\pi$ is finite \'etale, one can use Grothendieck--Riemann--Roch for proper schemes over a field \cite[Thm.~18.3]{fulton_intersection_theory}, while noting that the relative Todd class $\mathrm{td}(T_\pi)$ is equal to~$1$. 
		Alternatively one can reduce to case \labelcref{item:torsor_finite_group_scheme}
by using the fact \cite[\href{https://stacks.math.columbia.edu/tag/03SF}{Tag 03SF}]{stacks-project} that
there exists a Galois cover $Y' \to X$ dominating $\pi$.
	\end{proof}

	\begin{theorem}\label{thm:splinters and covers}
		Let $X$ be a proper scheme over an integral Noetherian scheme $S$ of positive characteristic and let $\pi\colon Y \to X$ be a  morphism of schemes over $S$.
		Assume that  $H^0(X,\mcO_X) = H^0(Y,\mcO_Y)$.
		In addition, assume either of the following conditions\,:
		\begin{enumerate}[label=\rm{(\alph*)}]
			\item \label{item:finite_etale_morph} $\pi$ is finite \'etale.
			\item \label{item:finite_torsor} $\pi$ is a finite torsor.
		\end{enumerate}
		If $X$ is a splinter, 
		then $\pi$ is an isomorphism.
	\end{theorem}
	\begin{proof}
		Let $\eta$ be the generic point of $S$. It is is enough to show that the restriction $\pi_\eta \colon Y_\eta \to X_\eta$ of $\pi$ to $\eta$ is an isomorphism.
		By \cref{lem:generic_fiber_splinter}, if $X$ is a splinter, then $X_\eta$ is a splinter. Therefore, we may and do assume that $S$ is the spectrum of a field.
		Moreover, since a splinter is normal,
		we may and do assume that $X$ is connected, in which case $H^0(X,\mcO_X)$ is a field.
		On the other hand, by 		\cref{lem:splinter_ascend_pi_!}\labelcref{item:etale}, or more simply 
		\cref{lem:splinter_ascend_pi_!}\labelcref{item:sameH0}, in the \'etale case and
		by \cref{prop:torsor} in the torsor case, $Y$ is a splinter.
		Since the structure sheaf of a proper splinter in positive characteristic has trivial positive cohomology by \cref{prop:semiample}, we have $\chi(\mcO_X) = \chi(\mcO_Y)=1$, where the dimension is taken with respect to the field $H^0(X, \mcO_X)$
		and we conclude with \cref{lem:chi_formula_cover} that $\pi$ is an isomorphism.
	\end{proof}

	\begin{theorem}
		\label{prop:trivial_fundamental_group}
		Let $X$ be a connected proper scheme over a field $k$ of positive characteristic with a $k$-rational point $x\in X(k)$.
	Assume that $X$ is a splinter.
		\begin{enumerate}
			\item If $k$ is separably closed, then the \'etale fundamental group $\pi_1^{\text{\'et}}(X,x)$ of $X$ is trivial. \label{enum:geometric_fund_group_trivial}
			\item \emph{(Theorem~\ref{item:thm:triv_fund_grp})}
			 The Nori fundamental group scheme $\pi_1^N(X,x)$ of $X$  is trivial. \label{enum:Nori_fund_group_trivial}
		\end{enumerate}
	\end{theorem}
	
	\begin{proof}
		Statement \labelcref{enum:geometric_fund_group_trivial} follows from \cref{thm:splinters and covers} since for $k$ separably closed any connected finite \'etale cover $\pi \colon Y \to X$ satisfies $H^0(Y, \mcO_Y)=H^0(X, \mcO_X)=k$. 
		For statement \labelcref{enum:Nori_fund_group_trivial}, first note that the Nori fundamental group scheme $\pi_1^N(X,x)$ is well-defined as a splinter is reduced. 
		Assume for contradiction that $\pi_1^N(X,x)$ is nontrivial. Since $\pi_1^N(X,x)$ is pro-finite, there is a surjective group scheme  homomorphism $\pi_1^N(X,x) \twoheadrightarrow G$ to a nontrivial finite group scheme $G$. 
		By \cite[Prop.~3, p.~87]{nori_the_fundamental_group_scheme}, there exists a $G$-torsor $Y\to X$ with $H^0(Y, \mcO_Y)=H^0(X, \mcO_X)$. 
		This contradicts \cref{thm:splinters and covers}.
	\end{proof}

	\begin{remark}[On the triviality of the Nori fundamental group scheme]
		As mentioned in \cref{cor:splinter_geom_normal}, a connected proper splinter $X$ is geometrically normal, hence geometrically reduced, over the field $K\coloneqq H^0(X, \mcO_X)$\,; in particular it acquires a rational point after some finite separable field extension of $K$. 
		Moreover, recall the general facts that the Nori fundamental group scheme is invariant under separable base change, and that the triviality of the Nori fundamental group scheme is independent of the choice of base point.
	\end{remark}

	We say that a finite \'etale cover $\pi\colon Y \to X$ is trivial if it is isomorphic over $X$ to a disjoint union of copies of $X$. We say that a finite torsor $\pi\colon Y\to X$ under a finite group scheme $G$ over $k$ is trivial if it is isomorphic to $X\times_k G$ over $X$.
	An immediate consequence of \cref{prop:trivial_fundamental_group} is the following\,:
	
\begin{corollary} \label{cor:torsor-splinter}
	Let $X$ be a connected proper scheme over a field $k$ of positive characteristic. Assume that $X$ is a splinter.
	\begin{enumerate}
		\item \label{item:etale_cover_trivial} If $k$ is separably closed, then  any finite \'etale cover of $X$ is trivial.
		\item \label{item:finite_torsor_trivial} If $k$ is algebraically closed, then any finite torsor over $X$ is trivial.
	\end{enumerate}
\end{corollary}
\begin{proof}
	Statement \labelcref{item:etale_cover_trivial} is clear from (the proof of) \cref{prop:trivial_fundamental_group}, while statement \labelcref{item:finite_torsor_trivial} follows from the fact \cite{nori_the_fundamental_group_scheme} that
	 for a $k$-point of $x\in X(k)$ there is an equivalence of categories between the category of finite torsors $Y \to X$ equipped with a $k$-point $y \in Y(k)$ mapping to $x$ and the category of finite group schemes $G$ over $k$ equipped with a $k$-group scheme homomorphism $\pi_1^N(X,x) \to G$.
\end{proof}
	
\section{Proper splinters have negative Kodaira dimension}
	\label{S:splitners_neg_kod_dim}
	Let $X$ be a Gorenstein projective scheme over a field $k$ of positive characteristic. Since $-K_X$ big implies that $X$ has negative Kodaira dimension,
	it is expected in view of \cref{conj:splinter_big_anticanonical_class} that, if $X$ is a splinter, then its Kodaira dimension is negative. In this section, we confirm this expectation (without assuming $X$ to be Gorenstein) and prove Theorem~\ref{item:thm:negative_kod_dim}.\medskip
	
	Let $X$ be a normal proper variety over a field $k$ and let $D$ be a Weil divisor on $X$. We define the \emph{Iitaka dimension} of $D$ to be
	$$ \kappa(X, D)\coloneqq \min \left \{r \in \bbZ_{\geq 0}  \mid (h^0(X, \mcO_X(dD)) / d^r)_{d\geq 0} \text{ is bounded} \right \}. $$
	By convention, if  $h^0(X, \mcO_X(dD)) \coloneqq \dim_k H^0(X, \mcO_X(dD))=0$ for all $d>0$,  we set $\kappa(X, D)=-\infty$. Beware that we deviate from usual conventions as the Iitaka dimension is usually defined for invertible sheaves on projective varieties. 
	If $X$ is a smooth projective variety over $k$, $\kappa(X, K_X)$ agrees with the Kodaira dimension of $X$.
	The following proposition refines the observation from \cref{lem:splinter_trivial_canonical_sheaf} showing that if $X$ is a proper splinter in positive characteristic, then $K_X$ is not effective.
	\begin{theorem}[Theorem~\ref{item:thm:negative_kod_dim}]
		\label{prop:K_equvialenceodaira_dim_f_split_variety}
		Let $X$ be a  positive-dimensional connected proper scheme over a field of positive characteristic. 
		If $X$ is a splinter, 
		then $\kappa(X, K_X)=-\infty$.
	\end{theorem}

	First we have the following variant of a well-known lemma\,; see, e.g., \cite[Ex.~2.12]{patakfalvi_schwede_tucker_positive_char_alg_geometry}. 
	
	\begin{lemma}\label{lem:kodaira_dim_F_split}
		Let $X$ be a proper scheme over a field $k$ of positive characteristic $p>0$.
		Assume either of the following conditions\,:
		\begin{enumerate}[label=\rm{(\alph*)}]
			\item \label{item:X_splinter} $X$ is a splinter.
			\item \label{item:X_F_split} $X$ is normal and $F$-split.
		\end{enumerate}
		Then the Weil divisor $(1-p)K_X$ is effective. In particular,
		 either  $\kappa(X, K_X)=-\infty$, or $K_X$ is torsion (in which case $\kappa(X, K_X)=0$).
	\end{lemma}
	\begin{proof}
		First assume that $X$ is a splinter.
		Since $X$ is normal, we can and do assume that $X$ is connected. After replacing the base field $k$ by $H^0(X, \mcO_X)$, we consider the base change $\pi \colon X_{\bar{k}} \to X$ along an algebraic closure $k \to \bar{k}$.
		By \cref{lem:splinter_base change}, $X_{\bar{k}}$ is a splinter.
		The base change formula for the exceptional inverse image \cite[\href{https://stacks.math.columbia.edu/tag/0E9U}{Tag 0E9U}]{stacks-project} shows that $\pi^* \omega_X = \omega_{X_{\bar{k}}}$.
		Thus, it is enough to show the statement for $X_{\bar{k}}$, since $\pi$ flat implies $H^0(X_{\bar{k}}, \pi^* \mcO_X((1-p)K_X)) = H^0(X,  \mcO_X((1-p)K_X)) \otimes_k \bar{k}$.
		Since $\bar{k}$ is $F$-finite, $X_{\bar{k}}$ is in particular $F$-split. Thus, it is enough to show the statement under the assumptions \labelcref{item:X_F_split}.

			Assume that $X$ is normal and $F$-split.
			Parts of the arguments below can for example be found in \cite[\S 4.2]{schwede_smith_globally_f_regular_and_log_Fano_varieties}. We provide nonetheless a proof for the sake of completeness.
			Choosing a (non-canonical) isomorphism $F^! k = \Hom_k (F_* k, k) \cong k $ as $k$-vector spaces, we obtain an isomorphism $F^!\omega_X^\bullet \cong \omega_X^\bullet$ in $\sfD_{\mathrm{Coh}}(\mcO_X)$ and, by restricting to $X_{\mathrm{reg}}$ and taking cohomology sheaves, an isomorphism $F^!\omega_{X_{\mathrm{reg}}} \cong \omega_{X_{\mathrm{reg}}}$ in $\Coh({X_{\mathrm{reg}}})$.
			Since the involved sheaves are reflexive, $F^!\omega_X \cong \omega_X$ holds in $\Coh(X)$ and we obtain $\sheafhom_{\mcO_X} (F_*\mcO_X, \omega_X) \cong F_* \omega_X$.

		By assumption, there exists a map $s$ such that the composition
		$$\mcO_X \to F_* \mcO_X \xrightarrow{s} \mcO_X$$
		is the identity.
		We apply $\sheafhom_{\mcO_X}(-, \omega_X)$ to obtain
		$$\omega_X \leftarrow F_* \omega_X \xleftarrow{s^\vee} \omega_X,$$
		where the composition is the identity.
		After restricting to the regular locus, we can twist with $\omega_{X_{\mathrm{reg}}}^{-1}$ and apply the projection formula to obtain a diagram	
		$$\mcO_{X_{\mathrm{reg}}} \leftarrow F_* \mcO_{X_{\mathrm{reg}}}  ((1-p)K_{X_{\mathrm{reg}}} )\xleftarrow{s^\vee} \mcO_{X_{\mathrm{reg}}},$$
		where the composition is the identity. Using that the involved sheaves are reflexive, we obtain a nonzero global section of $F_* \mcO_X((1-p)K_X)$. 
		This gives a nonzero element of $H^0(X, \mcO_X((1-p)K_X))$, hence  $(1-p)K_X$ is effective.
		If no positive multiple of $K_X$ is effective, then $\kappa(X, K_X)=-\infty$. 
		So assume that $nK_X$ is effective for some $n>0$. 
		 Since $n(p-1)K_X$ and $ n(1-p)K_X$ are both effective, $n(1-p)K_X$ is trivial by \cref{lem:torsion_weil_divisor}.
	\end{proof}
	
	\begin{remark}
	Note that $F$-split varieties
may have trivial canonical divisor.
	 For instance, ordinary elliptic curves and ordinary $K3$ or abelian surfaces are $F$-split\,; see \cite[Rmk.~7.5.3(i)]{brion_kumar_frobenius_spliting_methods_in_gemoetry_and_representation_theory}.
	\end{remark}

	In order to prove \cref{prop:K_equvialenceodaira_dim_f_split_variety}, it remains to show that the canonical divisor of a proper
	 splinter is not torsion. First we note that the Picard group of a proper splinter is torsion-free\,; this is a small generalization of a result of 
	Carvajal-Rojas \cite[Cor.~5.4]{carvajal_rojas_finite_torsors_over_strongly_f_regular_singularities} who showed that a globally $F$-regular projective variety has torsion-free Picard group. In particular, this provides a proof of \cref{prop:K_equvialenceodaira_dim_f_split_variety} 
	if $K_X$ is Cartier, e.g., if $X$ is in addition assumed to be 
Gorenstein.

	\begin{proposition}\label{prop:pic_torsion_free}
		Let $X$ be a proper scheme over a field of positive characteristic. If $X$ is a splinter, then $\mathrm{Pic}(X)$ is torsion-free.
	\end{proposition}
	\begin{proof}
		We argue as in \cite[Rmk.~5.6]{carvajal_rojas_finite_torsors_over_strongly_f_regular_singularities} which is concerned with the globally $F$-regular case. If $\mcL$ is a torsion invertible sheaf, then it is in particular semiample and therefore if $X$ is a splinter, then we have $\chi(X, \mcL)=h^0(X, \mcL)$ by \cite[Prop.~7.2]{bhatt_derived_splinters_in_positive_characteristic} (recalled in \cref{prop:semiample}). But then $0=\chi(X, \mcL)=\chi(X, \mcO_X)=1$ if $\mcL$ is nontrivial. This is impossible.
		
		Alternately, one can argue using \cref{thm:splinters and covers} as follows.
		An $n$-torsion invertible sheaf $\mcL$ on $X$ gives rise to a nontrivial $\mu_n$-torsor $\pi \colon Y \to X$, where $Y$ is defined to be the relative spectrum of the finite $\mcO_X$-algebra $\mcO_X \oplus \mcL \oplus \dots \oplus \mcL^{ n-1}$.
		If $n>1$, we must have $H^0(X, \mcL)=0$, since any nontrivial section $s\colon \mcO_X \to \mcL$ would also give a nontrivial section $s^{ n}$ of $\mcO_X$, so $s$ would be nowhere vanishing and therefore $\mcL$ would be trivial.
		This yields the equality $$H^0(Y, \mcO_Y)=H^0(X, \pi_* \mcO_Y) = H^0(X, \mcO_X \oplus \mcL \oplus \dots \oplus \mcL^{ n-1})=H^0(X, \mcO_X).$$
		We conclude from \cref{thm:splinters and covers} that if $X$ is a splinter, then $\deg(\pi) = 1$, i.e., $n=1$ and $\mcL$ is trivial.
	\end{proof}
	
	To deal with the non-Gorenstein case, we have\,:
	
	\begin{proposition}\label{prop:splinter_dualizing_sheaf_not_torsion}
		Let $X$ be a connected positive-dimensional proper scheme over a field of positive characteristic. 
		If $X$ is a splinter, then
		the canonical divisor $K_X$ is not torsion.
	\end{proposition}
	\begin{proof} 
		Assume for contradiction that $K_X$ is torsion of order $r$, i.e., that $\omega_X \vert_{X_\mathrm{reg}}$ is a torsion invertible sheaf of order $r$. 
		By considering the relative spectrum, we obtain a $\mu_r$-quasi-torsor 
		$$\pi \colon Y = \Spec_X \left (\bigoplus_{i=0}^{r-1} \mcO_X(iK_X) \right ) \to X ,$$
		which, over $X_{\mathrm{reg}}$, restricts to a $\mu_r$-torsor $\pi|_U\colon V \coloneqq \pi^{-1}(X_{\mathrm{reg}}) \to X_{\mathrm{reg}} \eqqcolon U$.

	In addition, $\pi_* \mcO_Y = \bigoplus_{i=0}^{r-1} \mcO_X(iK_X)$, and by \cref{lem:torsion_weil_divisor}  the sheaves $\mcO_X(iK_X)$ have no nonzero global sections for $1 \leq i \leq r-1$.
		Thus $H^0(V,\mcO_V) = H^0(X_{\mathrm{reg}}, \mcO_{X_{\mathrm{reg}}}) = H^0(X, \mcO_{X})$, where the second equality holds by normality of $X$, is a field and we conclude by \cref{prop:torsor}\labelcref{item:i} or~\labelcref{ii} that $V$ is a splinter.
		In particular, $V$ is normal and therefore, by, e.g.,
\cite[\href{https://stacks.math.columbia.edu/tag/035E}{Tag 035E}]{stacks-project}, 
 the normalization $Y^\nu \to Y$ is an isomorphism over $V$. Since normalization is finite, $Y^\nu \setminus V$ has codimension $\geq 2$ and thus $Y^\nu$ is a splinter by \cref{lem:splinter_invariant_codim_2_surgery}.

		Now $\pi^!_U \mcO_{X_{\mathrm{reg}}} = \mcO_V$ holds by \cref{lem:covers} and this implies $\pi^*_U \omega_{X_{\mathrm{reg}}} = \omega_V $.
		On the other hand, we have isomorphisms
		$$(\pi_U)_* \pi_U^* \omega_{X_{\mathrm{reg}}} \cong \omega_{X_{\mathrm{reg}}} \otimes_{\mcO_{X_{\mathrm{reg}}}} \bigoplus_{i=0}^{r-1} \omega_{X_{\mathrm{reg}}}^i  = \bigoplus_{i=1}^{r} \omega_{X_{\mathrm{reg}}}^i \cong \bigoplus_{i=0}^{r-1} \omega_{X_{\mathrm{reg}}}^i 
		\cong (\pi_U)_*\mcO_V$$
		as $(\pi_U)_*\mcO_V$-modules.
		Hence $V=\Spec_{X_{\mathrm{reg}}} \bigoplus_{i=0}^{r-1} \omega_{X_{\mathrm{reg}}}^i $ has trivial dualizing sheaf. 
		Consequently, since $\omega_{Y^\nu}$ is a reflexive sheaf, $Y^\nu$ has trivial dualizing sheaf. But, by \cref{lem:splinter_trivial_canonical_sheaf}, a proper splinter cannot have trivial canonical sheaf.
	\end{proof}

	\begin{proof}[Proof of \cref{prop:K_equvialenceodaira_dim_f_split_variety}]
	 	The canonical divisor $K_X$ is not torsion by \cref{prop:splinter_dualizing_sheaf_not_torsion} (or more simply by \cref{prop:pic_torsion_free} if $K_X$ is Cartier, e.g.\ if $X$ is Gorenstein), and it follows from \cref{lem:kodaira_dim_F_split} that $\kappa(X, K_X)=-\infty$.
	\end{proof}

	\begin{remark}[Torsion Weil divisors on splinters] Proper splinters may have nontrivial torsion Weil divisor classes. 
	Indeed, projective toric varieties are globally $F$-regular and in particular splinters, and Carvajal-Rojas \cite[Ex.~5.7]{carvajal_rojas_finite_torsors_over_strongly_f_regular_singularities} gives an example of a projective toric surface that admits a nontrivial $2$-torsion Weil divisor class.
	\end{remark}
	
	\section{Vanishing of global differential forms}
	\label{S:vanishing_global_1_forms}
	Fix a perfect field $k$ of positive characteristic $p$ and let $X$ be a smooth proper variety over $k$. Let $\Omega_{X/k}^\bullet$ be the de~Rham complex and recall, e.g., from \cite[Thm.~7.2]{katz_nilpotent_connections_and_the_monodromy_theorem}, that there exists an isomorphism of graded $\mcO_X$-modules
		$$C^{-1}_X\colon \bigoplus_{j \geq 0} \Omega_X^j \to  \bigoplus_{j \geq 0} \mcH^j (F_* \Omega_X^\bullet),$$
		whose inverse $C_X$ is the so-called \emph{Cartier operator}. It gives rise for all $j\geq 0$ to short exact sequences of $\mcO_X$-modules
	\begin{equation}\label{eq:short_exact_sequence_cycles_boundaries_cartier_operator}
		0 \to B_X^j \to Z_X^j \xrightarrow{C_X^j} \Omega_X^j \to 0,
	\end{equation}
	where $B_X^j$ denotes the $j$-th coboundaries and $Z_X^j$ the $j$-th cocycles of $F_*\Omega_X^\bullet$.
	Note that these coincide with the image under $F_*$ of the coboundaries and cocycles of $\Omega_X^\bullet$.
	Moreover, there is a short exact sequence
	$$ 0 \to \mcO_X \to F_*\mcO_X \to B_X^1 \to 0.$$	
	The proof of the following theorem is inspired by the proof of \cite[Lem.~6.3.1]{achinger_witaszek_zdanowics_global_frobenius_liftability_i}.
	\begin{theorem}[Theorem~\ref{item:thm:global-1-forms}]
		 \label{thm:global-1-forms}
		Let $X$ be a smooth proper variety over a  
		 field $k$ of positive characteristic $p$. If $X$ is a splinter, then $H^0(X, \Omega_X^1)=0$.
	\end{theorem}
	\begin{proof} 
		Clearly, we may and do assume that $X$ is connected.
		Let $\bar k$ be an algebraic closure of $k$. It is enough to show that $H^0(X_{\bar k}, \Omega_{X_{\bar k}}^1)=0$. Since $X$ is in particular geometrically reduced over $k$, $H^0(X,\mcO_X)$ is a finite separable extension of $k$\,; see, e.g., \cite[\href{https://stacks.math.columbia.edu/tag/0BUG}{Tag 0BUG}]{stacks-project}. It follows that $X_{\bar k}$ is the disjoint union of $\dim_k H^0(X,\mcO_X)$ copies of $X\times_{H^0(X,\mcO_X)} \bar k$.
		From \cref{cor:splinter_bc}, we find that $X_{\bar k}$ is a splinter.
		Therefore it is enough to establish the theorem in case $k$ is algebraically closed. So assume $k$ is algebraically closed, in which case the $p$-th power map $k=H^0(X, \mcO_X) \to H^0(X, F_* \mcO_X)=k$ is an isomorphism. 
		Since by \cref{prop:semiample}, for a splinter~$X$, $\mcO_X$ has zero cohomology in positive degrees, the long exact sequence associated to 
		$$ 0 \to \mcO_X \to F_*\mcO_X \to B_X^1 \to 0$$
		gives that $H^j(X, B_X^1 )=0$ for all $j\geq 0$.
		From the long exact sequence associated to \cref{eq:short_exact_sequence_cycles_boundaries_cartier_operator}, we find that the Cartier operator induces an isomorphism $H^0(X, Z_X^1) \to H^0(X, \Omega_X^1)$.
		In particular, $\dim H^0(X, Z_X^1) =\dim  H^0(X, \Omega_X^1)$.
		The inclusion of closed $1$-forms $\ker (d\colon \Omega_X^1\to \Omega_X^2) \subseteq \Omega_X^1$ yields an injection
		$H^0(X, \ker (d\colon \Omega_X^1\to \Omega_X^2) ) \subseteq H^0(X, \Omega_X^1)$.
		Since $H^0(X, \ker (d\colon \Omega_X^1\to \Omega_X^2) ) = H^0(X, Z_X^1)$, the above inclusion is in fact an equality.
		In other words, any global $1$-form on $X$ is closed.
		By \cite[Prop.~4.3]{van_der_geer_katsura_on_the_height_of_calabi_yau_varieties_in_positive_characteristic}, there exists an isomorphism
		$$ H^0(X, \Omega_X^1) \cong \Pic(X)[p] \otimes_\bbZ k,$$
		where $\Pic(X)[p]$ denotes the $p$-torsion invertible sheaves on $X$.
		By \cref{prop:pic_torsion_free}, $\Pic(X)$ is torsion-free.
	\end{proof}

\section{On the splinter property for proper surfaces}
	\label{S:splinters_surfaces}
		A proper curve over an algebraically closed field of positive characteristic is a splinter if and only if it is isomorphic to the projective line. In this section, we investigate which proper surfaces over an algebraically closed field of positive characteristic are splinters. First we show in  \cref{prop:splinter_surface_rational} that proper surface splinters are rational. We then show in \cref{prop:blow_up_of_points_on_a_line_or_conic} that the blow-up of the projective plane in any number of closed points lying on a given conic is a splinter. On the other hand, we give examples of rational surfaces that are not splinters in \cref{sec:surfaces_which_are_not_splinters}.

	\subsection{Proper splinter surfaces are rational}
	The fact proved in \cref{prop:K_equvialenceodaira_dim_f_split_variety} that proper splinters in positive characteristic have negative Kodaira dimension can be used to show that proper surface splinters over an algebraically closed field of positive characteristic are rational\,:
	
		\begin{proposition}\label{prop:splinter_surface_rational}
			Let $X$ be an irreducible proper surface over an algebraically closed field of positive characteristic.
			If $X$ is a splinter,
then $X$ is rational.
		\end{proposition}
		\begin{proof} 
			
			If $X$ is not smooth, choose a resolution of singularities $\pi\colon \tilde{X} \to X$ such that $\pi$ is an isomorphism over the regular locus $X_{\mathrm{reg}}$, which exists by \cite[\S 2]{lipman_rational_singularities_with_applications_to_algebraic_surfaces_and_unique_factorization}.
			It suffices to show that $\tilde X$ is rational.
			Note that Grauert--Riemenschneider vanishing holds for surfaces, see, e.g., \cite[\href{https://stacks.math.columbia.edu/tag/0AXD}{Tag 0AXD}]{stacks-project}. Thus, the proof of \cite[Thm.~2.12]{bhatt_derived_splinters_in_positive_characteristic}
			shows that the above resolution of singularities satisfies $\mathrm{R} \pi_* \mcO_{\tilde{X}} = \mcO_X$, that is, that $X$ has rational singularities.
			Therefore
			$$\chi(\mcO_{\tilde{X}})=\chi(\mathrm{R} \pi_* \mcO_{\tilde{X}}) = \chi(\mcO_X)=1.$$
			By Castelnuovo's rationality criterion it remains to show that $\omega_{\tilde{X}}^{ 2}$ has no nonzero global sections.
			So assume that there is a nonzero section $s \in H^0(\tilde{X}, \omega_{\tilde{X}}^{ 2})$. Then $s$ is nonzero after restriction to the open subset $\pi^{-1}(X_{\mathrm{reg}})$. Since $\pi$ is an isomorphism over $X_{\mathrm{reg}}$, the section $s$ would 
			provide a nonzero section of $\mcO_X(2K_X)$, contradicting \cref{prop:K_equvialenceodaira_dim_f_split_variety}.
		\end{proof}

	\subsection{Examples of projective rational surfaces that are splinters}
We give examples of projective rational surfaces that are splinters\,; in all cases, this is achieved by showing that they are globally $F$-regular. We start with already known examples.

		\begin{example}[Del Pezzo surfaces, {\cite[Ex.~5.5]{hara_a_characterization_of_rational_sing_in_terms_of_injectivity_of_frob_maps}}]
			\label{ex:globFreg}
				Let $X$ be a smooth projective del Pezzo surface over an algebraically closed field of characteristic $p>0$. 
				 Then $X$ is globally $F$-regular if one of the following conditions holds\,:
				\begin{enumerate}
					\item $K_X^2>3$.
					\item $K_X^2=3$ and $p>2$.
					\item $K_X^2=2$ and $p>3$.
					\item $K_X^2=1$ and $p>5$.
				\end{enumerate}
				Moreover, if none of the above conditions are satisfied, there are globally $F$-regular and non globally $F$-regular cases\,; for instance, the Fermat cubic surface in characteristic 2 is not globally $F$-regular.
		\end{example}
	\begin{example}[Hirzebruch surfaces, {\cite[Prop.~3.1]{gongyo_takagi_surfaces_of_globally_f_regular_and_f_split_type}}]
		If $X$ is a Hirzebruch surface $\bbP(\mcO_{\bbP^1} \oplus \mcO_{\bbP^1}(-n))$ over a perfect field of positive characteristic, then $X$ is globally $F$-regular. 
	\end{example}

	To the above list of examples, we can add blow-ups of $\bbP^2$ in any number of points lying on a conic\,:

	\begin{proposition}\label{prop:blow_up_of_points_on_a_line_or_conic}
		Let $k$ be an algebraically closed
		 field of characteristic $p>0$ and let $p_1, \dots, p_n$ be distinct closed points in $\bbP^2_k$.
		Assume either of the following conditions\,:
		\begin{enumerate}[label=\rm{(\alph*)}]
			\item The points $p_1, \dots, p_n$ lie on a line $L \subseteq \bbP^2_k$.
			\item The points $p_1, \dots, p_n$ lie on a (possibly singular) conic $C \subseteq \bbP^2_k$.
		\end{enumerate}
		then the blow-up $X\coloneqq \mathrm{Bl}_{\{p_1, \dots, p_n\}}\bbP^2_k$ is globally $F$-regular.
		In particular, the blow-up of $\bbP^2_k$ in at most $5$ closed points is globally $F$-regular.
	\end{proposition}

In some sense \cref{prop:blow_up_of_points_on_a_line_or_conic} is optimal since by \cref{ex:globFreg} the Fermat cubic surface, which is the blow-up of $\bbP^2_k$ in 6 points, is not globally $F$-regular if $\operatorname{char} k =2$.
	The proof adapts a strategy which is similar to the arguments of \cite[\S\S 1.3-1.4]{brion_kumar_frobenius_spliting_methods_in_gemoetry_and_representation_theory}. 
	First we  determine, under the isomorphism $\Hom_X(F_*\mcO_X, \mcO_X) \cong H^0(X, \mcO_X((1-p)K_X))$, which global sections correspond to splittings of $\mcO_X \to F_*\mcO_X$ in the case where $X=\bbP^n_k$. 
For the sake of completeness, we state and prove the following lemma, which appears as an exercise in \cite{brion_kumar_frobenius_spliting_methods_in_gemoetry_and_representation_theory}.
	
	\begin{lemma}[{cf.~\cite[Ex.~1.3.E(1)]{brion_kumar_frobenius_spliting_methods_in_gemoetry_and_representation_theory}}]\label{lem:identification_of_evaluation_map_projective_space}
		Let $X =\bbP^n_k = \Proj k[X_0, \dots, X_n]$ for a field $k$ of characteristic $p>0$. Let $F^{e\sharp} \colon \mcO_X \to F_*^e \mcO_X$ be the canonical map and consider the following chain of isomorphisms
		$$\Phi \colon \Hom_X(F_*^e \mcO_X , \mcO_X) \xrightarrow{\cong} \Ext^n_X(\mcO_X, (F_*^e \mcO_X) \otimes \omega_X)^\vee \xrightarrow{\cong} H^n(X, \omega_X^{p^e})^\vee \xrightarrow{\cong} H^0(X, \omega_X^{1-p^e})$$
		where the first and last isomorphism are given by Serre duality and the second isomorphism follows from the projection formula.
		Then the following diagram commutes
		\[
		\begin{tikzcd}
			\Hom_X(F_*^e \mcO_X, \mcO_X) \rar{-\circ F^{e\sharp}}\dar{\Phi}& \Hom_X(\mcO_X, \mcO_X) \dar{\mathrm{ev}_1}\\
			H^0(X, \mcO_X(\omega_X^{1-p^e}))\rar{\tau} & k,
		\end{tikzcd}
		\]
		where $\mathrm{ev}_1$ is the evaluation at the constant global section $1$ and $\tau$ is the map sending a homogeneous polynomial $P \in H^0(X, \mcO_X((n+1)(p^e-1)))$ of degree $(n+1)(p^e-1)$ to the coefficient of the monomial $(X_0\cdots X_n)^{(p^e-1)}$ in $P$. 
		In particular, $\Phi^{-1} (P)$ provides a splitting of $F^{e\sharp}$ if and only if $\tau(P)=1$.
	\end{lemma}
	\begin{proof}
		Recall, e.g., from \cite[Ch.~III, Thm.~5.1 \& Thm.~7.1]{hartshorne_algebraic_geometry},
		that Serre duality for invertible sheaves on $\bbP^n$ is given by the bilinear form
		\begin{align*}
			H^0(X, \mcO_X(a))\otimes H^n(X, \mcO_X(-(n+1)-a)) &\to H^n(X, \mcO_X(-(n+1))) = k\frac{1}{X_0 \cdots X_n} \\
			(P, Q) &\mapsto \text{coefficient of }(X_0 \cdots X_n)^{-1} \text{ in } PQ,
		\end{align*}
	where,  for $b<0$, we identify $H^n(X, \mcO_X(b))$ with the degree $b$ part of the negatively graded $k$-algebra $(X_0 \cdots X_n)^{-1}k[X_0^{-1}, \dots, X_n^{-1}]$.
		Consider the following commutative diagram of isomorphisms
		\[
		\begin{tikzcd}
			& \Hom_X(F_*^e \mcO_X, \mcO_X) \rar{-\circ F^{e\sharp}}\dar{\mathrm{SD}} & \Hom_X(\mcO_X, \mcO_X) \dar{\mathrm{SD}} \rar{\mathrm{ev}_1}& k\\
			H^0(X, \mcO_X((1-p^e)K_X)) \rar{\mathrm{SD}} & H^n(X, \mcO_X(p^e K_X )) ^\vee \rar{({F^e}^*)^\vee} & H^n(X, \mcO_X(K_X))^\vee ,&
		\end{tikzcd}
		\]
		where SD stands for Serre duality.
		Since ${F^e}^*\colon H^n(X,\mcO_X(-(n+1))) \to H^n(X,\mcO_X(-p^e(n+1)))$ raises a polynomial to its $p^e$-th power,
we find  that a monomial in $H^0(X, \mcO_X((1-p^e)K_X))$ is sent to $1$ in $k$ along the above diagram if it is $(X_0\cdots X_n)^{p^e-1}$ and to zero otherwise.
	\end{proof}
	
	For $X$ the blow-up of $\bbP^2_k$ in $n$ distinct closed points $p_1, \dots, p_n \in \bbP^2_k$, we denote by $E_i$ the exceptional curve over the point $p_i$ and we let $H \in \Pic(X)$ be the pullback of the class of a hyperplane in $\bbP^2_k$.
	Then
	$$\Pic(X) = \bbZ H \oplus \bbZ E_1 \oplus \dots \oplus \bbZ E_n,$$
	is an orthogonal decomposition with respect to the intersection pairing, and we have $H^2=1$ and $E_i^2=-1$ for all $1 \leq i \leq n$.
	The canonical class of $X$ is $K_X = -3H+\sum_i E_i$.
	If $C \subseteq \bbP^2_k$ is an irreducible curve, its strict transform	$\tilde{C}\subseteq X$ has class
	\begin{equation*}
		\tilde{C}= dH-\sum_{i=1}^n m_i E_i \in \Pic(X),
	\end{equation*}
	where $d$ is the degree of $C$ and $m_i$ is the multiplicity of $C$ at $p_i$.
	Any irreducible curve in $X$ is either one of the exceptional curves $E_i$, or the strict transform of an irreducible curve in $\bbP^2_k$.
	Note that for $d>n$, the divisor $dH-\sum_i E_i$ is ample, since it has positive square and since the intersection with any integral curve in $X$ is positive.
	
	\begin{proof}[Proof of \cref{prop:blow_up_of_points_on_a_line_or_conic}]
		Since $X$ is smooth, it suffices  by
		\cref{thm:schwede_smith_f_reg_one_divisor} (see \cref{rmk:SS-criterion}) to prove that there exists an ample divisor $D$ on $X$ such that  $\mcO_X \to F_*^e\mcO_X(D)$ splits for some $e>0$.

		Let $d>0$ be an integer such that the divisor $D\coloneqq dH-\sum_i E_i$ on $X$ is ample and fix a global section $\sigma \colon \mcO_X \to \mcO_X(D)$. We can interpret $\sigma$ as a homogeneous polynomial of degree $d$ vanishing at the points $p_1, \dots, p_n$. 
		Now, as in \cref{lem:identification_of_evaluation_map_projective_space}, we have an isomorphism 
		$$\Psi \colon \Hom_X(F_*^e\mcO_X(D), \mcO_X) \xrightarrow{\cong} H^0(X, \mcO_X((1-p^e)K_X - D))$$
		and global sections of $\mcO_X((1-p^e)K_X - D)$ correspond again to polynomials of a certain degree, vanishing to some certain order at the points $p_i$. The following claim is similar to \cite[p.~39]{brion_kumar_frobenius_spliting_methods_in_gemoetry_and_representation_theory}.		
		\begin{claim}
			A section $\varphi \in H^0(X, \mcO_X((1-p^e)K_X - D))$ defines a section of $\mcO_X \to F_*^e\mcO_X(D)$ if and only if $\Psi^{-1}(\varphi) \sigma \in H^0(X, \mcO_X((1-p^e)K_X))$ defines a splitting of $\mcO_X \to F_*^e \mcO_X$, where $\Psi^{-1}(\varphi)\sigma$ is the usual product of polynomials.
		\end{claim}
		\begin{proof}[Proof of the claim]
			A map $\varphi \in \Hom_X(F_*^e\mcO_X(D), \mcO_X)$ is a section of $\mcO_X \to F_*^e\mcO_X(D)$ if and only if $\varphi \circ F_*^e(\sigma)$ is a section of $\mcO_X \to F_*^e\mcO_X$. Thus, we have to check that the composition $\varphi \circ F_*^e(\sigma)$ corresponds to the product $\Psi(\varphi)\sigma$. This is done by verifying, that the following diagram commutes
			\[
			\begin{tikzcd}
				\Hom_X(F_*^e \mcO_X(D), \mcO_X) \arrow{rrr}{-\circ F_*^e(\sigma)} \dar{\mathrm{SD}} &&& \Hom_X(F_*^e \mcO_X, \mcO_X) \dar{\mathrm{SD}} \\
				H^n(X, \omega_X \otimes F_*^e \mcO_X(D)) ^\vee \arrow{rrr}{(\omega_X \otimes F_*^e(\sigma))^\vee} \dar{\text{projection formula}} &&& H^n(X, \omega_X \otimes F_*^e \mcO_X)^\vee \dar{\text{projection formula}} \\
				H^n(X, F_*^e \mcO_X (p^eK_X + D)) ^\vee  \arrow{rrr}{(F_*^e( \mcO_X(p^e K_X) \otimes \sigma))^\vee} \dar &&& H^n(X, F_*^e \mcO_X(p^e K_X))^\vee \dar \\
				H^n(X, \mcO_X(p^e K_X +D)) ^\vee \arrow{rrr}{(\mcO_X(p^e K_X) \otimes \sigma)^\vee} \dar{\mathrm{SD}} &&& H^n(X, \mcO_X(p^e K_X))^\vee \dar{\mathrm{SD}} \\
				H^0(X, \mcO_X((1-p^e)K_X-D) )  \arrow{rrr}{-\cdot \sigma} &&& H^0(X, \mcO_X((1-p^e)K_X)). \qedhere
			\end{tikzcd}
			\]				 
		\end{proof}
		Denote by $\mu \colon X \to \bbP^2$ the blow-up map. Since $\mu_*\mcO_X = \mcO_{\bbP^2}$, $\mu_*$ induces an isomorphism $\End(\mcO_X) \to \End(\mcO_{\bbP^2})$.
		Since any morphism of schemes commutes with the Frobenius, a map $\varphi \colon F_*^e \mcO_X \to \mcO_X$ is a splitting of $\mcO_X \to F_*^e \mcO_X$ if and only if $\mu_*(\varphi)$ is a splitting of $\mcO_{\bbP^2} \to F_*^e \mcO_{\bbP^2}$.
		The proposition will follow if we can find suitable polynomials $\varphi \in H^0(X, \mcO_X((1-p^e)K_X -D))$ and $\sigma\in H^0(X,\mcO_X(D))$ such that $\mu_*(\varphi \sigma)$ defines a splitting of $\mcO_{\bbP^2} \to F_*^e\mcO_{\bbP^2}$.
		In terms of \cref{lem:identification_of_evaluation_map_projective_space}, the monomial $(XYZ)^{p^e-1}$ has to occur with coefficient $1$ in $\varphi \sigma$ (here we are using coordinates $\bbP^2_k  = \Proj k[X,Y,Z]$).
		
		We first compute
		$$(1-p^e)K_X -D= 3(p^e-1)H-(p^e-1)\sum_i E_i - dH + \sum_i E_i=(3(p^e-1)-d)H-(p^e-2)\sum_i E_i.$$
		If all the points $p_i$ lie on a line $L$, we can assume without loss of generality that $L=V(Z)$.
		Consider the polynomials $\tilde{\varphi}\coloneqq X^{(p^e-1)-(d-1)}Y^{p^e-1}$ and $\tilde{\sigma}\coloneqq X^{d-1}$.
		Moreover, if we set $\varphi\coloneqq \tilde{\varphi}Z^{p^e-2}$ and $\sigma\coloneqq \tilde{\sigma}Z$, then $\varphi \in H^0(X, \mcO_X((1-p^e)K_X -D))$ and $\sigma \in H^0(X, \mcO_X(D))$ and the coefficient of $(XYZ)^{p^e-1}$ in $\varphi \sigma$ is $1$.
		
		If the points lie on a conic $C$, we can assume after possible change of coordinates
	 	that the conic is given by an equation of the form $XY-Z^2$, $XY$, or $X^2$\,; see the elementary \cref{lem:conics} below.
		In the last case, the points lie on the line $X=0$, thus we may assume that $C$ is given by one of the equations $\tilde{\sigma}=XY-Z^2$ or $\tilde{\sigma}=XY$.
		Now set $\sigma \coloneqq  Z^{d-2}\tilde{\sigma}$ and $\varphi\coloneqq Z^{(p^e-1)-(d-2)}\tilde{\sigma}^{p^e-2}$ and observe that $(XYZ)^{p^e-1}$ occurs with coefficient $1$ in $\varphi \sigma$.
	\end{proof}
	
	\begin{remark}\label{rmk:blow-up_on_ordinary_elliptic_curve_f_split}
		Recall, e.g.~from \cite[Ch.~IV, Prop.~4.21]{hartshorne_algebraic_geometry}, that an elliptic curve $C = V(P) \subseteq \bbP_k^2$ is \emph{ordinary} if $(XYZ)^{p-1}$ occurs with nonzero coefficient in $P^{p-1}$. The above \cref{lem:identification_of_evaluation_map_projective_space} can also be used to show, similarly to \cite[Rmk.~6.3]{hara_looking_out_for_frobenius_summands_on_blown_up+surface_of_P2}, that the blow-up of $\bbP^2_k$ in any number of points, which lie on an ordinary elliptic curve, is $F$-split.
	\end{remark}

	\begin{lemma} \label{lem:conics}
		Let $k$ be an algebraically closed field. After suitable coordinate transform a conic in~$\bbP^2_k$ is given one of the following equations\,:
		$$XY-Z^2,\ XY, \ \text{or}\; X^2.$$
	\end{lemma}
	\begin{proof}
	 First recall that $\mathrm{PGL}_3(k)$ acts $2$-transitively on $\bbP^2_k = \Proj k[X,Y,Z]$. In particular, $\mathrm{PGL}_3(k)$ acts also $2$-transitively on the set of lines in $\bbP^2_k$, which is $\bbP(H^0(\bbP_k^2, \mcO_{\bbP^2_k}(1))) =\bbP (k[X,Y,Z]_{\deg 1})$.
		If a conic $C \subseteq \bbP^2_k$ is reducible, then it is either a double line or the union of two lines and we can assume $C$ to be defined by the equations $X^2=0$ or $XY=0$, resp.
		It remains to show that an integral conic $C$ is defined by the equation $XY-Z^2=0$ after suitable coordinate transformation. 	We follow the arguments of \cite[Cor.~3.12]{kirwan_complex_algebraic_curves}.
		Since $C$ has only finitely many singular points, we can assume without loss of generality that $[0:1:0]$ is a smooth point of $C$ and the line $Z=0$ is tangent to $C$ at $[0:1:0]$. This means that $C$ is given by an equation of the form
		\begin{equation}\label{eq:equation_of_conic}
			aYZ + bX^2 + cXZ +dZ^2,
		\end{equation}
		since the line tangent to $C$ at the point $[0:1:0]$ is precisely the line given by the linear factors in the dehomogenized equation defining $C$ on the affine open $\{Y \neq 0\} \cong \Spec k[X,Z]$.
		By assumption \cref{eq:equation_of_conic} is an irreducible polynomial. This implies that $b\neq 0$ and $a \neq 0$. We conclude by noting that $XY-Z^2$ is mapped to \cref{eq:equation_of_conic} under the coordinate transformation
		\begin{equation*}
			[x,y,z] \mapsto [\sqrt{b}x, ay+cx+dz, -z]. \qedhere
		\end{equation*}
	\end{proof}

	\subsection{Examples of projective rational surfaces that are not splinters}\label{sec:surfaces_which_are_not_splinters}
		In this paragraph, we use Bhatt's \cref{prop:semiample} to show that certain projective rational surfaces are not splinters.
		First, we have the following example of surfaces that are not globally $F$-regular.

		\begin{example}[{\cite[Ex.~6.6]{schwede_smith_globally_f_regular_and_log_Fano_varieties}}]
			\label{Ex:9points} Let $k$ be an algebraically closed field of positive characteristic.
		If $X$ is the blow-up of $\bbP_k^2$ in $9$ closed points in general position, then $-K_X$ is not big. Therefore, by \cref{prop:gFr-big}, $X$ is not globally $F$-regular. Furthermore, this shows that the blow-up of $\bbP^2$ in at least 9 points in general position is not globally $F$-regular, as global $F$-regularity descends along birational morphisms (\cref{lem:gFr_descent}).
		
		\noindent On the other hand, we note that if $X$ is the blow-up of $\bbP^2$ in $9$ points lying on an ordinary elliptic curve, then $X$ is $F$-split\,; see \cref{rmk:blow-up_on_ordinary_elliptic_curve_f_split}.
		Since being ordinary is an open property, it follows that the blow-up of $\bbP^2$ in $9$ points in general position is $F$-split.
	\end{example}
	
	We can extend \cref{Ex:9points}  and show that in some cases  the blow-up of $\bbP^2$ in  9 points is not a splinter\,:

	\begin{proposition}\label{prop:example_9_points_not_splinter}
		Let $X$ be the blow-up of $\bbP_k^2$ in $9$ distinct $k$-rational points.
		Assume either of the following conditions\,:
		\begin{enumerate}
			\item \label{item:first_case} The base field $k$ is the algebraic closure of a finite field and the $9$ points lie on a smooth cubic curve (e.g., the $9$ points are in general position).
			\item \label{item:second_case} The base field $k$ has positive characteristic and the $9$ points lie at  the transverse intersection of two cubic curves in $\bbP^2_k$.
		\end{enumerate}
		Then $X$ is not a splinter.
	\end{proposition}
	\begin{proof}
		In both cases, we show that the anticanonical divisor $-K_X$ is semiample 
		and satisfies $H^1(X,\mcO_X(-nK_X)) \neq 0$ for some $n>0$. It follows from \cref{prop:semiample} that $X$ is not a splinter.
		
		In case \labelcref{item:first_case}, the anticanonical divisor $-K_X$ is the strict transform of the smooth cubic curve and is therefore smooth of genus 1. 
		Since $(-K_X)^2=0$,
		we get from \cite[Thm.~2.1]{totaro_moving_codimension_one_subvarieties_over_finite_fields} that $-K_X$ is semiample, that is, there exists a positive integer $n$ such that $-nK_X$ is basepoint free. In particular, since $K_X$ is not torsion, $h^0(-nK_X)\geq 2$. By Riemann--Roch $\chi(-nK_X)=1$, and hence $h^1(-nK_X)\neq 0$.
		
		In case \labelcref{item:second_case}, $-K_X$ is basepoint free (in particular semiample)  and satisfies $h^0(-K_X) \geq 2$. Indeed, $-K_X$ admits two sections, corresponding to the strict transforms of the cubics, which do not meet in any point in the blow-up as they pass through the $9$ points in $\bbP_k^2$ from different tangent directions. 
		On the other hand, we have $\chi(-K_X) = 1$ and it follows that $h^1(-K_X) \neq 0$.
	\end{proof}

	\begin{remark}
		An alternative proof of \cref{prop:example_9_points_not_splinter}\labelcref{item:second_case} can be obtained by using \cref{lem:generic_fiber_splinter}. 
		Indeed, if the $9$ blown up points lie in the intersection of two distinct smooth cubic curves, the set of cubic curves passing through the $9$ points forms a pencil and we obtain an elliptic fibration $X \to \bbP^1$. Since the generic fiber of this fibration is not a splinter, $X$ is not a splinter.
	\end{remark}
	
	\cref{prop:example_9_points_not_splinter} leaves open the question of whether the blow-up of $\bbP^2_k$ in 9 closed points in very general position over an uncountable algebraically closed field $k$ of positive characteristic is a splinter.
	Finally, further examples of rational surfaces over $\overline{\bbF}_p$ that are not splinters are provided by the following\,:
	
	\begin{proposition}\label{prop:15_points_on_quartic_curve} Let $k$ be the algebraic closure of a finite field. 
		Let $d\geq 4$ be an integer and let $C$ be an irreducible curve of degree $d$ in $\bbP_k^2$.
		Let $n \coloneqq  {2+d \choose 2}=\frac{(d+2)(d+1)}{2}$\,; e.g.,  $n=15$ if $d=4$. 
		The blow-up $X$ of $\bbP_k^2$ in $n$ distinct smooth points of $C$ is not a splinter.
		
	\end{proposition}
	
	\begin{proof}
		By \cref{prop:semiample}, it suffices to construct a semiample invertible sheaf $\mcL$ on $X$ such that $H^1(X,\mcL)\neq 0$. For that purpose, we consider the class $D$ of the strict transform of the curve~$C$. We have
		$D=dH-\sum_{i=1}^n E_i,$
		where $H$ is the pullback of the hyperplane class in $\bbP_k^2$ and $E_i$ are the exceptional curves lying above the $n$ blown up points. We claim that the invertible sheaf $\mcL \coloneqq  \mcO_X(D)$ is semiample and satisfies $H^1(X,\mcL) \neq 0$.
		On the one hand, $D$ is nef since it is effective and satisfies $D^2 = d^2-n >0$ for $d\geq 4$. Being nef and having positive self intersection, it is also big, see, e.g., \cite[Cor.~2.16]{kollar_rational_curves_algebraic_varieties}. 
		By Keel's \cite[Cor.~0.3]{keel_basepoint_freeness_for_nef_and_big_line_bundles_in_positive_characteristic} any nef and big invertible sheaf on a surface over the algebraic closure of a finite field is semiample.
		On the other hand, by Riemann--Roch $$\chi(D)= 1 + \frac{1}{2}(D^2-K_X\cdot D) = 1 + \frac{1}{2}(d^2+3d-2n) = 0,$$
		for $K_X =-3H+\sum_iE_i$ the canonical divisor on $X$.
		Since $D$ is effective, $H^0(X,\mcL) \neq 0$, and we conclude that $H^1(X,\mcL) \neq 0$.
	\end{proof}

	\section{$K$-equivalence, $\mcO$-equivalence, and $D$-equivalence}
	\label{S:O-eq}
	
	The aim of this section is to study the derived invariance of the (derived) splinter property and of global $F$-regularity for projective varieties over a field of positive characteristic. For that purpose, we introduce the notion of $\mcO$-equivalence, which is closely related to $K$-equivalence but offers more flexibility, and explore whether the (derived) splinter property and global $F$-regularity are preserved under (strong) $\mcO$-equivalence.
	
	\subsection{$K$-equivalence}
	\label{SS:K}
	In this paragraph, we fix an excellent Noetherian scheme $S$ admitting a dualizing complex~$\omega_S^\bullet$. Any scheme $X$ over $S$ with structure morphism $h\colon X \to S$ of finite type and separated will be endowed with the dualizing complex $\omega_X^\bullet \coloneqq h^!\omega_S^\bullet$.
	If $X$ is normal,
	there is a unique, up to linear equivalence, Weil divisor $K_X$ on $X$ such that $\omega_X \cong \mcO_X(K_X)$. 
	The following notions are classical, at least in the case of smooth varieties over a field (where they agree)\,:
	
	\begin{definition}[$K$-equivalence and strong $K$-equivalence]\label{def:K_equvialence}
		Let $X$ and $Y$ be integral normal $\bbQ$-Gorenstein schemes of finite type and separated over $S$.
		We say $X$ and $Y$ are \emph{$K$-equivalent} if there exists a normal scheme $Z$ over~$S$ with proper birational $S$-morphisms $p\colon Z \to X$ and $q\colon Z \to Y$ such that $p^*K_X$ and $q^*K_Y$ are $\bbQ$-linearly equivalent.

If in addition $X$ and $Y$ are Gorenstein, we say $X$ and $Y$ are \emph{strongly $K$-equivalent} if  there exists a normal scheme $Z$ over~$S$ with proper birational $S$-morphisms $p\colon Z \to X$ and $q\colon Z \to Y$ such that $p^*K_X$ and $q^*K_Y$ are linearly equivalent.
	\end{definition}

	Obviously, if $X$ and $Y$ are Gorenstein and strongly $K$-equivalent, then they are $K$-equivalent. The converse holds provided $X$ and $Y$ are not too singular\,; see \cref{prop:K_equvialence-eqivalence_implies_small_birational_map}\labelcref{item:canonical} below.
	The following \cref{prop:K_equvialence-eqivalence_implies_small_birational_map}\labelcref{item:temrinal} says, in particular, that $K$-equivalent normal terminal varieties are isomorphic in codimension~$1$. 
This is certainly well-known, at least in characteristic zero, see, e.g., \cite[Lem.~4.2]{kawamata_d_equivalence_and_k_equivalence},
	but as we were not able to find a suitable reference for our more general setting, we provide a proof.
	An integral normal excellent
	scheme $X$ of finite type and separated over $S$ is said to be \emph{terminal} (resp.\ \emph{canonical})
	 if $X$ is $\bbQ$-Gorenstein and for any proper birational $S$-morphism $f\colon X' \to X$ with $X'$ normal, the 
discrepancies of the exceptional divisors  are all positive (resp.\ nonnegative).

	\begin{proposition}
		\label{prop:K_equvialence-eqivalence_implies_small_birational_map}
		Let $X$ and $Y$ be integral normal schemes of finite type and separated over $S$.
	 Assume that $X$ and $Y$ are $K$-equivalent.
		\begin{enumerate}
			\item \label{item:canonical} 
			If $X$ and $Y$ are canonical and Gorenstein, then $X$ and $Y$ are strongly $K$-equivalent.
			\item \label{item:one_canonical_proper}
			If $X$ and $Y$ are Gorenstein proper varieties over $S = \Spec k $ a field, and either $X$ or $Y$ is canonical, then $X$ and $Y$ are strongly $K$-equivalent.
			\item \label{item:temrinal} 
			If $X$ and $Y$ are terminal, 	then the induced birational map $X\dashrightarrow  Y$ is small in the sense of \cref{def:small_bir_map}.
		\end{enumerate}
	\end{proposition}
	\begin{proof}
		We first prove \labelcref{item:canonical,item:temrinal}.
		Let $Z$ be a normal scheme over $S$ with proper birational morphisms $p\colon Z \to X$ and $q\colon Z \to Y$ such that $p^*K_X \sim_\bbQ q^*K_Y$.
		Let $E_i \subseteq \mathrm{Exc}(p)$ and $F_j \subseteq \mathrm{Exc}(q)$ be the irreducible components, endowed with their reduced structure, of codimension~$1$ of the exceptional loci of $p$ and $q$, respectively.
		Let $K_{Z/X}= \sum_i a_i E_i$ and $K_{Z/Y} =\sum_j b_j F_j$  be the relative canonical $\bbQ$-divisors, where $a_i, b_j \in \bbQ_{\geq0}$ (resp.\ $a_i, b_j \in \bbQ_{>0}$) if 
		 $X$ and $Y$ are canonical (resp.\ terminal).
			Let 
		$$  D\coloneqq K_{Z/Y} - K_{Z/X} = \sum b_j F_j - \sum a_i E_i.$$
		The divisor $D$ is $\bbQ$-Cartier and we have $D \sim_\bbQ p^*K_X - q^*K_Y \sim_\bbQ 0$.
		Note that $-D$ is $p$-nef and $p_*D = \sum_j b_j p_*F_j$ is effective.
		By the Negativity Lemma   \cite[Lem.~2.16]{bhatt_et_al_globally_+_regular_varieties_and_mmp_for_threefolds_in_mixed_char}  (which can be applied since $S$ is assumed to be excellent and since we can assume that $p$ and $q$ are projective after applying Chow's Lemma  \cite[\href{https://stacks.math.columbia.edu/tag/02O2}{Tag 02O2}]{stacks-project} and normalizing \cite[\href{https://stacks.math.columbia.edu/tag/035E}{Tag~035E}]{stacks-project}),
		$D$ is effective. 
		(For the classical version of the Negativity Lemma in characteristic zero, see \cite[Lem.~3.39]{kollar_mori_birational_geometry_of_algebraic_varieties}.)		
		Arguing similarly for $-D$ shows that $-D$ is effective, and therefore that $D =0$.
		This establishes~\labelcref{item:canonical}, as under the additional assumption that $X$ and $Y$ are Gorenstein we have that $a_i, b_i \in \bbZ_{\geq 0}$ and $p^*K_X - q^*K_Y \sim D = 0$.
		Regarding~\labelcref{item:temrinal}, if $X$ and $Y$ are terminal, we have $\Supp(K_{Z/X})=\mathrm{Exc}(p)$ and $\Supp(K_{Z/Y}) = \mathrm{Exc}(q)$ up to some locally closed subsets of codimension~$\geq 2$.
		Therefore, $D=0$ implies that
		$\mathrm{Exc}(p)=\mathrm{Exc}(q)$ up to some locally closed subsets of codimension $\geq 2$.
		This proves the statement, since $p(\mathrm{Exc}(p))$ and $q(\mathrm{Exc}(q))$ both have codimension $\geq 2$ (a proper birational morphism to a normal Noetherian scheme has geometrically connected fibers by \cite[\href{https://stacks.math.columbia.edu/tag/03H0}{Tag 03H0}]{stacks-project}).
		
		We now prove \labelcref{item:one_canonical_proper}. Assume that $Y$ is canonical and keep the notation as above.
		Then $b_j \in \bbZ_{\geq 0}$ for all $j$ so that $p_* D$ is effective and we can conclude as above that $D$ is effective.
		Thus, $D$ is an effective divisor with $nD \sim 0$ for some $n>0$.
		But we can identify the linear system $ \lvert nD \rvert = \bbP (H^0(Z, \mcO_Z) )= \{0\}$ as sets,  showing that $nD = 0$ and consequently that $D=0$.
	\end{proof}

	Together with \cref{lem:splinter_invariant_codim_2_surgery}, we obtain that both the splinter property and global $F$-regularity for normal terminal varieties over a field of positive characteristic are invariant under $K$-equivalence\,:
	
	\begin{corollary} \label{cor:splinter-gFr-K-eq}
		Let $X$ and $Y$ be $K$-equivalent integral normal terminal schemes of finite type and separated over~$S$.
	The following statements hold.
		\begin{enumerate}
			\item $X$ is a splinter if and only if $Y$ is a splinter. 
			\item $X$ is globally $F$-regular if and only if $Y$ is globally $F$-regular, provided $S=\operatorname{Spec} k$ with $k$ an $F$-finite field.
		\end{enumerate}
	\end{corollary}	
	\begin{proof}
		This is the combination of \cref{lem:small_birational_map_globF}, \cref{prop:small_birational_map}, and \cref{prop:K_equvialence-eqivalence_implies_small_birational_map}\labelcref{item:temrinal}.
	\end{proof}
	
	\begin{remark}\label{rmk:crepant-small}
		\cref{prop:K_equvialence-eqivalence_implies_small_birational_map}\labelcref{item:temrinal} 
		 does not hold without restrictions on the singularities of $X$ and~$Y$. 
		Consider indeed any crepant morphism $p\colon Y \to X$ to a Gorenstein variety $X$ over a field $k$, with exceptional locus containing a divisor.
		Then 
		$Y= Z \to X$ provides a strong $K$-equivalence between the Gorenstein proper varieties $X$ and $Y$ that does not induce a small birational map.
		Nonetheless, we know from \cref{rmk:crepant} that if $p \colon Y \to X$ is a crepant morphism of normal varieties over a field $k$ of positive characteristic, then $X$ is a splinter if and only if $Y$ is a splinter.  
		In the next paragraph, we will introduce the notion of $\mcO$-equivalence and will use it to circumvent going through small birational maps to improve upon
		\cref{cor:splinter-gFr-K-eq}
		 and show that the derived splinter property for proper varieties is invariant under strong $\mcO$-equivalence\,; see \cref{cor:splinters-O-eq}.
	\end{remark}
	
	The following proposition in characteristic zero is \cite[Cor.~5.24]{kollar_mori_birational_geometry_of_algebraic_varieties}. We provide an analogous proof in the  general situation of normal excellent Noetherian schemes admitting a dualizing complex.
	
	\begin{proposition}\label{prop:canonical_pseudorational_Gorenstein}
		Let $X$ be a normal scheme of finite type and separated over $S$. 
		If $X$ is Gorenstein, then $X$ is pseudo-rational if and only if $X$ is canonical.
	\end{proposition}
	\begin{proof}
		Recall from \cite[Cor.\ on p.107]{lipman_teissier_pseudorational_local_rings} that a normal excellent Cohen--Macaulay Noetherian scheme $X$  admitting a dualizing complex is pseudo-rational if and only if for every proper birational morphism $p \colon Y \to X$ with $Y$ normal, the canonical map $p_* \omega_Y \to \omega_X$ is an isomorphism.		
		Let $p \colon Y \to X$ be a proper birational map with $Y$ normal and denote by $\mathrm{Exc}(p)\subseteq Y$ the exceptional locus so that $p$ is an isomorphism over $U \coloneqq X \setminus p(\mathrm{Exc}(p))$.
		Since $X$ is normal, $\codim_X p(\mathrm{Exc}(p))\geq 2$, and the map $p_* \omega_Y \to \omega_X$ can be identified with
		$$H^0(V, p_* \omega_Y) \to H^0(V \cap U, p_* \omega_Y)  = H^0(V \cap U, \omega_X) = H^0(U, \omega_X)$$
		for any open subset $V \subseteq X$,
		where the first map is the restriction map and the last equality holds since $\omega_X$ is reflexive\,;
		cf.\ \cite[p.~109]{lipman_teissier_pseudorational_local_rings}. Assume now that $X$ is Gorenstein, and
		let $K_Y \sim p^*K_X + E$ with $E$  an effective divisor supported on $\mathrm{Exc}(p)$.
		Via the projection formula, the canonical map $p_*\omega_Y \to \omega_X$ identifies with the map $\omega_X \otimes p_*\mcO_Y(E) \to \omega_X \otimes \mcO_X = \omega_X$ obtained by tensoring the identity on $\omega_X$ with the map $p_*\mcO_Y(E) \to \mcO_X$  given by restricting a local section of $p_* \mcO_Y(E)$ along $U \hookrightarrow X$.
		Since $E$ is supported on $\mathrm{Exc}(p)$, this map $p_*\mcO_Y(E) \to \mcO_X$ is an isomorphism if and only if $E$ is effective.
		Since $\omega_X$ is invertible, this proves the statement of the proposition.
	\end{proof}

	\subsection{$\mcO$-equivalence}
	
		In this paragraph, we use the formalism of the  exceptional inverse image functor, as described in  \cite[\href{https://stacks.math.columbia.edu/tag/0A9Y}{Tag 0A9Y}]{stacks-project} (under the name of \emph{upper shriek functor}), for separated schemes of finite type over a fixed Noetherian base\,; see  \cref{SS:uppershriek}.

	\begin{definition}[$\mcO$-equivalence and strong $\mcO$-equivalence]\label{def:O_equiv}
		Let $S$ be a Noetherian scheme and let $X$ and $Y$ be schemes of finite type and separated over $S$.
		
		We say $X$ and $Y$ are \emph{strongly $\mcO$-equivalent}, if there exists a scheme $Z$
	 over $S$ with proper birational $S$-morphisms $p\colon Z \to X$ and $q \colon Z \to Y$  such that
		$p^! \mcO_X \cong q^! \mcO_Y$ holds in $\sfD_{\mathrm{Coh}}(\mcO_Z)$.
		
			We say $X$ and $Y$ are \emph{$\mcO$-equivalent}, if there exists a scheme $Z$
		over $S$ with proper birational $S$-morphisms $p\colon Z \to X$ and $q \colon Z \to Y$ and a proper surjective morphism $\mu\colon \tilde{Z} \to Z$ such that
		$\mu^! p^! \mcO_X \cong \mu^!q^! \mcO_Y$ holds in $\sfD_{\mathrm{Coh}}(\mcO_{\tilde{Z}})$.
	\end{definition}
	
Obviously, strong $\mcO$-equivalence implies $\mcO$-equivalence.
	As will become clear, 
	$\mcO$-equivalence offers more generality and more flexibility than $K$-equivalence\,: in particular, it does not involve any ($\bbQ$-)Gorenstein assumption.
	For integral normal Gorenstein schemes, (strong) $\mcO$-equivalence and (strong) $K$-equivalence relate as follows\,:
	
	\begin{proposition} \label{prop:K-O-eq}
		Let $X$ and $Y$ be 
		integral normal 
Gorenstein  schemes of finite type and separated over an excellent Noetherian scheme $S$ admitting a dualizing complex $\omega_S^\bullet$.
Consider the following statements\,:
\begin{enumerate}
\item  \label{item:strong_K} $X$ and $Y$ are strongly $K$-equivalent.
\item  \label{item:strong_O}$X$ and $Y$ are strongly $\mcO$-equivalent.
\item  \label{item:K} $X$ and $Y$ are  $K$-equivalent.
\item  \label{item:O} $X$ and $Y$ are  $\mcO$-equivalent.
\end{enumerate} 
Then $\labelcref{item:strong_K} \Leftrightarrow \labelcref{item:strong_O} \Rightarrow \labelcref{item:K} \Leftrightarrow \labelcref{item:O}$. 
If in addition $X$ and $Y$ are canonical, then  $\labelcref{item:strong_K} \Leftrightarrow \labelcref{item:strong_O} \Leftrightarrow \labelcref{item:K} \Leftrightarrow \labelcref{item:O}$. 
	\end{proposition}
	\begin{proof}
Let $p\colon Z \to X$ and $q\colon Z \to Y$ be birational morphisms of schemes of finite type and separated over a Noetherian scheme $S$ admitting a dualizing complex. 
	Recall from \cref{lem:k_and_o_equvalence} that if $X$ and $Y$ are Gorenstein, 
then $p^! \mcO_X \cong q^! \mcO_Y$ if and only if $p^* \omega_X \cong q^*\omega_Y$. 
If $X$ is reduced, then $Z$ is generically reduced and hence \cite[\href{https://stacks.math.columbia.edu/tag/0BXC}{Tag 0BXC}]{stacks-project} the normalization $Z^\nu \to Z$ is birational. We thus see that 
strong $\mcO$-equivalence coincides with strong $K$-equivalence, i.e., that $\labelcref{item:strong_K} \Leftrightarrow \labelcref{item:strong_O}$.

The implication $\labelcref{item:strong_O} \Rightarrow \labelcref{item:O}$ is clear and holds without assuming $X$ and $Y$ to be normal Gorenstein.

For  $\labelcref{item:K}\Rightarrow \labelcref{item:O}$, consider the cyclic covering associated to the torsion invertible sheaf $\mathcal L \coloneqq p^*\omega_X \otimes q^*\omega_Y^{-1}$, i.e., 
$$\mu\colon \tilde{Z} \coloneqq \Spec_Z \left (\bigoplus_{i=0}^{r-1} \mcL^{i} \right ) \to Z,$$
where $r$ is the torsion index of $\mathcal L$.
As in the proof of \cref{prop:splinter_dualizing_sheaf_not_torsion}, we have that $\mu^* \mcL = \mcO_{\tilde{Z}}$, i.e., $\mu^* p^*\omega_X \cong \mu^*q^*\omega_Y$.
One concludes with \cref{lem:k_and_o_equvalence} that $\mu^!p^!\mcO_X \cong \mu^!q^!\mcO_Y$.

For $\labelcref{item:O}\Rightarrow \labelcref{item:K}$, recall, e.g., from
\cite[\href{https://stacks.math.columbia.edu/tag/02U9}{Tag 02U9}]{stacks-project}, that for a proper morphism $\mu \colon \tilde{Z} \to Z$ and an invertible sheaf $\mcL$ on $Z$ we have
\begin{equation}\label{eq:proj_formula_Chern}
\mu_*(\cc_1(\mu^*\mcL) \cap [\tilde Z]) =  \cc_1(\mcL)\cap \mu_*[\tilde Z] \quad \text{in}\;\CH_{\dim Z-1}(Z).
\end{equation}
Assume now that $X$ and $Y$ are $\mcO$-equivalent and let $\mu\colon \tilde Z \to Z$, $p\colon Z\to X$ and $q\colon Z \to Y$ be as in \cref{def:O_equiv}. 
Up to normalizing $\tilde{Z}$ and $Z$, we may and do assume that $Z$ is normal. 
By Chow's Lemma \cite[\href{https://stacks.math.columbia.edu/tag/02O2}{Tag 02O2}]{stacks-project} and by taking hyperplane sections, we may assume that $\mu$ is generically finite.
 The condition $\mu^! p^! \mcO_X \cong \mu^!q^! \mcO_Y$ together with \cref{eq:proj_formula_Chern} then imply that
the divisor class associated to 
 $p^*\omega_X \otimes q^*\omega_Y^{-1}$ is torsion, and hence by normality of $Z$ that $p^*\omega_X \otimes q^*\omega_Y^{-1}$ is a torsion invertible sheaf on $Z$, i.e.\ $p^*K_X \sim_\bbQ q^*K_Y$.

Finally if $X$ and $Y$ are canonical, then $\labelcref{item:strong_K} \Leftrightarrow \labelcref{item:K}$ by \cref{prop:K_equvialence-eqivalence_implies_small_birational_map}.
	\end{proof}

	Moreover, under some regularity assumption, the cohomology of the structure sheaf is invariant under strong $\mcO$-equivalence\,:
	
	\begin{proposition}\label{prop:O-eq-Hodge-numbers}
		Let $X$ and $Y$ be excellent regular schemes of finite type and separated over a Noetherian scheme $S$.
		If there exists an excellent regular scheme $Z$
		of finite type and separated over~$S$ with projective birational morphisms $p\colon Z \to X$ and $q \colon Z \to Y$ over $S$ such that
		$p^! \mcO_X \cong q^! \mcO_Y$, then $H^i(X,\mcO_X) \cong H^i(Y,\mcO_Y)$ as $H^0(S,\mcO_S)$-modules for all $i\geq 0$.
	\end{proposition}
	\begin{proof} 
		By  \cite[Thm.~1.1]{chatzistamatiou_ruelling_vanishing_of_the_higher_direct_images_of_the_structure_sheaf}, both canonical maps $\mcO_X \to \mathrm{R} p_* \mcO_Z$ and $\mcO_Y \to \mathrm{R} q_* \mcO_Z$ are isomorphisms. 
		Using the fact that the exceptional inverse image functor commutes with shifts, we obtain the chain of isomorphisms
		\begin{align*}
		H^i(Y,\mcO_Y) & =
		\Hom_{\sfD_{\mathrm{Coh}}(\mcO_Y)}(\mcO_Y, \mcO_Y[i])  = 
		\Hom_{\sfD_{\mathrm{Coh}}(\mcO_Y)}(\mathrm{R} q_* \mcO_Z, \mcO_Y[i]) \\
		& = \Hom_{\sfD_{\mathrm{Coh}}(\mcO_Z)}( \mcO_Z, q^! \mcO_Y[i]) 
		\cong \Hom_{\sfD_{\mathrm{Coh}}(\mcO_Z)}( \mcO_Z, p^! \mcO_X[i]) \\
		& = \Hom_{\sfD_{\mathrm{Coh}}(\mcO_X)}(\mathrm{R} p_* \mcO_Z, \mcO_X[i])
		= \Hom_{\sfD_{\mathrm{Coh}}(\mcO_X)}(\mcO_X, \mcO_X[i])
		=H^i(X,\mcO_X).\qedhere
		\end{align*}
	\end{proof}

\begin{remark} \label{rem:O-eq-Hodge-numbers}
	Let $X$ and $Y$ be regular schemes of finite type and separated over a perfect field~$k$. 
	If $X$ and $Y$ are strongly $\mcO$-equivalent and if resolution of singularities holds for reduced schemes of dimension $\dim X$ over $k$, then by applying Chow's Lemma to $Z$ as in \cref{def:O_equiv} and then resolving singularities, one may choose $Z$ to be regular. Hence, provided resolution of singularities holds for reduced schemes of dimension $\dim X$ over $k$ (which is the case if $\dim X \leq 3$ by \cite{CP}), \cref{prop:O-eq-Hodge-numbers} shows that if $X$ and $Y$ are strongly $\mcO$-equivalent, then $H^i(X,\mcO_X) \cong H^i(Y,\mcO_Y)$ for all $i \geq 0$.
\end{remark}

	\subsection{$D$-equivalence} The following is classical.
	
	\begin{definition}[$D$-equivalence]
		Two 
		proper varieties $X$ and $Y$ over a field $k$ are said to be \emph{$D$-equivalent} (or \emph{derived equivalent}) if there is a $k$-linear equivalence of categories $\sfD^b(X) \cong \sfD^b(Y)$ between their bounded derived categories of coherent sheaves.
	\end{definition}
	
Kawamata~\cite[Thm.~2.3(2)]{kawamata_d_equivalence_and_k_equivalence} showed that if $X$ and $Y$ are $D$-equivalent smooth projective varieties over an algebraically closed field of characteristic zero and if $K_X$ or $-K_X$ is big, then $X$ and $Y$ are $K$-equivalent. We have the same result in the broader context of normal Gorenstein projective varieties over any field\,:
	
	\begin{proposition}\label{prop:d_equiv_implies_o_equiv}
	Let $X$ and $Y$ be normal Gorenstein
	projective
	varieties over a
	field $k$.
	Assume that $K_X$ or $-K_X$ is big.
	If $X$ and $Y$ are $D$-equivalent, then $X$ and $Y$ are $K$-equivalent (equivalently, $\mcO$-equivalent).
	\end{proposition}
	\begin{proof}
		By \cite[Cor.~9.17]{lunts_orlov_uniqueness_of_enhancements_for_triangulated_categories},
		a $k$-linear equivalence $\sfD^b(X) \cong \sfD^b(Y)$ is induced by a Fourier--Mukai transform with kernel $K \in \sfD^b(X\times_k Y)$.
		Denote by $\pi_X \colon X \times_k Y \to X$ and $\pi_Y \colon X\times_k Y \to Y$ the projections to $X$ and $Y$.
		By \cite[Prop.~4.2]{hernandez_ruiperez_lopez_sancho_de_salas_relative_fm_transforms_singular_var}, we have a natural isomorphism
		\begin{equation}\label{eq:left_and_right_kernel}
			\mathrm{R}\sheafhom_{\mcO_{X\times_k Y}} (K, \pi_X^!\mcO_{X})  \cong \mathrm{R}\sheafhom_{\mcO_{X\times_k Y}} (K, \pi_Y^!\mcO_{Y}).
		\end{equation}
		(Note that the
		proof of \cite[Prop.~4.2]{hernandez_ruiperez_lopez_sancho_de_salas_relative_fm_transforms_singular_var} does not require that the base field be algebraically
		closed.)
		On the other hand, by base change \cite[\href{https://stacks.math.columbia.edu/tag/0E9U}{Tag 0E9U}]{stacks-project} and since $X$ and $Y$ are Gorenstein, $\pi_X^! \mcO_X \cong \mathrm{L}\pi_Y^* \omega_Y^\bullet = \pi_Y^*\omega_Y[\dim Y]$ and $\pi_Y^! \mcO_Y \cong \mathrm{L}\pi_X^* \omega_X^\bullet = \pi_X^* \omega_X[\dim X]$.
		Thus, \cref{eq:left_and_right_kernel} is equivalent to 
		\begin{equation*}
			K^\vee \cong K^\vee \otimes \pi_X^* \omega_X \otimes \pi_Y^* \omega_Y^{-1}[\dim X - \dim Y] \quad \text{where}\; K^\vee\coloneqq\mathrm{R}\sheafhom_{\mcO_{X\times_k Y}} (K, \mcO_{X\times_k Y}).
		\end{equation*}
		As shown in \cite[Proof of Prop.~2.10]{hernandez_ruiperez_lopez_sancho_de_salas_relative_fm_transforms_singular_var}, $K^\vee$ lies in $\sfD^b(X\times_k Y)$, and hence $\dim X = \dim Y$.
		
		Let $\nu\colon  Z^\nu \to Z$ be the normalization of an irreducible component $Z$ of $\mathrm{Supp} (K^\vee)$ 
		and set $p = \pi_X \circ \nu$ and $q = \pi_Y \circ \nu$.
		Then, there exists $i \in \bbZ$ such that $\nu^*\mcH^i(K^\vee)\vert_Z$ generically has positive rank $r>0$.
		Arguing as in \cite[Lem.~6.9]{huybrechts_fm_transforms}, we obtain 
		\begin{equation*}\label{eq:K}
			\mcO_{Z^\nu}(r p^*K_X) \cong \mcO_{Z^\nu}(r q^*K_Y).
		\end{equation*}
		Since $K_X$ or $-K_X$ is big, arguing as in the proof of \cite[Thm.~2.3(2)]{kawamata_d_equivalence_and_k_equivalence} (see also \cite[Prop.~6.19]{huybrechts_fm_transforms}) shows that
		 there exists a component $Z$ that dominates $X$ and $Y$
		and is such that $p$ and $q$ are birational morphisms.
		This proves that $X$ and $Y$ are $K$-equivalent. 
		By \cref{prop:K-O-eq}, $X$ and $Y$ are also $\mcO$-equivalent.
	\end{proof}
	
	\begin{remark}
		In the same way that $D$-equivalent smooth proper complex varieties are not necessarily $K$-equivalent \cite{uehara_an_example_of_fourier_mukai_partners_of_minimal_elliptic_surfaces}, $D$-equivalent smooth proper varieties in positive characteristic are not necessarily  $\mcO$-equivalent.
		Indeed, in \cite{ab}, Addington and Bragg have produced $D$-equivalent smooth projective threefolds over $\overline{\bbF}_3$ with different Hodge numbers~$h^{0,i}$ for $i=1$ and~$2$.
		 Such varieties are not $\mcO$-equivalent by \cref{prop:O-eq-Hodge-numbers} and \cref{rem:O-eq-Hodge-numbers}.
	\end{remark}

	\subsection{Invariance of the splinter property under strong $\mcO$-equivalence}
	Since a crepant morphism $p\colon Y \to X$ provides a strong $\mcO$-equivalence between $Y$ and $X$, the following theorem extends \cref{rmk:crepant}\labelcref{item:crepant-Dsplinter} in the case of proper varieties over a field.
	
	\begin{theorem}[Theorem~\ref{thm:O-splinter},
		Derived splinters are stable under strong $\mcO$-equivalence]
		\label{cor:splinters-O-eq}
		Let $X$ and~$Y$ be proper varieties over a field.
		If  $X$ and $Y$ are strongly $\mcO$-equivalent, then $X$ is a derived splinter if and only if $Y$ is a derived splinter.
	\end{theorem}
	\begin{proof}
		Suppose $X$ is a derived splinter.
		Let $f\colon B \to Y$ be a proper surjective morphism. We have to show that $f^\sharp \colon \mcO_Y \to \mathrm{R}f_*\mcO_B$ splits in $\sfD_{\mathrm{Coh}}(\mcO_Y)$.  
		Consider the diagram
		\[
		\begin{tikzcd}
				& Z \arrow[swap]{dl}{p} \arrow{dr}{q} && B' \arrow{ll}[swap]{f'} \arrow{d}{q'}\\
				X & & Y & B \arrow[swap]{l}{f}
		\end{tikzcd}
		\]
		where $p$ and $q$ are proper birational over $S$ with $p^!\mcO_X \cong q^!\mcO_Y$, and where $B'\coloneqq B\times_Y Z$.
			We then have the following commutative diagram\,:

		\[
		\begin{tikzcd}
			\Hom_{\sfD_{\Coh}(\mcO_X)} (\mathrm{R} p_* \mathrm{R} f'_* \mcO_{B'}, \mcO_X) \rar{- \circ \mathrm{R} p_* {f'}^\sharp}\dar{=} & \Hom_{\sfD_{\Coh}(\mcO_X)} (\mathrm{R} p_*  \mcO_{Z}, \mcO_X) \rar{-\circ p^\sharp} \dar{=} &\Hom_{\sfD_{\Coh}(\mcO_X)} (  \mcO_{X}, \mcO_X) \\
			 \Hom_{\sfD_{\Coh}(\mcO_Z)} ( \mathrm{R} f'_* \mcO_{B'}, p^! \mcO_X) \rar{- \circ {f'}^\sharp} \dar{\cong} & \Hom_{\sfD_{\Coh}(\mcO_Z)} ( \mcO_{Z}, p^! \mcO_X)  \dar{\cong}&\\		
			  \Hom_{\sfD_{\Coh}(\mcO_Z)} ( \mathrm{R} f'_* \mcO_{B'}, q^! \mcO_Y) \rar{- \circ  {f'}^\sharp} \dar{=} & \Hom_{\sfD_{\Coh}(\mcO_Z)} ( \mcO_{Z}, q^! \mcO_Y)  \dar{=}&\\
			  \Hom_{\sfD_{\Coh}(\mcO_Y)} (\mathrm{R} q_* \mathrm{R} f'_* \mcO_{B'}, \mcO_Y) \rar{- \circ \mathrm{R} q_* {f'}^\sharp} & \Hom_{\sfD_{\Coh}(\mcO_Y)} (\mathrm{R} q_*  \mcO_{Z}, \mcO_Y) \rar{-\circ q^\sharp}  &\Hom_{\sfD_{\Coh}(\mcO_Y)} (  \mcO_{Y}, \mcO_Y).
		\end{tikzcd}
		\]
		Let $s \colon \mathrm{R} p_*  \mathrm{R} f'_* \mcO_{B'} \to \mcO_X$ be a section of $(p\circ f')^\sharp \colon  \mcO_X \to \mathrm{R} p_*  \mathrm{R} f'_* \mcO_{B'}$
		and let $t\colon \mathrm{R} q_* \mathrm{R}f'_* \mcO_{B'}\to \mcO_Y$ be the image of $s$ under composition of the left vertical arrows above.
		Let $s'  \coloneqq s \circ \mathrm{R}p_*{f'}^\sharp $ and $t'  \coloneqq t \circ \mathrm{R}q_*{f'}^\sharp$.
		Since $\Hom_{\mcO_Y}(\mcO_Y, \mcO_Y)$ is a field, it is enough to show that $t' \circ q^\sharp$ is nonzero.
		Choose dense open subsets $U \subseteq X$ and $V \subseteq Y$ such that $p^{-1}(U) = q^{-1}(V)$ and such that $p$ and $q$ restrict to isomorphisms over $U$ and $V$, respectively.
		Since $Y$ is integral, it is enough to show that $t' \circ q^\sharp\vert_V$ is nonzero.
		As $-\circ q^\sharp \colon \Hom_{\sfD_{\Coh}(\mcO_V)} (\mathrm{R} q_*  \mcO_{Z}\vert_V, \mcO_Y\vert_V) \to \Hom_{\sfD_{\Coh}(\mcO_V)} (  \mcO_{Y}\vert_V, \mcO_Y\vert_V)$ is an isomorphism, the latter is equivalent to $t'\vert_V$ being nonzero.
		Since $s'$ is a section of $\mcO_X \to \mathrm{R} p_* \mcO_Z$, $s'\vert_U$ is nonzero.
		This shows that $t'\vert_V$ is nonzero, since $t'\vert_V$ is the image of $s'\vert_U$ under the middle vertical isomorphism when restricted to $U \leftarrow p^{-1}(U) = q^{-1} (V) \to V$.
		Hence, $\mcO_Y\to \mathrm{R}q_* \mathrm{R}f'_* \mcO_{B'} = \mathrm{R}f_* \mathrm{R}q'_* \mcO_{B'}$ splits and thus,
		$\mcO_Y\to \mathrm{R}f_* \mcO_{B}$ splits.
	\end{proof}

	Combined with \cref{prop:d_equiv_implies_o_equiv}, we obtain a partial answer to the question, whether the splinter property for smooth projective schemes over a field is stable under derived equivalence\,: 
	
	\begin{corollary}[Theorem~\ref{thm:D-splinter}]
		\label{cor:d_equvialent_splinters_pseudo-rational}
		Let $X$ and $Y$ be
		normal  Gorenstein
		projective varieties over a field of positive characteristic.
		Assume that 
		$-K_X$ is big.
		If
		$X$ and $Y$ are $D$-equivalent, then $X$ is a splinter if and only if $Y$ is a splinter.
	\end{corollary}
	\begin{proof}
		Under the assumption that $-K_X$ is big, we have from \cref{prop:d_equiv_implies_o_equiv} that $X$ and $Y$ are $\mcO$-equivalent. Assume now that $X$ and $Y$ are $\mcO$-equivalent and that $X$ is a splinter.
	Recall from \cref{prop:bhatt_normal_CM} that a splinter is normal pseudo-rational, and from
	 \cref{prop:canonical_pseudorational_Gorenstein} that a normal Gorenstein pseudo-rational variety has canonical singularities. 
		Therefore $X$ and $Y$ are strongly $\mcO$-equivalent by \cref{prop:K_equvialence-eqivalence_implies_small_birational_map}\labelcref{item:one_canonical_proper} and \cref{prop:K-O-eq}.
			We may now conclude with  \cref{cor:splinters-O-eq} that $Y$ is a splinter by using the fact due to Bhatt~\cite[Thm.~1.4]{bhatt_derived_splinters_in_positive_characteristic} that a Noetherian scheme of positive characteristic is a splinter if and only if it is a derived splinter.
	\end{proof}

	\begin{remark}
		According to \cref{conj:splinter_big_anticanonical_class}, the assumption that $-K_X$ is big in \cref{cor:d_equvialent_splinters_pseudo-rational} is conjecturally superfluous. 
	\end{remark}

	\subsection{Invariance of global $F$-regularity under strong $\mcO$-equivalence and $D$-equivalence}

	\begin{theorem}[Theorem~\ref{thm:O-gFr}, Global $F$-regularity is stable under strong $\mcO$-equivalence]\label{prop:globF_stable_O_equiv}
		Let $X$ and~$Y$ be projective varieties over 
		a field~$k$ of positive characteristic.
		If $X$ and $Y$ are strongly $\mcO$-equivalent, then $X$ is normal globally $F$-regular if and only if $Y$ is normal globally $F$-regular.
	\end{theorem}
	
	\begin{proof}
		Let $p\colon Z \to X$ and $q \colon Z \to Y$ be as in \cref{def:O_equiv}.
		After possibly normalizing $Z$, we may and do assume that $Z$ is normal.
		Assume that $X$ is normal globally $F$-regular.
		By \cref{prop:gFr-splinter}, $X$ is a splinter. By \cite[Thm.~1.4]{bhatt_derived_splinters_in_positive_characteristic}, $X$ is a derived splinter\,; 
		in particular the map $\mcO_X \to \mathrm{R}p_* \mcO_Z$ splits. Moreover, by \cref{cor:splinters-O-eq}, $Y$ is also a derived splinter and hence is normal.
		Fix nonempty regular affine open subsets $U \subseteq X$ and $V \subseteq Y$ such that $p^{-1}(U) = q^{-1}(V)$ 
		and such that $p\vert_U$ and $q\vert_V$ are isomorphisms.
		By \cite[\href{https://stacks.math.columbia.edu/tag/0BCU}{Tag 0BCU}]{stacks-project},
		$D \coloneqq X\setminus U$ and $E \coloneqq Y \setminus V$ are divisors. 
		Since $Y$ is projective, any Weil divisor on $Y$ is dominated by a Cartier divisor and we may thus further assume that $E$ is Cartier.
		By \cref{thm:schwede_smith_f_reg_one_divisor} it suffices to show that there exists $e>0$ such that $\mcO_Y \to F_*^e \mcO_Y(E)$ splits.
		Let $D'\geq D$ be a Cartier divisor on $X$. Since $q^*E = \sum_i a_i E_i$ for some $a_i \in \bbZ_{>0}$ with $E_i$ the irreducible components of $\Supp(q^*E) \subseteq \Supp(p^*D ')$, there exists $n \in \bbZ_{>0}$ such that $q^*E \leq np^*D'$.
		Let $U' \coloneqq X \setminus D'$ and set $V' \coloneqq q(p^{-1}(U'))$.
		Then $U'$ and $V'$ are regular affine open subsets such that $p\vert_{U'}$ and $q\vert_{V'}$ are isomorphisms.
		
		Since $X$ is globally $F$-regular, there exists an integer $e>0$ such that $\sigma_{nD'} \colon \mcO_X \to F_*^e \mcO_X(nD')$ splits.
		By the projection formula, the splitting of $\mcO_X \to \mathrm{R}p_* \mcO_Z$ gives a splitting of $\id_{\Coh(X)} \to \mathrm{R}p_* \mathrm{L}p^*$.
		Thus, we obtain a splitting $s$ of
		$$\alpha \colon \mcO_X \to F_*^e \mcO_X(nD') \to F_*^e \mathrm{R}p_* \mathrm{L}p^* \mcO_X(nD') =\mathrm{R}p_* F_*^e \mcO_Z(np^*D').$$
		Note that the map $\alpha$ can also be obtained as the composition
		$$ \mcO_X \to \mathrm{R}p_* \mcO_Z \to\mathrm{R}p_* F_*^e  \mcO_Z \to \mathrm{R}p_* F_*^e  \mcO_Z(np^*D').$$
		We obtain a commutative diagram (where $\sfD(-)$ stands for $\sfD_{\Coh}(-)$)\,:
		\[
		\begin{tikzcd}
			\Hom_{\sfD_{}(\mcO_X)} ( \mathrm{R}p_* F_*^e \mcO_Z(np^*D' ), \mcO_X) \arrow[rr,"- \circ \mathrm{R}p_* \sigma_{np^* D'}^\sharp"]
		    \dar{=} & &\Hom_{\sfD_{}(\mcO_X)} (\mathrm{R} p_*  \mcO_{Z}, \mcO_X) \rar{-\circ p^\sharp} \dar{=} &\Hom_{\sfD_{}(\mcO_X)} (  \mcO_{X}, \mcO_X) \\
			\Hom_{\sfD_{}(\mcO_Z)} (  F_*^e \mcO_Z(np^*D' ), p^! \mcO_X) \arrow[rr,"- \circ \sigma_{np^* D'}^\sharp"] \dar{\cong} && \Hom_{\sfD_{}(\mcO_Z)} ( \mcO_{Z}, p^! \mcO_X)  \dar{\cong}&\\		
			\Hom_{\sfD_{}(\mcO_Z)} ( F_*^e \mcO_Z(np^*D' ), q^! \mcO_Y) \arrow[rr,"- \circ  \sigma_{np^* D'}^\sharp"] \dar{\gamma} && \Hom_{\sfD_{}(\mcO_Z)} ( \mcO_{Z}, q^! \mcO_Y)  \dar{=}&\\
			\Hom_{\sfD_{}(\mcO_Z)} ( F_*^e \mcO_Z(q^*E ), q^! \mcO_Y) \arrow[rr,"- \circ  \sigma_{q^*E}^\sharp"] \dar{=} && \Hom_{\sfD_{}(\mcO_Z)} ( \mcO_{Z}, q^! \mcO_Y)  \dar{=}&\\
			\Hom_{\sfD_{}(\mcO_Y)} (\mathrm{R} q_* F_*^e\mcO_Z(q^*E ), \mcO_Y) \arrow[rr,"- \circ \mathrm{R}q_*\sigma_{q^*E}^\sharp"] && \Hom_{\sfD_{}(\mcO_Y)} (\mathrm{R} q_*  \mcO_{Z}, \mcO_Y) \rar{-\circ q^\sharp}  &\Hom_{\sfD_{}(\mcO_Y)} (  \mcO_{Y}, \mcO_Y),
		\end{tikzcd}		
		\]
		where $s$ maps to the identity under the composition of the top horizontal arrows.
		Here the vertical map $\gamma$ is the precomposition with $F_*^e \mcO_Z(q^*E ) \to F_*^e \mcO_Z(np^*D' )$.

		Let $t\colon \mathrm{R} q_* F_*^e\mcO_Z(q^*E )\to \mcO_Y$ be the image of $s$ under the composition of the left vertical arrows.
		As in the proof of \cref{cor:splinters-O-eq}, we conclude that $t$ maps to a nonzero element of $\Hom_{\sfD_{}(\mcO_Y)} (  \mcO_{Y}, \mcO_Y) $, thereby providing a splitting of
		$\mcO_Y \to F_*^e \mcO_Y(E) \to F_*^e \mathrm{R} q_* \mathrm{L}q^* \mcO_Y (E) = \mathrm{R} q_* F_*^e \mcO_{Z} (q^* E).$
	\end{proof}

	\begin{remark}[The $F$-split property is stable under strong $\mcO$-equivalence]
		\label{rnk:F-split-O-eq} 
		Assume that
			$X$ and~$Y$ are normal proper varieties over a field~$k$ of positive characteristic.	
		If $X$ and $Y$ are strongly $\mcO$-equivalent, then $X$ is $F$-split if and only if $Y$ is $F$-split. 
		This follows indeed by considering  in the proof of \cref{prop:globF_stable_O_equiv} dense opens $U'=U\subseteq X$ and $V'=V\subseteq Y$ such that  $p^{-1}(U) = q^{-1}(V)$ and such that $p\vert_U$ and $q\vert_V$ are isomorphisms, and by setting $D=0$ and $E=0$.
	\end{remark}

	\begin{corollary}[Theorem~\ref{thm:D-gFr}, Global $F$-regularity is stable under $D$-equivalence] \label{cor:d-equiv-gFr}
		Let $X$ and~$Y$ be normal Gorenstein
		projective varieties over 
		a field $k$ of positive characteristic.
		If 
		$X$ and $Y$ are $D$-equivalent, then $X$ is globally $F$-regular if and only if $Y$ is  globally $F$-regular.
	\end{corollary}
	\begin{proof}
		Since normal globally $F$-regular projective varieties have big anticanonical class, this follows as in \cref{cor:d_equvialent_splinters_pseudo-rational} from \cref{prop:K_equvialence-eqivalence_implies_small_birational_map}\labelcref{item:one_canonical_proper},  \cref{prop:d_equiv_implies_o_equiv,prop:K-O-eq,prop:canonical_pseudorational_Gorenstein}, combined with \cref{prop:globF_stable_O_equiv}.
	\end{proof}

	\begin{remark}[The $F$-split property and $D$-equivalence] 
		As asked by Zsolt Patakfalvi,
		we are unaware whether the $F$-split property is stable under $D$-equivalence. As a partial result, we mention that under the assumptions of   \cref{cor:d-equiv-gFr}, if 	$X$ and $Y$ are $D$-equivalent and if $-K_X$ is big, then $X$ is $F$-split if and only if $Y$ is $F$-split. 
		For this, one argues as in the proof of \cref{cor:d_equvialent_splinters_pseudo-rational} by using \cref{rnk:F-split-O-eq}. 
		
		Moreover, if $X$ and $Y$ are $D$-equivalent abelian varieties or  $D$-equivalent strict Calabi--Yau varieties over a perfect field, then $X$ is $F$-split if and only if $Y$ is $F$-split. 
        In the case of abelian varieties, $X$ and $Y$ are then isogenous by \cite[Thm.~2.19]{orlov_derived_categories_of_coh_sheaves_on_ab_var_and_equiv_between_them}
        and therefore have the same height.
		In the case of strict Calabi--Yau varieties,
	  $X$ and $Y$ have the same height 	by \cite[Thm.~5.1]{antieau_bragg_derived_invariants_from_topological_hochschild_homology}. 
			One concludes with the fact that such varieties have height $1$ if and only if they are $F$-split\,; this is classical in the case of abelian varieties and is  \cite[Thm.~2.1]{van_der_geer_katsura_on_the_height_of_calabi_yau_varieties_in_positive_characteristic} in the case of strict Calabi--Yau varieties.
	\end{remark}

\begin{remark}
	It would be interesting to know whether
     the derived splinter property is stable under (strong) $\mcO$-equivalence for schemes of finite type and separated over a Noetherian base, and  whether Theorems~\labelcref{thm:D-splinter,thm:D-gFr} hold without assuming $X$ and $Y$ to be Gorenstein.
\end{remark}

	\printbibliography[
	title=References, ]

\end{document}